\documentclass[reqno,11pt]{amsproc}
\DeclareFontEncoding{OT2}{}{}
\DeclareFontSubstitution{OT2}{cmr}{m}{n}
\DeclareFontFamily{OT2}{cmr}{\hyphenchar\font45 }
\DeclareFontShape{OT2}{cmr}{m}{n}{
<5><6><7><8><9>gen*wncyr
<10><10.95><12><14.4><17.28><20.74><24.88>wncyr10}{}
\DeclareFontShape{OT2}{cmr}{b}{n}{
<5><6><7><8><9>gen*wncyb
<10><10.95><12><14.4><17.28><20.74><24.88>wncyb10}{}
\DeclareMathAlphabet{\mathcyr}{OT2}{cmr}{m}{n}
\DeclareMathAlphabet{\mathcyb}{OT2}{cmr}{b}{n}
\SetMathAlphabet{\mathcyr}{bold}{OT2}{cmr}{b}{n}
\usepackage{fullpage}
\usepackage{fancybox,ascmac,comment,xcolor}
\usepackage{amssymb,mathrsfs}
\usepackage{amsfonts}
\usepackage{amsmath}
\usepackage{amsthm}
\usepackage{mathtools}
\usepackage{enumerate}

\usepackage{mleftright}
\mleftright
\usepackage[all]{xy}
\usepackage{tikz}
\usepackage{tikz-cd}
\mathtoolsset{showonlyrefs=true}
\numberwithin{equation}{section}
\makeatletter
\newcommand{\shortmathcal}[1]{\@tfor\ch:=#1\do{
\expandafter\edef\csname c\ch\endcsname{\noexpand\mathcal{\ch}}
}}
\newcommand{\shortmathbb}[1]{\@tfor\ch:=#1\do{
\expandafter\edef\csname bb\ch\endcsname{\noexpand\mathbb{\ch}}
}}
\newcommand{\shortmathbf}[1]{\@tfor\ch:=#1\do{
\expandafter\edef\csname b\ch\endcsname{\noexpand\mathbf{\ch}}
}}
\newcommand{\shortboldsymbol}[1]{\@tfor\ch:=#1\do{
\expandafter\edef\csname bs\ch\endcsname{\noexpand\boldsymbol{\ch}}
}}
\newcommand{\shortmathfrak}[1]{\@tfor\ch:=#1\do{
\expandafter\edef\csname f\ch\endcsname{\noexpand\mathfrak{\ch}}}}
\newcommand{\shortmathscr}[1]{\@tfor\ch:=#1\do{
\expandafter\edef\csname s\ch\endcsname{\noexpand\mathscr{\ch}}}}
\newcommand{\shortmathrm}[1]{\@tfor\ch:=#1\do{
\expandafter\edef\csname r\ch\endcsname{\noexpand\mathrm{\ch}}
}}
\let\@@span\span
\def\sp@n{\@@span\omit\advance\@multicnt\m@ne}
\makeatother
\shortmathcal{ABCDEFGHIJKLMNOPQRSTUVWXYZ}
\shortmathbb{ABCDEFGHIJKLMNOPQRSTUVWXYZ}
\shortmathbf{ABCDEFGHIJKLMNOPQRSTUVWXYZabcdefghijklmnopqrstuvwxyz}
\shortboldsymbol{ABCDEFGHIJKLMNOPQRSTUVWXYZabcdefghijklmnopqrstuvwxyz}
\shortmathfrak{ABCDEFGHIJKLMNOPQRSTUVWXYZabcdefghjklmnopqrstuvwxyz}
\shortmathscr{ABCDEFGHIJKLMNOPQRSTUVWXYZ}
\shortmathrm{ABCDEFGHIJKLMNOPQRSTUVWXYZabcdefghijklmnopqrstuvwxyz}
\DeclareFontEncoding{OT2}{}{}
\DeclareFontSubstitution{OT2}{cmr}{m}{n}
\DeclareFontFamily{OT2}{cmr}{\hyphenchar\font45 }
\DeclareFontShape{OT2}{cmr}{m}{n}{
<5><6><7><8><9>gen*wncyr
<10><10.95><12><14.4><17.28><20.74><24.88>wncyr10}{}
\DeclareFontShape{OT2}{cmr}{b}{n}{
<5><6><7><8><9>gen*wncyb
<10><10.95><12><14.4><17.28><20.74><24.88>wncyb10}{}
\DeclareMathAlphabet{\mathcyr}{OT2}{cmr}{m}{n}
\DeclareMathAlphabet{\mathcyb}{OT2}{cmr}{b}{n}
\SetMathAlphabet{\mathcyr}{bold}{OT2}{cmr}{b}{n}

\newcommand{\Hom}{\mathrm{Hom}}
\newcommand{\End}{\mathsf{End}}
\newcommand{\id}{\mathrm{id}}
\newcommand{\Grp}{\mathsf{Grp}}
\newcommand{\Set}{\mathsf{Set}}
\newcommand{\Vect}{\mathsf{Vect}}

\newcommand{\Alg}{\mathsf{Alg}}
\newcommand{\CAlg}{\mathsf{CAlg}}

\newcommand{\rre}{\mathrm{re}}
\newcommand{\BiVa}{\mathsf{BiVa}}
\newcommand{\RepEnd}{\mathsf{RepEnd}}
\newcommand{\AffGrpSch}{\mathsf{AffGrpSch}}

\newcommand{\Lie}{\mathrm{Lie}}
\newcommand{\Spec}{\mathrm{Spec}}
\newcommand{\ad}{\mathrm{ad}}
\newcommand{\Ad}{\mathrm{Ad}}
\newcommand{\Der}{\mathrm{Der}}
\newcommand{\GL}{\mathrm{GL}}
\newcommand{\hotimes}{\mathbin{\widehat{\otimes}}}

\newcommand{\emp}{\varnothing}
\newcommand{\ep}{\varepsilon}
\newcommand{\jump}[1]{\ensuremath{[\![#1]\!]}}
\renewcommand{\mod}{\mathrm{mod}}
\DeclareMathOperator{\Span}{span}
\DeclareMathOperator{\Fil}{Fil}
\DeclareMathOperator{\gr}{gr}
\newcommand{\balpha}{\boldsymbol{\alpha}}

\newcommand{\bgamma}{\boldsymbol{\gamma}}
\newcommand{\sh}{\mathbin{\mathcyr{sh}}}
\newcommand{\BIFF}{\mathsf{BIFF}}

\newcommand{\MU}{\mathsf{MU}}
\newcommand{\LU}{\mathsf{LU}}
\newcommand{\BIMU}{\mathsf{BIMU}}
\newcommand{\as}{\mathsf{as}}
\newcommand{\al}{\mathsf{al}}
\newcommand{\mmu}{\mathsf{mu}}

\newcommand{\lu}{\mathsf{lu}}
\newcommand{\invmu}{\mathsf{invmu}}
\newcommand{\gepar}{\mathsf{gepar}}

\newcommand{\preari}{\mathsf{preari}}
\newcommand{\expari}{\mathsf{expari}}

\newcommand{\invgari}{\mathsf{invgari}}
\newcommand{\fragari}{\mathsf{fragari}}
\newcommand{\fragira}{\mathsf{fragira}}
\newcommand{\invgami}{\mathsf{invgami}}
\newcommand{\invgani}{\mathsf{invgani}}
\newcommand{\gira}{\mathsf{gira}}
\newcommand{\girat}{\mathsf{girat}}
\newcommand{\giwat}{\mathsf{giwat}}
\newcommand{\garit}{\mathsf{garit}}
\newcommand{\gari}{\mathsf{gari}}
\newcommand{\arit}{\mathsf{arit}}
\newcommand{\ari}{\mathsf{ari}}
\newcommand{\irat}{\mathsf{irat}}
\newcommand{\gaxit}{\mathsf{gaxit}}
\newcommand{\gaxi}{\mathsf{gaxi}}
\newcommand{\axit}{\mathsf{axit}}

\newcommand{\gamit}{\mathsf{gamit}}
\newcommand{\gami}{\mathsf{gami}}
\newcommand{\amit}{\mathsf{amit}}

\newcommand{\ganit}{\mathsf{ganit}}
\newcommand{\gani}{\mathsf{gani}}
\newcommand{\anit}{\mathsf{anit}}

\newcommand{\GARI}{\mathsf{GARI}}
\newcommand{\ARI}{\mathsf{ARI}}

\newcommand{\GAMI}{\mathsf{GAMI}}

\newcommand{\GANI}{\mathsf{GANI}}

\newcommand{\GIFF}{\mathsf{GIFF}}
\newcommand{\DIFF}{\mathsf{DIFF}}

\newcommand{\fre}{\mathfrak{re}}
\newcommand{\fSe}{\mathfrak{Se}}
\newcommand{\fes}{\mathfrak{es}}
\newcommand{\fos}{\mathfrak{os}}
\newcommand{\fess}{\mathfrak{ess}}
\newcommand{\foss}{\mathfrak{oss}}
\newcommand{\dess}{\ddot{\mathfrak{e}}\mathfrak{ss}}
\newcommand{\doss}{\ddot{\mathfrak{o}}\mathfrak{ss}}

\newcommand{\dSo}{\mathfrak{S}\ddot{\mathfrak{o}}}
\newcommand{\fez}{\mathfrak{ez}}
\newcommand{\foz}{\mathfrak{oz}}

\newcommand{\ful}[1]{{}_{#1}\lceil}
\newcommand{\fll}[1]{{}_{#1}\lfloor}
\newcommand{\fur}[1]{\rceil_{#1}}
\newcommand{\flr}[1]{\rfloor_{#1}}

\renewcommand{\neg}{\mathsf{neg}}
\newcommand{\swap}{\mathsf{swap}}
\newcommand{\anti}{\mathsf{anti}}
\newcommand{\pari}{\mathsf{pari}}
\newcommand{\push}{\mathsf{push}}
\newcommand{\pus}{\mathsf{pus}}
\newcommand{\mantar}{\mathsf{mantar}}
\newcommand{\gantar}{\mathsf{gantar}}
\newcommand{\ras}{\mathsf{ras}}
\newcommand{\rash}{\mathsf{rash}}
\newcommand{\crash}{\mathsf{crash}}
\renewcommand{\slash}{\mathsf{slash}}
\newcommand{\leng}{\mathsf{leng}}

\newcommand{\iwat}{\mathsf{iwat}}
\newcommand{\fHe}{\mathfrak{He}}
\newcommand{\invgira}{\mathsf{invgira}}
\newcommand{\fTe}{\mathfrak{Te}}

\newcommand{\dTo}{\mathfrak{T}\ddot{\mathfrak{o}}}
\newcommand{\der}{\mathsf{der}}
\newcommand{\dO}{\ddot{\fO}}
\newcommand{\dro}{\fr\ddot{\fo}}

\theoremstyle{plain}
\newtheorem{theorem}{Theorem}[section]
\newtheorem{proposition}[theorem]{Proposition}
\newtheorem{lemma}[theorem]{Lemma}
\newtheorem{corollary}[theorem]{Corollary}

\theoremstyle{definition}
\newtheorem{definition}[theorem]{Definition}
\theoremstyle{remark}
\newtheorem{remark}[theorem]{Remark}

\allowdisplaybreaks
\title{A Note on Flexion Units}
\author{Hanamichi Kawamura}
\address[Hanamichi Kawamura]{Department of Mathematics, Graduate School of Science, Tokyo University of Science, 1-3 Kagurazaka, Shinjuku-ku, Tokyo, 162-8601, Japan}
\email{1125512@ed.tus.ac.jp}
\allowdisplaybreaks
\begin{document}
\begin{abstract}
    This article is a survey of \'{E}calle's theory of flexion units.
    In particular, we provide complete proofs of several key assertions that were stated without proof in \'{E}calle's original works.
\end{abstract}
\maketitle
\tableofcontents
\section{Introduction}
This article treats \emph{flexion units} introduced by \'{E}calle \cite{ecalle11}.
They are functions with two variables satisfying certain conditions related to its parity and symmetry:
\begin{definition}[Rough version. See Definition \ref{def:unit} for actual uses]
    Let $\fE\colon G\times H\to R$ be a map from the direct product of additive abelian groups $G$ and $H$ to a ring $R$.
    We call $\fE$ a \emph{flexion unit} if,
    \begin{enumerate}
        \item it is an even function $\fE\binom{-u}{-v}=-\fE\binom{u}{v}$ for every $u\in G$ and $v\in H$, and
        \item they obey the following rule which is called the \emph{tripartite identity}: 
        \[\fE\binom{u_{1}}{v_{1}}\fE\binom{u_{2}}{v_{2}}=\fE\binom{u_{1}+u_{2}}{u_{1}}\fE\binom{u_{2}}{v_{2}-v_{1}}+\fE\binom{u_{1}}{v_{1}-v_{2}}\fE\binom{u_{1}+u_{2}}{u_{2}}.\]
    \end{enumerate}
\end{definition}
In \'{E}calle's theory of dimorphy, there are several functions constructed from flexion units, and they play a crucial role in the theory. 
In this note, we reveal how general flexion units and associated bimoulds behave.
Especially, we give completed proofs of the following properties of the \emph{primary} and \emph{secondary} bimoulds, that seem to have been unproven\footnote{Schneps \cite{schneps15} partially solved $\crash(\dess)=\fez$ and conjectured $\crash(\fess)=\fez$, in the case where $\fE$ is given as a \emph{polar} one.}:
\begin{theorem}[$=$ Proposition \ref{prop:crash_pal} and Theorem \ref{thm:crash_pil}]\label{thm:crash_main}
    We have $\crash(\fess)=\crash(\dess)=\fez$.
\end{theorem}
\begin{theorem}[$=$ Propositions \ref{prop:slash_pil} and \ref{prop:slash_pal}]
    We have $\slash(\fess)=\slash(\dess)=\fes$.
\end{theorem}
\begin{theorem}[$=$ Theorem \ref{thm:pal_bisymmetral}]
    The bimoulds $\fess$ and $\doss$ are $\gantar$-invariant.
\end{theorem}

Additionally, we explain the machinery of the group $\GIFF_{x}$, which is important to construct and investigate the secondary bimoulds.
In particular, we give a proof for the following assertion called \emph{\'{E}calle's separation lemma}:
\begin{theorem}[$=$ Theorem \ref{thm:separation}]
    Let $f(x)=x+\sum_{r=1}^{\infty}a_{r}^{f}x^{r+1}$ belong to $\GIFF_{x}(R)$.
    Then we have
    \[\gepar(\fSe_{R}(f))=\sum_{r=0}^{\infty}(r+1)a_{r}^{f}\mmu^{r}(\fO),\]
    where $a_{0}^{f}\coloneqq 1$.
\end{theorem}

This article is organized as follows: we review several basic notions of bimoulds in Section \ref{sec:basic}. 
In Section \ref{sec:flexion}, we show some algebraic structures defined in the terms of \emph{flexions} and give a proof for some statements that seem to be unproven. 
Sections \ref{sec:unit} and \ref{sec:primary} are used for introducing flexion units and the \emph{primary bimoulds} associated to such units.
Our main contents are treated in Section \ref{sec:secondary}, which is preceded by Section \ref{sec:appendix} in which we introduce a useful functor $\fSe$. 

\section{Basic of bimoulds}\label{sec:basic}
In what follows, $\Vect_{\bbK}$ (resp.~$\CAlg_{\bbK}$) denotes the category of $\bbK$-vector spaces (resp.~commutative $\bbK$-algebras).
\subsection{One formulation for bimoulds}
The following formulation for bimoulds is inspired by Furusho--Hirose--Komiyama's idea \cite[\S 1.1]{fhk23} which is called ``family of functions'', but it seems different from their one.

\begin{definition}\label{def:deck}
    \begin{enumerate}
        \item Let $\bbK$ be a field. We denote by $\BiVa_{\bbK}$ the category whose object is a non-negative integer and whose morphism is given by
        \[\Hom_{\BiVa_{\bbK}}(m,n)\coloneqq\Hom_{\Vect_{\bbK}}(V_{2m},V_{2n}),\]
        where $V_{2r}$ denotes the free $\bbK$-vector space $\bigoplus_{i=1}^{r}(\bbK x_{i}\oplus \bbK y_{i})$ spanned by formal symbols $\{x_{1},\ldots,x_{r},y_{1},\ldots,y_{r}\}$.
        \item We say that a pair $(\bbK,\cX)$ is a \emph{deck for bimoulds} or simply a \emph{deck} if it consists of the following data.
        \begin{enumerate}
            \item $\bbK$ is a field whose characteristic is $0$. 
            \item $\cX$ is a functor $\BiVa_{\bbK}\to\RepEnd(\CAlg_{\bbK})$; $r\mapsto \cX_{r}$, where the codomain is the full subcategory of $\End(\CAlg_{\bbK})$ whose object is an endofunctor $F\colon\CAlg_{\bbK}\to\CAlg_{\bbK}$ which is corepresentable when regarded as copresheaf $(\text{\texttt{forget}})\circ F\colon\CAlg_{\bbK}\to\Set$.
        \end{enumerate}
        For practical applications, furthermore we always assume that the first component $\cX_{0}$ of a deck $(\bbK,\cX)$ is naturally isomorphic to the affine line $\bbA_{\bbK}^{1}$ as schemes. 
    \end{enumerate}
\end{definition}
From the definition of $\RepEnd(\CAlg_{\bbK})$, its object can be understood as an affine scheme over $\bbK$. Moreover, focusing on the product structure of its value, we regard $\RepEnd(\CAlg_{\bbK})$ as a full subcategory of the category $\AffGrpSch_{\bbK}$ of affine group schemes over $\bbK$ (indeed it is always abelian as a group object). In this point of view, for any deck $(\bbK,\cX)$, we denote by $C_{\cX,r}$ the canonical coordinate ring $\Hom_{[\CAlg_{\bbK},\Grp]}(\cX_{r},\bbA^{1}_{\bbK})$ of $\cX_{r}$.
\\

As a convenient notation, we here explain our implementation of \emph{substitution} in this setting (see also \cite[p.~3, l.~5--7]{fhk23}):
\begin{definition}[Substitution]\label{def:substitution}
    Let $m$ and $n$ be non-negative integers and $f\in\Hom_{\BiVa_{\bbK}}(m,n)$. Put $u_{i}\coloneqq f(x_{i})$ and $v_{i}\coloneqq f(y_{i})$ for $1\le i\le m$.
    For a commutative $\bbK$-algebra $R$ and an element $M\in \cX_{m}(R)$, we denote by $M\binom{u_{1},\ldots,u_{m}}{v_{1},\ldots,v_{m}}$ the image of $M$ under the morphism $\cX_{f,R}\colon \cX_{m}(R)\to \cX_{n}(R)$. 
\end{definition}
In the rest of this paper, we fix a deck $(\bbK,\cX)$ and use the symbol $R$ to indicate an arbitrary commutative $\bbK$-algebra $R$, unless otherwise noted.

\subsection{Module of bimoulds}

\begin{definition}[Bimoulds]
    In this paper, we define the set $\BIMU(R)$ as
    \begin{equation}\label{eq:bimould}
        \BIMU(R)\coloneqq\prod_{r=0}^{\infty}\cX_{r}(R).
    \end{equation}    
    We call its element a \emph{bimould}. A bimould of the form $(\underbrace{0,\ldots,0}_{r-1},M,0,\ldots)$ is called a \emph{bimould of length $r$} for a non-negative integer $r$. For emphasizing, we sometimes use the notation $\BIMU_{r}(R)$ instead of $\cX_{r}(R)$.
\end{definition}

\begin{remark}
    When we indicate a specified component of a bimould, say it is the $r$-th one, we usually use the notation $M\binom{u_{1},\ldots,u_{r}}{v_{1},\ldots,v_{r}}$, following the principle of substitution (Definition \ref{def:substitution}). If $r$ is $0$, it always becomes simply $M(\emp)$, which belongs to the $\bbK$-algebra $\cX_{0}(R)=R$.
    This notation includes the symbols $u_{i}\coloneqq f(x_{i})$ and $v_{i}\coloneqq f(y_{i})$ that stand for the images under $f\in\Hom_{\BiVa_{\bbK}}(r,s)$ with arbitrary choices of $s$ and $f$.
    Moreover, we use the abbreviation $w_{i}\coloneqq\binom{u_{i}}{v_{i}}$ and
    \[\bw\coloneqq w_{1}\cdots w_{r}\coloneqq (w_{1},\ldots,w_{r})\coloneqq\binom{u_{1},\ldots,u_{r}}{v_{1},\ldots,v_{r}},\]
    as if they are a kind of words in the context of free monoids. We sometimes write the length of such sequences as $r=\ell(\bw)$. For a further precise formulation, one may identify a letters, $u_{i}$ or $v_{i}$, with elements of the direct limit $\varinjlim_{n} n$ in the category $\BiVa_{\bbK}$, where the transition $n\to n+1$ given by the identical inclusion. We sometimes denote by $\fV$ the set of sequences $\bw=(w_{1},\ldots,w_{r})$.
\end{remark}

In the defining equation \eqref{eq:bimould}, we used the product symbol $\prod$ in the right-hand side to indicate a direct product in $\Set$. By regarding it as a direct product taken in $\Vect_{\bbK}$ and endowing it with the pointwise $R$-action, we get an $R$-module structure in $\BIMU(R)$ and of course in each $\BIMU_{r}(R)$. Therefore $\BIMU$ gives a functor on $\CAlg_{\bbK}$ to the category of $R$-modules.

For later use, we prepare a topological structure on $\BIMU(R)$.
Define the decreasing filtration $\{\Fil^{\ge n}\BIMU(R)\}_{n\ge 0}$ by
\[\Fil^{\ge n}\BIMU(R)\coloneqq\begin{cases}
    \left\{M\in\BIMU(R)\mid M(w_{1},\ldots,w_{r})=0~(0\le r<n)\right\} & \text{if }n\ge 1,\\
    \BIMU(R) & \text{if }n=0,
\end{cases}\]
and equip $\BIMU(R)$ with the topology induced by this filtration.
Then $\BIMU(R)$ becomes a topological $R$-module, and we see that any bimould $M$ has the expression as an infinite sum
\[M=\sum_{r=0}^{\infty}\leng_{r}(M),\]
where the summand $\leng_{r}(M)\in\BIMU_{r}(R)$ is defined as
\[\leng_{r}(M)(w_{1},\ldots,w_{n})\coloneqq\begin{cases}
    M(w_{1},\ldots,w_{n}) & \text{if }r=n,\\
    0 & \text{if }r\neq n.
\end{cases}\]

\subsection{Unary operations on bimoulds}
We here list up some useful unary operators defined on $\BIMU(R)$: for $r\ge 1$, we define

\begin{align}\label{eq:operators}
    \begin{split}
    \neg(A)\binom{u_{1},\ldots,u_{r}}{v_{1},\ldots,v_{r}}&\coloneqq A\binom{-u_{1},\ldots,-u_{r}}{-v_{1},\ldots,-v_{r}},\\
    \anti(A)\binom{u_{1},\ldots,u_{r}}{v_{1},\ldots,v_{r}}&\coloneqq A\binom{u_{r},\ldots,u_{1}}{v_{r},\ldots,v_{1}},\\
    \pari(A)\binom{u_{1},\ldots,u_{r}}{v_{1},\ldots,v_{r}}&\coloneqq(-1)^{r}A\binom{u_{1},\ldots,u_{r}}{v_{1},\ldots,v_{r}},\\
    \pus(A)\binom{u_{1},\ldots,u_{r}}{v_{1},\ldots,v_{r}}&\coloneqq A\binom{u_{r},u_{1},\ldots,u_{r-1}}{v_{r},v_{1},\ldots,v_{r-1}},\\
    \push(A)\binom{u_{1},\ldots,u_{r}}{v_{1},\ldots,v_{r}}&\coloneqq A\binom{-u_{1}-\cdots-u_{r},u_{1},\ldots,u_{r-1}}{-v_{r},v_{1}-v_{r},\ldots,v_{r-1}-v_{r}},\\
    \mantar(A)\binom{u_{1},\ldots,u_{r}}{v_{1},\ldots,v_{r}}&\coloneqq(-1)^{r-1}A\binom{u_{r},\ldots,u_{1}}{v_{r},\ldots,v_{1}}.
    \end{split}
\end{align}

Note that the five operators, except for $\mantar$, preserve the value of a bimould at $\emp$. We also define an important self-involution $\BIMU(R)\to\BIMU(R)$:
\[\swap(A)\binom{u_{1},\ldots,u_{r}}{v_{1},\ldots,v_{r}}\coloneqq A\binom{v_{r},v_{r-1}-v_{r},\ldots,v_{1}-v_{2}}{u_{1}+\cdots+u_{r},\ldots,u_{1}+u_{2},u_{1}}.\]
Following \'{E}calle \cite{ecalle11}, we call $\swap(A)$ the \emph{swappee} of a bimould $A$.\\

The following formula is easy to prove, but plays an important role in several computation on bimoulds.
\begin{proposition}[{\cite[(2.12)]{ecalle11}}]\label{prop:negpush}
    We have
    \begin{equation}\label{eq:negpush}
        \neg\circ\push=\anti\circ\swap\circ\anti\circ\swap=\mantar\circ\swap\circ\mantar\circ\swap.
    \end{equation}
\end{proposition}
\begin{corollary}
    The six operators $\anti\circ\push$, $\push\circ\anti$, $\swap\circ\push$, $\push\circ\swap$, $\mantar\circ\push$ and $\push\circ\mantar$ are involutions.
\end{corollary}
\begin{proof}
    From Proposition \ref{prop:negpush}, we have
    \[\anti\circ\push=\neg\circ\swap\circ\anti\circ\swap\]
    and the right-hand side is clearly involutive from commutativity of $\neg$ with the others. We can show the involutivity of the others in a similar way.
\end{proof}

\subsection{Uninflected substructure in bimoulds}
There is a standard product rule for bimoulds. Following \cite{ecalle11}, we denote it by $\mmu$ and call it the \emph{uninflected product}.

\begin{definition}[Uninflected product]
    Let $A$ and $B$ be bimoulds in $\BIMU(R)$.
    Then we define a new bimould $\mmu(A,B)\in\BIMU(R)$ by the equation
    \begin{equation}\label{eq:mu}
        \mmu(A,B)(w_{1},\ldots,w_{r})\coloneqq\sum_{i=0}^{r}A(w_{1},\ldots,w_{i})B(w_{i+1},\ldots,w_{r})
    \end{equation}
\end{definition}

Since the product rule $\mmu$ is $R$-bilinear, unital, associative but non-commutative by definition, $\BIMU(R)$ becomes an $R$-algebra.
The unit is given by the bimould $(1,0,0,\ldots)$ of length $0$, which is denoted as $1$ by abuse of notation. Similarly, we denote by $0$ the unit element $(0,0,\ldots)\in\BIMU(R)$ for the addition of bimoulds.
\begin{lemma}\label{lem:operators}
    Both $\neg$ and $\pari$ give an $R$-algebra homomorphism $\BIMU(R)\to\BIMU(R)$. The operator $\anti$ behaves linearly as well, but it exchanges the order of the product $\mmu$: the formula $\anti(\mmu(A,B))=\mmu(\anti(B),\anti(A))$ for any $A,B\in\BIMU(R)$.
    Consequently, for $A,B\in\BIMU(R)$ we have $\mantar(\mmu(A,B))=-\mmu(\mantar(A),\mantar(B))$.
\end{lemma}
\begin{proof}
    These assertions are immediate from the definitions. For the last equality, use $\pari\circ\anti+\mantar=0$.
\end{proof}
\begin{definition}
    We put
    \[\MU(R)\coloneqq\{M\in\BIMU(R)\mid M(\emp)=1\}\]
    and
    \[\LU(R)\coloneqq\{M\in\BIMU(R)\mid M(\emp)=0\}.\]
\end{definition}
\begin{proposition}
    \begin{enumerate}
        \item $(\MU(R),\mmu)$ is a group.
        \item $(\LU(R),\lu)$ is a Lie algebra over $\bbK$, where $\lu(A,B)\coloneqq\mmu(A,B)-\mmu(B,A)$ is the commutator for $\mmu$. 
    \end{enumerate}
\end{proposition}
\begin{proof}
    For the first assertion, indeed we can construct length by length the inverse element for each given element of $\MU(R)$, from the defining equality \eqref{eq:mu} and the assumption that its constant term ($=$ value at $\emp$) is $0$.
    For $\lu$, it obviously satisfies the bilinearity, antisymmetry and Jacobi's identity since it is a commutator. The condition about the constant term is also easily checked.
\end{proof}
From the above proposition, we can always the inverse element for a given bimould $A$ in $\MU(R)$. It is denoted by $\invmu(A)$.

\subsection{Symmetries}
\begin{definition}[Shuffle product]
    Let $L$ be a set and $\cR$ the free non-commutative $\bbK$-algebra generated by $L$. We equip $\cR$ with the $\bbK$-bilinear product $\sh$ in the following inductive way:
    \[\emp \sh \bw=\bw\sh\emp=\bw\]
    and
    \[w\bw\sh w'\bw'=w(\bw\sh w'\bw')+w'(w\bw\sh\bw')\]
    for arbitrary $w,w'\in L$ and $\bw,\bw'\in\cR$.
\end{definition}

It is well-known that $(\cR,\sh)$ is a commutative $\bbK$-algebra.
Applying these rules to the variables of bimoulds, we define two types of symmetries.

\begin{definition}[Alternality and symmetrality]
    Let $A$ belong to $\MU(R)$ (resp.~$\LU(R)$) and linearly extend it on $\Span_{\bbK}\fV$.
    We say that $A$ is \emph{symmetral} (resp.~\emph{alternal}) if, for all pairs $(\bw,\bw')$ of non-emptys, the equation
    \[A(\bw \sh \bw')=A(\bw)A(\bw')\qquad (\text{resp.}~A(\bw\sh\bw')=0)\]
    holds. Let us denote by $\MU_{\as}(R)$ (resp.~$\LU_{\al}(R)$) the set consisting of all symmetral (resp.~alternal) bimoulds.
\end{definition}

\begin{proposition}[{\cite[\S 2.5, \S 2.6]{ecalle11}}]
    The pair $(\MU_{\as}(R),\mmu)$ is a subgroup of $(\MU(R),\mmu)$. Similarly, $(\LU_{\al}(R),\lu)$ is a Lie subalgebra of $(\LU(R),\lu)$.
\end{proposition}
\begin{proof}
    The proof is given in \cite[Proposition 2.9]{komiyama21}.
\end{proof}

\begin{proposition}[{\cite[\S 2.4]{ecalle11}}]\label{prop:mantar_gantar}
    Every element of $\MU_{\as}(R)$ (resp.~$\LU_{\al}(R)$) is $\gantar$-invariant (resp.~$\mantar$-invariant). Here the operator $\gantar$ is defined to be $\invmu\circ\pari\circ\anti$.
\end{proposition}
\begin{proof}
    It is a basic fact of the shuffle algebra that the identity
    \begin{equation}\label{eq:antipode}
        \sum_{i=0}^{r}(-1)^{i}(w_{1},\ldots,w_{i})\sh (w_{r},\ldots, w_{i+1})=0
    \end{equation}
    holds for all $r\ge 1$ and arbitrary $w_{1},\ldots,w_{r}$.
    Letting $A$ be an element of $\LU_{\al}(R)$ and apply $A$ on both sides of \eqref{eq:antipode}, we see that the terms with $i=1,\ldots,r-1$ vanish due to the alternality of $A$. Then we obtain the equality
    \[A(w_{1},\ldots,w_{r})+(-1)^{r}A(w_{1},\ldots,w_{r})=0,\]
    which implies the $\mantar$-invariance. Similarly, for $A\in\MU_{\as}(R)$, by virtue of \eqref{eq:antipode} we have
    \begin{align}
    \mmu(A,(\pari\circ\anti)(A))(w_{1},\ldots,w_{r})
    &=\sum_{i=0}^{r}(-1)^{r-i}A(w_{1},\ldots,w_{i})A(w_{r},\ldots,w_{i+1})\\
    &=(-1)^{r}\sum_{i=0}^{r}(-1)^{i}A((w_{1},\ldots,w_{i})\sh (w_{r},\ldots,w_{i+1}))\\
    &=0
    \end{align}
    whenever $r\ge 1$. Thus the bimould $\mmu(A,(\pari\circ\anti)(A))$ is of length $0$, whose only non-zero component is $1$ because $A\in\MU(R)$. This shows $A=(\invmu\circ\pari\circ\anti)(A)=\gantar(A)$.
\end{proof}

\section{Flexion operations}\label{sec:flexion}

\begin{definition}[Flexion markers]
    Let $\alpha=\binom{u_{1},\ldots,u_{i}}{v_{1},\ldots,v_{i}}$ and $\beta=\binom{u_{i+1},\ldots,u_{r}}{v_{i+1},\ldots,v_{r}}$. Then we define
    \begin{align}
        \ful{\alpha}\beta&\coloneqq\binom{u_{1}+\cdots+u_{i}+u_{i+1},u_{i+2},\ldots,u_{r}}{v_{i+1},\ldots,v_{r}},\\
        \alpha\fur{\beta}&\coloneqq\binom{u_{1},\ldots,u_{i-1},u_{i}+u_{i+1}+\cdots+u_{r}}{v_{1},\ldots,v_{i}},\\
        \fll{\alpha}\beta&\coloneqq\binom{u_{i+1},\ldots,u_{r}}{v_{i+1}-v_{i},\ldots,v_{r}-v_{i}},\\
        \alpha\flr{\beta}&\coloneqq\binom{u_{1},\ldots,u_{i}}{v_{1}-v_{i+1},\ldots,v_{i}-v_{i+1}}.
    \end{align}
\end{definition}

If $i=0$, we understand as $\ful{\alpha}\beta=\fll{\alpha}\beta=\beta$ and $\alpha\fur{\beta}=\alpha\flr{\beta}=\emp$.

\subsection{Flexion Lie algebras}

\begin{definition}[$\amit,\anit$ and $\axit$]
    Let $A$ be an element of $\LU(R)$.
    Then we define two operators $\amit(A),\anit(A)\colon\BIMU(R)\to\BIMU(R)$ by the equalities
    \[\amit(A)(B)(\bw)\coloneqq\sum_{\substack{\bw=\ba\bb\bc\\ \bb,\bc\neq\emp}}A(\ba\ful{\bb}\bc)B(\bb\flr{\bc})\]
    and
    \[\anit(A)(B)(\bw)\coloneqq\sum_{\substack{\bw=\ba\bb\bc\\ \ba,\bb\neq\emp}}A(\ba\fur{\bb}\bc)B(\fll{\ba}\bb),\]
    where $B$ is an arbitrary bimould in $\BIMU(R)$.
    Using them, we define the operator $\axit(A,A')$ on $\BIMU(R)$ attached to two elements $A,A'\in\LU(R)$ by the equation
    \[\axit(A,A')\coloneqq\amit(A)+\anit(A').\]
\end{definition}
A proof of the following important property can be found in \cite[Proposition 2.2.1, \S A.1]{schneps15}.
\begin{proposition}[{\cite[pp.~424--425]{ecalle03}}]\label{prop:axit_derivation}
    For any pair $(A,A')$ of elements in $\LU(R)$, the operator $\axit(A,A')\colon\BIMU(R)\to\BIMU(R)$ is a derivation about $\mmu$.
\end{proposition}
Restricting the pair in $\axit$ to the case where their sum is $0$, we define the operator $\arit$, namely, we put $\arit(A)\coloneqq\axit(A,-A)$ for $A\in\LU(R)$. Using it, we define the operation
\[\preari(A,B)\coloneqq\arit(B)(A)+\mmu(A,B)\qquad (A\in\LU(R),B\in\BIMU(R))\]
and its commutator
\[\ari(A,B)\coloneqq\preari(A,B)-\preari(B,A)\qquad (A,B\in\LU(R)).\]
\begin{proposition}[{\cite[p.~425]{ecalle03}}]
    The pair $\ARI(R)\coloneqq (\LU(R),\ari)$ is a Lie algebra over $\bbK$. Moreover, $\ARI_{\al}(R)\coloneqq(\LU_{\al}(R),\ari)$ becomes a Lie subalgebra of it.
\end{proposition}
\begin{proof}
    We can find a proof in \cite[\S A.1, \S A.2]{fk23}.
\end{proof}
\begin{lemma}\label{lem:neg_pari_axit}
    Applying either $\neg$ or $\pari$ inside $\axit$ amounts to taking the conjugation by them. Namely, we have
    \[\axit(h(A),h(A'))=h\circ\axit(A,A')\circ h\]
    for every $A,A'\in\LU(R)$ and $h\in\{\neg,\pari\}$.
    Moreover, the operator $\mantar\colon\ARI(R)\to\ARI(R)$ is a Lie algebra homomorphism ($=$ $\mantar$ preserves $\ari$).
\end{lemma}
\begin{proof}
    The assertion about $\neg$ and $\pari$ is obvious from the definition.
    We prove the claim related with $\mantar$. Let $A$ and $B$ be elements of $\LU(R)$.
    Using the property $\anti\circ\amit(\anti(X))\circ\anti=\anit(X)$ ($X\in\LU(R)$) which easily follows from the definition, we obtain
    \begin{align}
        &(\mantar\circ\arit(B))(A)\\
        &=(\pari\circ\anti\circ\anit(B))(A)-(\pari\circ\anti\circ\amit(A))(B)\\
        &=(\pari\circ\amit(\anti(B))\circ\anti)(A)-(\pari\circ\anit(\anti(A))\circ\anti)(B)\\
        &=\amit((\pari\circ\anti(B))\circ\pari\circ\anti)(A)-(\anit((\pari\circ\anti(A)))\circ\pari\circ\anti)(B)\\
        &=(\amit(\mantar(B))\circ\mantar)(A)-(\anit(\mantar(B))\circ\mantar)(B)\\
        &=(\arit(\mantar(B))\circ\mantar)(A).
    \end{align}
    Combining this computation and Lemma \ref{lem:operators}, we obtain
    \begin{align}
        \mantar(\ari(A,B))
        &=\mantar(\preari(A,B)-\preari(B,A))\\
        &=\mantar(\arit(B)(A)-\arit(A)(B)+\mmu(A,B)-\mmu(B,A))\\
        &=\begin{multlined}[t]
            (\arit(\mantar(B))\circ\mantar)(A)-(\arit(\mantar(A))\circ\mantar)(B)\\
            -\mmu(\mantar(B),\mantar(A))+\mmu(\mantar(A),\mantar(B))
        \end{multlined}\\
        &=\preari(\mantar(A),\mantar(B))-\preari(\mantar(B),\mantar(A))\\
        &=\ari(\mantar(A),\mantar(B)).  
    \end{align}
\end{proof}
\begin{proposition}[{\cite[(42)]{ecalle03}}]\label{prop:arit_composition}
    The map $\arit$ gives a Lie algebra anti-homomorphism from $\ARI(R)$ to $\Der(\BIMU(R))$. Namely, for any $A,B\in\ARI(R)$ we have
    \[\arit(B)\circ\arit(A)-\arit(A)\circ\arit(B)=\arit(\ari(A,B)).\]
\end{proposition}
\begin{proof}
    It is done in \cite[Proposition A.4]{fk23}.
\end{proof}

There are many variants of $\arit$ generated from $\axit$: here we define
\[\irat(A)\coloneqq\axit(A,-\push(A))\]
and
\[\iwat(A)\coloneqq\axit(A,\anti(A)).\]

\subsection{Flexion groups}
\begin{definition}[Anti-action $\gaxit$; {\cite[(41)]{ecalle03}}]
    Let $A_{1}$ and $A_{2}$ be elements of $\MU(R)$.
    Then we define an $R$-linear operator $\gaxit(A_{1},A_{2})\colon\BIMU(R)\to\BIMU(R)$ by
    \begin{multline}
        \gaxit(A_{1},A_{2})(B)(\bw)\\\coloneqq\begin{dcases}
            \sum_{s=1}^{\infty}\sum_{\{(\ba_{i};\bb_{i};\bc_{i})\}_{i=1}^{s}\in\cE_{s}(\bw)} B(\ful{\ba_{1}}\bb_{1}\fur{\bc_{1}}\cdots\ful{\ba_{s}}\bb_{s}\fur{\bc_{s}})\prod_{i=1}^{s}A_{1}(\ba_{i}\flr{\bb_{i}})A_{2}(\fll{\bb_{i}}\bc_{i}) & \text{if }\bw\neq\emp,\\
            B(\emp) & \text{if }\bw=\emp,
        \end{dcases}
    \end{multline}
    where $B\in\BIMU(R)$ is an arbitrary bimould and the inner sum is over
    \[\cE_{s}(\bw)\coloneqq\left\{(\ba_{1};\bb_{1};\bc_{1};\cdots;\ba_{s};\bb_{s};\bc_{s})~\left|~\begin{array}{c}\prod_{i=1}^{s}\ba_{i}\bb_{i}\bc_{i}=\bw,\\ \bb_{i}\neq\emp~(1\le i\le s),\\\bc_{i}\ba_{i+1}\neq\emp~(1\le i\le s-1)\end{array}\right.\right\}.\]
    As specializations of $\gaxit$, we define $\gamit(A)\coloneqq\gaxit(A,1)$ and $\ganit(A)\coloneqq\gaxit(1,A)$ for $A\in\MU(R)$.
\end{definition}

\begin{definition}[$\gaxi$-product]
    Let $\cA=(A_{1},A_{2})$ and $\cB=(B_{1},B_{2})$ be two pairs of elements in $\MU(R)$.
    Then we associate them with a new pair $\gaxi(\cA,\cB)$ by the formula
    \begin{equation}\label{eq:gaxi_definition}
        \gaxi(\cA,\cB)\coloneqq (\mmu(\gaxit(\cB)(A_{1}),B_{1}),\mmu(B_{2},\gaxit(\cB)(A_{2}))).
    \end{equation}
\end{definition}

\begin{lemma}\label{lem:gaxit_simple}
    Let us denote as
    \[\Sigma_{t}(w_{1},\ldots,w_{r})\coloneqq\left\{(\bA_{1};w_{i_{1}};\bC_{1};\cdots;\bA_{t};w_{i_{t}};\bC_{t})~\left|~\begin{array}{c}\prod_{j=1}^{t}\bA_{j}w_{i_{j}}\bC_{j}=w_{1}\cdots w_{r},\\ 1\le i_{1}<\cdots<i_{t}\le r\end{array}\right.\right\}.\]
    Then we have
    \begin{multline}\label{eq:gaxit_simple}
        \gaxit(A_{1},A_{2})(B)(\bw)\\
        =\sum_{t=1}^{\infty}\sum_{\{(\bA_{j};w_{i_{j}};\bC_{j})\}_{j=1}^{t}\in\Sigma_{s}(\bw)} B(\ful{\bA_{1}}w_{i_{1}}\fur{\bC_{1}}\cdots\ful{\bA_{t}}w_{i_{t}}\fur{\bC_{t}})\prod_{j=1}^{t}A_{1}(\bA_{j}\flr{w_{i_{j}}})A_{2}(\fll{w_{i_{j}}}\bC_{j}),
    \end{multline}
    for all $A_{1},A_{2}\in\MU(R)$ and $B\in\BIMU(R)$.
\end{lemma}
\begin{proof}
    This proof is almost based on existence of a nice bijection \[\Phi\colon\bigsqcup_{s=1}^{\infty}\cE_{s}(\bw)\to\bigsqcup_{t=1}^{\infty}\Sigma_{t}(\bw).\]
    We construct it by the following procedure: take $s\ge 1$ and $\cP=(\ba_{i};\bb_{i};\bc_{i})_{i=1}^{s}\in\cE_{s}(\bw)$.
    Next name each letters contained in $\bb_{i}$ as $\bb_{i}=(w_{a_{i}+1},\ldots,w_{b_{i}})$ for each $i=1,\ldots,s$ by using the sequence $(a_{1},b_{1},\ldots,a_{s},b_{s})$ of integers.
    By the definition of $\cE_{s}(\bw)$, it has to satisfy the inequality
    \[0\le a_{1}<b_{1}<\cdots<a_{s}<b_{s}\le r.\]
    Then we choose a suitable $t$ as $t\coloneqq\sum_{i=1}^{s}(b_{i}-a_{i})$, and associate $\Phi(\cP)\in\Sigma_{t}(\bw)$ with $\cP$ as
    \[\Phi(\cP)\coloneqq\{(\ba_{i};w_{a_{i}+1};\emp;\emp;w_{a_{i}+2};\emp;\cdots;\emp;w_{b_{i}-1};\emp;\emp;w_{b_{i}};\bc_{i})\}_{i=1}^{s}.\]
    Its inverse map is given as follows: for arbitrary $t$ and $\cS\coloneqq\{(\bA_{j};w_{i_{j}};\bC_{j})\}_{j=1}^{t}\in\Sigma_{t}(\bw)$, we consider the set
    \[\{e_{1},\ldots,e_{s-1}\}\coloneqq\{j\mid \bC_{j}\bA_{j+1}\neq\emp\}\subseteq\{1,\ldots,t-1\}\]
    with the order $e_{1}<\cdots<e_{s-1}$. Then we put $e_{0}\coloneqq 1$ and
    \[\Phi^{-1}(\cS)=\{(\bA_{e_{l-1}};w_{i_{e_{l-1}}}w_{i_{e_{l-1}}+1}\cdots w_{i_{e_{l}}};\bC_{e_{l}})\}_{l=1}^{s}.\]
    Now let us consider the more general set of partitions of $\bw$:
    \[\cT_{n}(\bw)\coloneqq\left\{(\ba_{1};\bb_{1};\bc_{1};\cdots;\ba_{n};\bb_{n};\bc_{n})~\left|~\prod_{i=1}^{n}\ba_{i}\bb_{i}\bc_{i}=\bw\right.\right\}~(\supseteq \cE_{n}(\bw),\Sigma_{n}(\bw)).\]
    For each element $\cP\in\cT_{n}(\bw)$ and bimoulds $A_{1},A_{2},B$ we have in the statement, we associate the value
    \[W(\cP)\coloneqq B(\ful{\ba_{1}}\bb_{1}\fur{\bc_{1}}\cdots\ful{\ba_{s}}\bb_{s}\fur{\bc_{s}})\prod_{i=1}^{s}A_{1}(\ba_{i}\flr{\bb_{i}})A_{2}(\fll{\bb_{i}}\bc_{i})\in R.\]
    Then, by definition we have
    \[\gaxit(A_{1},A_{2})(B)(\bw)=\sum_{s=1}^{\infty}\sum_{\cP\in\cE_{s}(\bw)}W(\cP),\]
    while values of $W$ are always invariant the action of $\Phi$, due to its construction and the assumption $A_{1}(\emp)=A_{2}(\emp)=1$. Thus we get
    \begin{align}
        \gaxit(A_{1},A_{2})(B)(\bw)
        &=\sum_{s=1}^{\infty}\sum_{\cP\in\cE_{s}(\bw)}W(\cP)\\
        &=\sum_{t=1}^{\infty}\sum_{\cT\in\Sigma_{t}(\bw)}W(\Phi^{-1}(\cT))\\
        &=\sum_{t=1}^{\infty}\sum_{\cT\in\Sigma_{t}(\bw)}W(\cT).
    \end{align}
    The last line above is just the right-hand side of \eqref{eq:gaxit_simple}.
\end{proof}

\begin{proposition}[{\cite[p.~425]{ecalle03}}]\label{prop:gaxit_associative}
    Let $\cA$ and $\cB$ be elements of $\MU(R)\times\MU(R)$. Then we have
    \[\gaxit(\cB)\circ\gaxit(\cA)=\gaxit(\gaxi(\cA,\cB)).\]
\end{proposition}
\begin{proof}
    Write $\cA=(A_{1},A_{2})$ and $\cB=(B_{1},B_{2})$, and take an arbitrary $M\in\BIMU(R)$.
    By \eqref{eq:gaxit_simple} and \eqref{eq:gaxi_definition}, we have
    \begin{multline}
        \gaxit(\gaxi(\cA,\cB))(M)(\bw)
        =\sum_{t=1}^{\infty}\sum_{\{(\bA_{j};w_{i_{j}};\bC_{j})\}_{j=1}^{t}\in\Sigma_{t}(\bw)}M\left(\prod_{j=1}^{t}\ful{\bA_{j}}w_{i_{j}}\fur{\bC_{j}}\right)\\
        \cdot\prod_{j=1}^{t}\mmu(\gaxit(\cB)(A_{1}),B_{1})(\bA_{j}\flr{w_{i_{j}}})\mmu(B_{2},\gaxit(\cB)(A_{2}))(\fll{w_{i_{j}}}\bC_{j}).
    \end{multline}
        From the definition of $\mmu$ and the fact that the lower flexion markers are distributive for the concatenation, we obtain
        \begin{multline}
            \gaxit(\gaxi(\cA,\cB))(M)(\bw)
        =\sum_{t=1}^{\infty}\sum_{\{(\bA_{j};w_{i_{j}};\bC_{j})\}_{j=1}^{t}\in\Sigma_{t}(\bw)}\sum_{\substack{\bA_{j}=\ba_{j}\balpha_{j}\\ \bC_{j}=\bgamma_{j}\bc_{j}\\ 1\le j\le t}}M\left(\prod_{j=1}^{t}\ful{\bA_{j}}w_{i_{j}}\fur{\bC_{j}}\right)\\
        \cdot\prod_{j=1}^{t}\left(\gaxit(\cB)(A_{1})(\ba_{j}\flr{w_{i_{j}}})B_{1}(\balpha_{j}\flr{w_{i_{j}}})B_{2}(\fll{w_{i_{j}}}\bgamma_{j})\gaxit(\cB)(A_{2})(\fll{w_{i_{j}}}\bgamma_{j})\right).
        \end{multline}
        Using \eqref{eq:gaxit_simple} again, it is shown that
        \begin{multline}
            \gaxit(\gaxi(\cA,\cB))(M)(\bw)
            =\sum_{t=1}^{\infty}\sum_{\{(\bA_{j};w_{i_{j}};\bC_{j})\}_{j=1}^{t}\in\Sigma_{t}(\bw)}M\left(\prod_{j=1}^{t}\ful{\bA_{j}}w_{i_{j}}\fur{\bC_{j}}\right)\\
            \cdot \sum_{\substack{\bA_{j}=\ba_{j}\balpha_{j}\\ \bC_{j}=\bgamma_{j}\bc_{j}}}\sum_{\substack{\sigma_{1},\ldots,\sigma_{t}>0\\ \tau_{1},\ldots,\tau_{t}>0}}\sum_{\substack{\{(\ba_{j,k};w_{g_{j,k}};\bc_{j,k})\}_{k=1}^{\sigma_{j}}\in\Sigma_{\sigma_{j}}(\ba_{j})\\ 1\le j\le t}}\sum_{\substack{\{(\ba'_{j,l};w_{h_{j,l}};\bc'_{j,l})\}_{l=1}^{\tau_{j}}\in\Sigma_{\tau_{j}}(\bc_{j})\\ 1\le j\le t}}\prod_{j=1}^{t}\left(B_{1}(\balpha_{j}\flr{w_{i_{j}}})B_{2}(\fll{w_{i_{j}}}\bgamma_{j})\right)\\
        \cdot\prod_{j=1}^{t}\left(A_{1}\left(\prod_{k=1}^{\sigma_{j}}\ful{\left(\ba_{j,k}\flr{w_{i_{j}}}\right)}\left(w_{g_{j,k}}\flr{w_{i_{j}}}\right)\fur{\left(\bc_{j,k}\flr{w_{i_{j}}}\right)}\right)A_{2}\left(\prod_{l=1}^{\tau_{j}}\ful{\left(\fll{w_{i_{j}}}\ba'_{j,l}\right)}\fll{w_{i_{j}}}w_{h_{j,l}}\fur{\left(\fll{w_{i_{j}}}\bc'_{j,l}\right)}\right)\right)\\
        \cdot\prod_{j=1}^{t}\prod_{k=1}^{\sigma_{j}}\left(B_{1}\left(\ba_{j,k}\flr{w_{i_{j}}}\flr{\left(w_{g_{j,k}}\flr{w_{i_{j}}}\right)}\right)B_{2}\left(\fll{\left(w_{g_{j,k}}\flr{w_{i_{j}}}\right)}\bc_{j,k}\flr{w_{i_{j}}}\right)\right)\\
        \cdot\prod_{j=1}^{t}\prod_{l=1}^{\tau_{j}}\left(B_{1}\left(\ba'_{j,l}\flr{w_{i_{j}}}\flr{\left(w_{h_{j,l}}\flr{w_{i_{j}}}\right)}\right)B_{2}\left(\fll{\left(w_{h_{j,l}}\flr{w_{i_{j}}}\right)}\bc'_{j,l}\flr{w_{i_{j}}}\right)\right).
        \end{multline}
    Deleting redundant flexion markers, we obtain 
        \begin{multline}
            \gaxit(\gaxi(\cA,\cB))(M)(\bw)
            =\sum_{t=1}^{\infty}\sum_{\{(\bA_{j};w_{i_{j}};\bC_{j})\}_{j=1}^{t}\in\Sigma_{t}(\bw)}M\left(\prod_{j=1}^{t}\ful{\bA_{j}}w_{i_{j}}\fur{\bC_{j}}\right)\\
            \cdot \sum_{\substack{\bA_{j}=\ba_{j}\balpha_{j}\\ \bC_{j}=\bgamma_{j}\bc_{j}}}\sum_{\substack{\sigma_{1},\ldots,\sigma_{t}>0\\ \tau_{1},\ldots,\tau_{t}>0}}\sum_{\substack{\{(\ba_{j,k};w_{g_{j,k}};\bc_{j,k})\}_{k=1}^{\sigma_{j}}\in\Sigma_{\sigma_{j}}(\ba_{j})\\ 1\le j\le t}}\sum_{\substack{\{(\ba'_{j,l};w_{h_{j,l}};\bc'_{j,l})\}_{l=1}^{\tau_{j}}\in\Sigma_{\tau_{j}}(\bc_{j})\\ 1\le j\le t}}\prod_{j=1}^{t}\left(B_{1}(\balpha_{j}\flr{w_{i_{j}}})B_{2}(\fll{w_{i_{j}}}\bgamma_{j})\right)\\
        \cdot\prod_{j=1}^{t}\left(A_{1}\left(\left(\prod_{k=1}^{\sigma_{j}}\ful{\ba_{j,k}}w_{g_{j,k}}\fur{\bc_{j,k}}\right)\flr{w_{i_{j}}}\right)A_{2}\left(\fll{w_{i_{j}}}\left(\prod_{l=1}^{\tau_{j}}\ful{\ba'_{j,l}}w_{h_{j,l}}\fur{\bc'_{j,l}}\right)\right)\right)\\
        \cdot\prod_{j=1}^{t}\left(\prod_{k=1}^{\sigma_{j}}\left(B_{1}(\ba_{j,k}\flr{w_{g_{j,k}}})B_{2}(\fll{w_{g_{j,k}}}\bc_{j,k})\right)\cdot \prod_{l=1}^{\tau_{j}}\left(B_{1}(\ba'_{j,l}\flr{w_{h_{j,l}}})B_{2}(\fll{w_{h_{j,l}}}\bc'_{j,l})\right)\right).
        \end{multline}
        Then, by renaming dummy variables as
        \begin{align}
            \ba_{j,\sigma_{j}+l+1}&\coloneqq\begin{cases}\balpha_{j} & \text{if }l=0,\\ \ba'_{j,l} & \text{if }1\le l\le\tau_{j},\end{cases}\\
            w_{g_{j,\sigma_{j}+l+1}}&\coloneqq \begin{cases}w_{i_{j}} & \text{if }l=0,\\ w_{h_{j,l}} & \text{if }1\le l\le\tau_{j},\end{cases}\\
            \bc_{j,\sigma_{j}+l+1}&\coloneqq\begin{cases}\bgamma_{j} & \text{if }l=0,\\ \bc'_{j,l} & \text{if }1\le l\le\tau_{j},\end{cases}\\
        \end{align}
        the range of the above long sums becomes
            \begin{multline}
                \sum_{t=1}^{\infty}\sum_{\{(\bA_{j};w_{i_{j}};\bC_{j})\}_{j=1}^{t}\in\Sigma_{t}(\bw)}\sum_{\substack{\bA_{j}=\ba_{j}\balpha_{j}\\ \bC_{j}=\bgamma_{j}\bc_{j}}}\sum_{\substack{\sigma_{1},\ldots,\sigma_{t}>0\\ \tau_{1},\ldots,\tau_{t}>0}}\sum_{\substack{\{(\ba_{j,k};w_{g_{k}};\bc_{j,k})\}_{k=1}^{\sigma_{j}}\in\Sigma_{\sigma_{j}}(\ba_{j})\\ 1\le j\le t}}\sum_{\substack{\{(\ba'_{j,l};w_{h_{l}};\bc'_{j,l})\}_{l=1}^{\tau_{j}}\in\Sigma_{\tau_{j}}(\bc_{j})\\ 1\le j\le t}}\\
                =\sum_{t=1}^{\infty}\sum_{\substack{\sigma_{1},\ldots,\sigma_{t}>0\\ \tau_{1},\ldots,\tau_{t}>0}}\sum_{\substack{\{(\ba_{j,k};w_{g_{j,k}};\bc_{j,k})\}_{k=1}^{\sigma_{j}+\tau_{j}+1}\in\Sigma_{\sigma_{j}+\tau_{j}+1}(\bw)\\ 1\le j\le t}}.
            \end{multline}
            Therefore we have
            \begin{multline}
                \gaxit(\gaxi(\cA,\cB))(M)(\bw)\\
                =\sum_{t=1}^{\infty}\sum_{\substack{\sigma_{1},\ldots,\sigma_{t}>0\\ \tau_{1},\ldots,\tau_{t}>0}}\sum_{\substack{\{(\ba_{j,k};w_{g_{j,k}};\bc_{j,k})\}_{k=1}^{\sigma_{j}+\tau_{j}+1}\in\Sigma_{\sigma_{j}+\tau_{j}+1}(\bw)\\ 1\le j\le t}} \prod_{j=1}^{t}\prod_{k=1}^{\sigma_{j}+\tau_{j}+1}\left(B_{1}(\ba_{j,k}\flr{w_{g_{j,k}}})B_{2}(\fll{w_{g_{j,k}}}\bc_{j,k})\right)\\
                \cdot M\left(\prod_{j=1}^{t}\ful{\left\{\left(\prod_{k=1}^{\sigma_{j}}\ba_{j,k}w_{g_{j,k}}\bc_{j,k}\right)\ba_{j,\sigma_{j}+1}\right\}}w_{g_{j,\sigma_{j}+1}}\fur{\left\{\bc_{j,\sigma_{j}+1}\prod_{k=\sigma_{j}+2}^{\sigma_{j}+\tau_{j}+1}\ba_{j,k}w_{g_{j,k}}\bc_{j,k}\right\}}\right)\\
                \cdot\prod_{j=1}^{t}A_{1}\left(\left(\prod_{k=1}^{\sigma_{j}}\ful{\ba_{j,k}}w_{g_{j,k}}\fur{\bc_{j,k}}\right)\flr{w_{g_{j,\sigma_{j}+1}}}\right)\\
                \cdot\prod_{j=1}^{t}A_{2}\left(\fll{w_{g_{j,\sigma_{j}+1}}}\left(\prod_{k=\sigma_{j}+2}^{\sigma_{j}+\tau_{j}+1}\ful{\ba_{j,k}}w_{g_{j,k}}\fur{\bc_{j,k}}\right)\right).
            \end{multline}
            Since the value of $M$ appearing above is equal to
            \begin{multline}
                M\left(\prod_{j=1}^{t}\ful{\left\{\left(\prod_{k=1}^{\sigma_{j}}\ba_{j,k}w_{g_{j,k}}\bc_{j,k}\right)\ba_{j,\sigma_{j}+1}\right\}}w_{g_{j,\sigma_{j}+1}}\fur{\left\{\bc_{j,\sigma_{j}+1}\prod_{k=\sigma_{j}+2}^{\sigma_{j}+\tau_{j}+1}\ba_{j,k}w_{g_{j,k}}\bc_{j,k}\right\}}\right)\\
                =M\left(\prod_{j=1}^{t}\ful{\left(\prod_{k=1}^{\sigma_{j}}\ful{\ba_{j,k}}w_{g_{j,k}}\fur{\bc_{j,k}}\right)}\ful{\ba_{j,\sigma_{j}+1}}w_{g_{j,\sigma_{j}+1}}\fur{\bc_{j,\sigma_{j}+1}}\fur{\left(\prod_{k=\sigma_{j}+2}^{\sigma_{j}+\tau_{j}+1}\ful{\ba_{j,k}}w_{g_{j,k}}\fur{\bc_{j,k}}\right)}\right),
            \end{multline}
            and we can add redundant markers as
            \begin{align}
                A_{1}\left(\left(\prod_{k=1}^{\sigma_{j}}\ful{\ba_{j,k}}w_{g_{j,k}}\fur{\bc_{j,k}}\right)\flr{w_{g_{j,\sigma_{j}+1}}}\right)&=A_{1}\left(\left(\prod_{k=1}^{\sigma_{j}}\ful{\ba_{j,k}}w_{g_{j,k}}\fur{\bc_{j,k}}\right)\flr{\left(\ful{\ba_{j,\sigma_{j}+1}}w_{g_{j,\sigma_{j}+1}}\fur{\bc_{j,\sigma_{j}+1}}\right)}\right),\\
                A_{2}\left(\fll{w_{g_{j,\sigma_{j}+1}}}\left(\prod_{k=\sigma_{j}+2}^{\sigma_{j}+\tau_{j}+1}\ful{\ba_{j,k}}w_{g_{j,k}}\fur{\bc_{j,k}}\right)\right)&=A_{2}\left(\fll{\left(\ful{\ba_{j,\sigma_{j}+1}}w_{g_{j,\sigma_{j}+1}}\fur{\bc_{j,\sigma_{j}+1}}\right)}\left(\prod_{k=\sigma_{j}+2}^{\sigma_{j}+\tau_{j}+1}\ful{\ba_{j,k}}w_{g_{j,k}}\fur{\bc_{j,k}}\right)\right),
            \end{align}
            we compute the factors depending on $A_{1},A_{2}$ and $M$ as
            \begin{align}
                &\gaxit(\gaxi(\cA,\cB))(M)(\bw)\\
                &=\begin{multlined}[t]\sum_{t=1}^{\infty}\sum_{\substack{\sigma_{1},\ldots,\sigma_{t}>0\\ \tau_{1},\ldots,\tau_{t}>0}}\sum_{\substack{\{(\ba_{j,k};w_{g_{j,k}};\bc_{j,k})\}_{k=1}^{\sigma_{j}+\tau_{j}+1}\in\Sigma_{\sigma_{j}+\tau_{j}+1}(\bw)\\ 1\le j\le t}} \prod_{j=1}^{t}\prod_{k=1}^{\sigma_{j}+\tau_{j}+1}\left(B_{1}(\ba_{j,k}\flr{w_{g_{j,k}}})B_{2}(\fll{w_{g_{j,k}}}\bc_{j,k})\right)\\
                \cdot M\left(\prod_{j=1}^{t}\ful{\left(\prod_{k=1}^{\sigma_{j}}\ful{\ba_{j,k}}w_{g_{j,k}}\fur{\bc_{j,k}}\right)}\ful{\ba_{j,\sigma_{j}+1}}w_{g_{j,\sigma_{j}+1}}\fur{\bc_{j,\sigma_{j}+1}}\fur{\left(\prod_{k=\sigma_{j}+2}^{\sigma_{j}+\tau_{j}+1}\ful{\ba_{j,k}}w_{g_{j,k}}\fur{\bc_{j,k}}\right)}\right)\\
                \cdot\prod_{j=1}^{t}A_{1}\left(\left(\prod_{k=1}^{\sigma_{j}}\ful{\ba_{j,k}}w_{g_{j,k}}\fur{\bc_{j,k}}\right)\flr{\left(\ful{\ba_{j,\sigma_{j}+1}}w_{g_{j,\sigma_{j}+1}}\fur{\bc_{j,\sigma_{j}+1}}\right)}\right)\\
                \cdot\prod_{j=1}^{t}A_{2}\left(\fll{\left(\ful{\ba_{j,\sigma_{j}+1}}w_{g_{j,\sigma_{j}+1}}\fur{\bc_{j,\sigma_{j}+1}}\right)}\left(\prod_{k=\sigma_{j}+2}^{\sigma_{j}+\tau_{j}+1}\ful{\ba_{j,k}}w_{g_{j,k}}\fur{\bc_{j,k}}\right)\right)
                \end{multlined}\\
                &=\sum_{n=1}^{\infty}\sum_{\{(\ba_{j};w_{i_{j}};\bc_{j})\}_{j=1}^{n}\in\Sigma_{n}(\bw)} \prod_{j=1}^{n}\left(B_{1}(\ba_{n}\flr{w_{i_{n}}})B_{2}(\fll{w_{i_{n}}}\bc_{n})\right)\gaxit(A_{1},A_{2})(M)\left(\prod_{j=1}^{n}\ful{\ba_{j}}w_{i_{j}}\fur{\bc_{j}}\right)\\
                    &=\gaxit(B_{1},B_{2})(\gaxit(A_{1},A_{2})(M))(\bw),
            \end{align}
            with the help of \eqref{eq:gaxit_simple}.
\end{proof}

\begin{corollary}[{\cite[(4)]{ecalle15}}]\label{cor:gaxit_separation}
    Let $\cA=(A_{1},A_{2})$ be an element of $\MU(R)\times\MU(R)$. Then we have the following two expressions of $\gaxit(\cA)$:
    \begin{align}
        \gaxit(\cA)
        &=\gamit(A_{1})\circ\ganit(\gamit(A_{1})^{-1}(A_{2}))\\
        &=\ganit(A_{2})\circ\gamit(\ganit(A_{2})^{-1}(A_{1})).
    \end{align}
\end{corollary}
\begin{proof}
    For the first expression, using Proposition \ref{prop:gaxit_associative}, we obtain
    \begin{align}
        &\gamit(A_{1})\circ\ganit(\gamit(A_{1})^{-1}(A_{2}))\\
        &=\gaxit(A_{1},1)\circ\gaxit(1,\gamit(A_{1})^{-1}(A_{2}))\\
        &=\gaxit(\gaxi((1,\gamit(A_{1})^{-1}(A_{2})),(A_{1},1)))\\
        &=\gaxit(\mmu(\gaxit(A_{1},1)(1),A_{1}),\mmu(1,\gaxit(A_{1},1)(\gamit(A_{1})^{-1}(A_{2}))))\\
        &=\gaxit(A_{1},A_{2}).
    \end{align}
    We can also show the latter equality in a similar way.
\end{proof}

\begin{proposition}\label{prop:gaxit_mu}
    The operator $\gaxit(A,A')\colon\BIMU(R)\to\BIMU(R)$ gives an $R$-algebra homomorphism for any $A,A'\in\MU(R)$.
\end{proposition}
\begin{proof}
    The $A=1$ case (that is, the same property as what we want for $\ganit$) is proven by Komiyama \cite[Proposition 3.5]{komiyama21}.
    Since the equality $\gamit(X)=\anti\circ\ganit(\anti(X))\circ\anti$, which is valid for any $X\in\MU(R)$, is checked by definition, we can show that $\gamit(X)$ always gives an algebra homomorphism as
    \begin{align}
        \gamit(X)(\mmu(A,B))
        &=(\anti\circ\ganit(\anti(X))\circ\anti)(\mmu(A,B))\\
        &=(\anti\circ\ganit(\anti(X)))(\mmu(\anti(B),\anti(A)))\\
        &=\anti(\mmu((\ganit(\anti(X))\circ\anti)(B),(\ganit(\anti(X))\circ\anti)(A)))\\
        &=\mmu((\anti\circ\ganit(\anti(X))\circ\anti)(A),(\anti\circ\ganit(\anti(X))\circ\anti)(B))\\
        &=\mmu(\gamit(X)(A),\gamit(X)(B)),
    \end{align} 
    where $A$ and $B$ belong to $\BIMU(R)$.
    Thus we get the desired result for $\gaxit$ as we now have it for $\gamit$ and $\ganit$, due to Corollary \ref{cor:gaxit_separation}.
\end{proof}

\begin{proposition}\label{prop:gaxi_group}
    The set $\MU(R)\times\MU(R)$ becomes a group with the operation $\gaxi$.
\end{proposition}
\begin{proof}
    The unit is given by $(1,1)$.
    Also we can check that $\gaxit(\cA)$ is always invertible for any $\cA\in\MU(R)\times\MU(R)$, by constructing the pair $\cI$ satisfying $\gaxit(\cI)\circ\gaxit(\cA)=\id$ length by length.
    We now give a proof of the associativity: what we want is the equality
    \[\gaxi(\gaxi(\cA,\cB),\cC)=\gaxi(\cA,\gaxi(\cB,\cC))\]
    which is valid for arbitrarily chosen elements $\cA=(A_{1},A_{2})$, $\cB=(B_{1},B_{2})$ and $\cC=(C_{1},C_{2})$ of $\MU(R)\times\MU(R)$.
    By writing $\cA=(A_{1},A_{2})$, the right-hand side is deformed
    \begin{align}
        &\gaxi(\cA,\gaxi(\cB,\cC))\\
        &=\gaxi(\cA,(\mmu(\gaxit(\cC)(B_{1}),C_{1}), \mmu(C_{2},\gaxit(\cC)(B_{2}))))\\
        &=\begin{multlined}[t]\left(\mmu(\gaxit(\gaxi(\cB,\cC))(A_{1}),\mmu(\gaxit(\cC)(B_{1}),C_{1})),\right.\\\left.\mmu(\mmu(C_{2},\gaxit(\cC)(B_{2})),\gaxit(\gaxi(\cB,\cC)(A_{2})))\right)\end{multlined}\\
        &=(\mmu(\gaxit(\gaxi(\cB,\cC))(A_{1}),\gaxit(\cC)(B_{1}),C_{1}),\mmu(C_{2},\gaxit(\cC)(B_{2}),\gaxit(\gaxi(\cB,\cC)(A_{2}))))\\
        &=(\mmu((\gaxit(\cC)\circ\gaxit(\cB))(A_{1}),\gaxit(\cC)(B_{1}),C_{1}),\mmu(C_{2},\gaxit(\cC)(B_{2}),(\gaxit(\cC)\circ\gaxit(\cB))(A_{2})))
    \end{align}  
    with help of Proposition \ref{prop:gaxit_associative}.
    Since $\gaxit(\cS)$ is always an algebra homomorphism for any $\cS\in\MU(R)\times\MU(R)$, we can continue the computation as
    \begin{align}
        &\gaxi(\cA,\gaxi(\cB,\cC))\\
        &=(\mmu(\gaxit(\cC)(\mmu(\gaxit(\cB)(A_{1}),B_{1})),C_{1}),\mmu(C_{2},\gaxit(\cC)(\mmu(B_{2},\gaxit(\cB)(A_{2})))))\\
        &=\gaxi((\mmu(\gaxit(\cB)(A_{1}),B_{1}),\mmu(B_{2},\gaxit(\cB)(A_{2}))),\cC)\\
        &=\gaxi(\gaxi(\cA,\cB),\cC).
    \end{align}
\end{proof}

Similarly to the definition of $\arit$, we define the operator $\garit(A)\colon\BIMU(R)\to\BIMU(R)$ by
\[\garit(A)\coloneqq\gaxit(A,\invmu(A))\]
for $A\in\MU(R)$.
As variants, we also define
\[\girat(A)\coloneqq\gaxit(A,(\push\circ\swap\circ\invmu\circ\swap)(A))\]
and
\[\giwat(A)\coloneqq\gaxit(A,\anti(A)).\]

\begin{definition}
    We define a map $\gari\colon\BIMU(R)\times\MU(R)\to\BIMU(R)$ by
    \[\gari(A,B)\coloneqq\mmu(\garit(B)(A),B).\]
    Similarly, with the same assumption, we define
    \[\gami(A,B)\coloneqq\mmu(\gamit(B)(A),B)\]
    and
    \[\gani(A,B)\coloneqq\mmu(B,\ganit(B)(A)).\]
    Note that these laws
    \begin{enumerate}
        \item have the bimould $1$ as their unit,
        \item are $R$-linear as a map on $\BIMU(R)$ of the first component, and
        \item have values in $\MU(R)$ whenever the first component is in $\MU(R)$, thus they define different product rules on $\MU(R)$.
    \end{enumerate}
\end{definition}

\begin{corollary}\label{prop:gari_group}
    The pairs $\GAMI(R)\coloneqq(\MU(R),\gami)$, $\GANI(R)\coloneqq(\MU(R),\gani)$ and $\GARI(R)\coloneqq(\MU(R),\gari)$ are groups.
\end{corollary}
\begin{proof}
    When elements $A$ and $B$ of $\MU(R)$ are given, we can always regard the above structures as specializations of $\gaxi$ in a way
    \begin{align}
        \gami(A,B)&=(\text{the first component of }\gaxi((A,1),(B,1))),\\
        \gani(A,B)&=(\text{the second component of }\gaxi((1,A),(1,B))),
    \end{align}
    and
    \[\gari(A,B)=(\text{the first component of }\gaxi((A,\invmu(A)),(B,\invmu(B)))).\]
    Thus we see that $\GAMI$, $\GANI$ and $\GARI$ become groups from Proposition \ref{prop:gaxi_group}.
\end{proof}

The map on $\MU(R)$ giving an inverse about $\gami$ (resp.~$\gani$, $\gari$) whose existence is guaranteed by Proposition \ref{prop:gari_group} is denoted by $\invgami$ (resp.~$\invgani$, $\invgari$).

\begin{remark}\label{rem:linearization}
    Considering the dual number ring $R[\ep]\coloneqq R\jump{\ep}/(\ep^{2})$ and bimoulds there, we see that $\arit$ and $\preari$ are linearization of $\garit$ and $\gari$, respectively: indeed, by definition we have
    \[\garit(1+\ep B)(A)=A+\ep\arit(B)(A)\]
    so as an operator $\garit(1+\ep B)=\id+\ep\arit(B)$, and thus
    \[\gari(A,1+\ep B)=A+\ep\preari(A,B)\]
    for any $A\in\BIMU(R)$ and $B\in\LU(R)$.
    Similarly, we have $\amit,\anit,\irat$ and $\iwat$ as the linearized objects of $\gamit$, $\ganit$, $\girat$ and $\giwat$, respectively.
\end{remark}

Hereafter, we usually use the Lie-exponential $\expari\colon\ARI\to\GARI$ (see Proposition \ref{prop:expari} in Section \ref{sec:appendix}).

\begin{proposition}\label{prop:expari_alternal}
    The pair $\ARI_{\al}(R)\coloneqq(\LU_{\al}(R),\ari)$ (resp.~$\GARI_{\as}(R)\coloneqq(\MU_{\as}(R),\gari)$) is a Lie subalgebra of $\ARI(R)$ (resp.~subgroup of $\GARI(R)$). Moreover, the restriction of $\expari$ gives the Lie-exponential $\ARI_{\al}\to\GARI_{\as}$.
\end{proposition}
\begin{proof}
    We can find a proof in \cite[\S A.2]{fk23} for the assertion about $\ARI_{\al}(R)$.
    The fact \[\expari(\ARI_{\al}(R))=\GARI_{\as}(R)\] is proven in \cite[Theorem A.7]{komiyama21}.
    Then we can say that $\GARI_{\as}(R)$ is a subgroup of $\GARI$ since $\expari$ gives the Lie-exponential $\ARI\to\GARI$.
\end{proof}

\begin{lemma}\label{lem:neg_pari_gaxit}
    Like the case for $\axit$, applying either $\neg$ or $\pari$ inside $\gaxit$ amounts to taking the conjugation by them. Namely, we have
    \[\gaxit(h(A),h(A'))=h\circ\gaxit(A,A')\circ h\]
    for every $A,A'\in\MU(R)$ and $h\in\{\neg,\pari\}$.
    Moreover, the operator $\gantar\coloneqq\invmu\circ\anti\circ\pari\colon\GARI(R)\to\GARI(R)$ is a group homomorphism.
\end{lemma}
\begin{proof}
    The former assertion is easy.
    For the fact that $\gantar$ preserves $\gari$, using the formula $\anti\circ\gamit(\anti(X))\circ\anti=\ganit(X)$ ($X\in\MU(R)$), we have for any $X\in\MU(R)$
    \begin{align}
        \anti\circ\garit(X)
        &=\anti\circ\gamit(X)\circ\ganit((\gamit(X)^{-1}\circ\invmu)(X))\\
        &=\ganit(\anti(X))\circ\anti\circ\ganit((\gamit(X)^{-1}\circ\invmu)(X))\\
        &=\ganit(\anti(X))\circ\gamit((\anti\circ\gamit(X)^{-1}\circ\invmu)(X))\circ\anti\\
        &=\ganit(\anti(X))\circ\gamit((\ganit(\anti(X))^{-1}\circ\anti\circ\invmu)(X))\circ\anti\\
        &=\garit((\invmu\circ\anti)(X))\circ\anti,
    \end{align}
    where we used the commutativity between $\anti$ and $\invmu$ (which follows from Lemma \ref{lem:operators}) in the last equality.
    Therefore, with the fact that $\gantar$ preserves $\mmu$, we obtain
    \begin{align}
        \gantar(\gari(A,B))
        &=\gantar(\mmu(\garit(B)(A),B))\\
        &=\mmu((\gantar\circ\garit(B))(A),\gantar(B))\\
        &=\mmu((\invmu\circ\pari\circ\garit((\invmu\circ\anti)(B))\circ\anti)(A),\gantar(B))\\
        &=\mmu((\invmu\circ\garit((\pari\circ\invmu\circ\anti)(B))\circ\pari\circ\anti)(A),\gantar(B))\\
        &=\mmu((\garit((\pari\circ\invmu\circ\anti)(B))\circ\invmu\circ\pari\circ\anti)(A),\gantar(B))\\
        &=\mmu((\garit(\gantar(B))\circ\gantar)(A),\gantar(B))\\
        &=\gari(\gantar(A),\gantar(B)).
    \end{align}
    Note that, the commutativity of $\invmu$ and $\garit(\bullet)$, which we used in the fifth equality above, follows from Proposition \ref{prop:gaxit_mu}.
\end{proof}

We conclude this section by introducing \emph{\'{E}calle's fundamental identity}.
Let $\fragari$ be the operator on $\BIMU(R)\times\MU(R)$ defined by 
\[\fragari(A,B)\coloneqq\gari(A,\invgari(B))\]
and define its $\swap$-conjugation $\fragira$ by
\[\fragira(A,B)\coloneqq\swap(\fragari(\swap(A),\swap(B))).\]

\begin{theorem}[\'{E}calle's fundamental identity; {\cite[(2.64)]{ecalle11}}]\label{thm:fundamental}
    Let $A$ be a bimould in $\BIMU(R)$ and $B$ in $\MU(R)$. Then we have
    \begin{equation}\label{eq:fundamental}
        \fragira(A,B)=\ganit(\crash(B))(\fragari(A,B))
    \end{equation}
    where the operator $\crash$ is defined as
    \[B\mapsto\mmu((\push\circ\swap\circ\invmu\circ\invgari\circ\swap)(B),(\swap\circ\invgari\circ\swap)(B)).\]
\end{theorem}
\begin{proof}
    It is given in \cite[\S 2.9]{schneps15}.
\end{proof}

\begin{remark}
    There is an equivalent formulation of the fundamental identity \eqref{eq:fundamental} excluding the \emph{fraction} along the $\gari$-product: let $\gira$ be the $\swap$-conjugation of $\gari$:
    \[(A,B)\mapsto\swap(\gari(\swap(A),\swap(B))).\]
    Then we have its interpretation as a restriction of $\gaxi$:
    \[\gira(A,B)=\mmu(\girat(B)(A),B)=\mmu(\gaxit(B,(\push\circ\swap\circ\invmu\circ\swap)(B))(A),B),\]
    and the identity 
    \[\gira(A,B)=\ganit(\rash(B))(\gari(A,\ras(B))),\]
    which is valid for any $A\in\BIMU(R)$ and $B\in\MU(R)$, where the operators $\ras$ and $\rash$ are defined by
    \[\ras(B)\coloneqq(\invgari\circ\swap\circ\invgari\circ\swap)(B)\]
    and
    \[\rash(B)\coloneqq\mmu((\push\circ\swap\circ\invmu\circ\swap)(B),B),\]
    respectively. We can find a proof of them in \cite[\S 2.9]{schneps15}.
    Note that, it is clear from the definition that the inverse operation $\invgira$ for the $\gira$-multiplication is given by $\invgira=\swap\circ\invgari\circ\swap$.
\end{remark}

\section{Definition of flexion units and some properties}\label{sec:unit}
\begin{definition}[{\cite[(3.9)]{ecalle11}}]\label{def:unit}
    Let $\fE$ be an element of $\BIMU_{1}(R)$. We say that $\fE$ is a \emph{flexion unit} if it satisfies the following conditions.
    \begin{enumerate}
        \item Parity condition: $\fE\binom{u_{1}}{v_{1}}=-\fE\binom{-u_{1}}{-v_{1}}$.
        \item Tripartite identity:
        \begin{equation}\label{eq:tripartite}
            \fE(w_{1})\fE(w_{2})=\fE(w_{1}\flr{w_{2}})\fE(\ful{w_{1}}w_{2})+\fE(w_{1}\fur{w_{2}})\fE(\fll{w_{1}}w_{2})
        \end{equation} 
    \end{enumerate}
\end{definition}

Recall the definition of $\push$:
\[\push(M)\binom{u_{1},\ldots,u_{r}}{v_{1},\ldots,v_{r}}\coloneqq M\binom{-u_{1}-\cdots-u_{r},u_{1},\ldots,u_{r-1}}{-v_{r},v_{1}-v_{r},\ldots,v_{r-1}-v_{r}}.\]
We say that a bimould $M$ is \emph{$\push$-neutral} if for each $r\ge 0$ the identity
\[(\id+\push+\push\circ\push+\cdots+\push^{\circ r})(M)\binom{u_{1},\ldots,u_{r}}{v_{1},\ldots,v_{r}}=0\]
holds.

\begin{proposition}[{\cite[(3.11)]{ecalle11}}]\label{prop:unit_push}
    Let $\fE\in\BIMU_{1}(R)$. Then the following are equivalent:
    \begin{enumerate}
        \item $\fE$ is a flexion unit.
        \item $\fE$ and $\mmu(\fE,\fE)$ are $\push$-neutral.
        \item $\mmu(\underbrace{\fE,\ldots,\fE}_{n})$ is $\push$-neutral for every $n\ge 1$.
    \end{enumerate}
\end{proposition}
\begin{proof}
    Since $\fE\in\BIMU_{1}(R)$, the condition that $\fE$ is $\push$-neutral is equivalent to
    \[\fE\binom{u_{1}}{v_{1}}+\push(\fE)\binom{u_{1}}{v_{1}}=0,\]
    which is rewritten as
    \[\fE\binom{u_{1}}{v_{1}}+\fE\binom{-u_{1}}{-v_{1}}=0,\]
    by definition of $\push$. This is equivalent to the parity condition $\fE\binom{u_{1}}{v_{1}}=-\fE\binom{-u_{1}}{-v_{1}}$.\\

    Since $\mmu^{n}(\fE)\coloneqq\mmu(\underbrace{\fE,\ldots,\fE}_{n})$ is explicitly expressed as
    \[\mmu^{n}(\fE)(w_{1},\ldots,w_{r})=\begin{cases}
        \fE(w_{1})\cdots\fE(w_{r}) & \text{if }r=n,\\
        0 & \text{otherwise,}
    \end{cases}\]
    its $\push$-neutrality becomes that
    \begin{align}
        0&=\sum_{i=0}^{n}\push^{\circ i}(\mmu^{n}(\fE))\binom{u_{1},\ldots,u_{n}}{v_{1},\ldots,v_{n}}\\
        &=\begin{multlined}[t]\mmu^{n}(\fE)\binom{u_{1},\ldots,u_{n}}{v_{1},\ldots,v_{n}}\\+\sum_{i=1}^{n}\mmu^{n}(\fE)\binom{u_{n-i+2},\ldots,u_{n},-u_{1}-\cdots-u_{n},u_{1},\ldots,u_{n-i}}{v_{n-i+2}-v_{n-i+1},\ldots,v_{n}-v_{n-i+1},-v_{n-i+1},v_{1}-v_{n-i+1},\ldots,v_{n-i}-v_{n-i+1}}\end{multlined}\\
        &=\begin{multlined}[t]\fE\binom{u_{1}}{v_{1}}\cdots\fE\binom{u_{n}}{v_{n}}\\
        +\sum_{i=1}^{n}\fE\binom{u_{1}}{v_{1}-v_{n-i+1}}\cdots\fE\binom{u_{n-i}}{v_{n-i}-v_{n-i+1}}\\
        \cdot\fE\binom{-u_{1}-\cdots-u_{n}}{-v_{n-i+1}}\fE\binom{u_{n-i+2}}{v_{n-i+2}-v_{n-i+1}}\cdots\fE\binom{u_{n}}{v_{n}-v_{n-i+1}}\end{multlined}\\
        &=\begin{multlined}[t]\fE\binom{u_{1}}{v_{1}}\cdots\fE\binom{u_{n}}{v_{n}}\\
            +\sum_{i=1}^{n}\fE\binom{u_{1}}{v_{1}-v_{i}}\cdots\fE\binom{u_{i-1}}{v_{i-1}-v_{i}}\fE\binom{-u_{1}-\cdots-u_{n}}{-v_{i}}\fE\binom{u_{i+1}}{v_{i+1}-v_{i}}\cdots\fE\binom{u_{n}}{v_{n}-v_{i}}\end{multlined}\\
        &=\begin{multlined}[t]\fE\binom{u_{1}}{v_{1}}\cdots\fE\binom{u_{n}}{v_{n}}\\-\sum_{i=1}^{n}\fE(w_{1}\flr{w_{i}})\cdots\fE(w_{i-1}\flr{w_{i}})\fE(\ful{w_{1}\cdots w_{i-1}}w_{i}\fur{w_{i+1}\cdots w_{n}})\fE(\fll{w_{i}}w_{i+1})\cdots\fE(\fll{w_{i}}w_{n}),\end{multlined}
    \end{align}
    where we regard $v_{n+1}=0$ and put $w_{i}=\binom{u_{i}}{v_{i}}$. In the last equality we have used the parity condition of $\fE$. In particular, when $n=2$ the $\push$-neutrality amounts to say that
    \[\fE(w_{1})\fE(w_{2})=\fE(w_{1}\flr{w_{2}})\fE(\ful{w_{1}}w_{2})+\fE(w_{1}\fur{w_{2}})\fE(\fll{w_{1}}w_{2}).\]
    holds under the assumption of the parity condition. This is nothing but the tripartite identity. Thus we complete the proof of (1) $\Leftrightarrow$ (2).\\

    Let us introduce the new bimould $\fez\coloneqq\invmu(1-\fE)$ and put
    \[F_{n}(w_{1},\ldots,w_{n})\coloneqq \sum_{i=1}^{n}\fez((w_{1}\cdots w_{i-1})\flr{w_{i}})\fE(\ful{w_{1}\cdots w_{i-1}}w_{i}\fur{w_{i+1}\cdots w_{n}})\fez(\fll{w_{i}}(w_{i+1}\cdots w_{n})).\]
    Note that $\fez(w_{1},\ldots,w_{r})=\mmu^{r}(\fE)(w_{1},\ldots,w_{r})=\fE(w_{1})\cdots \fE(w_{r})$ for every $r\ge 0$.
    It suffices to prove that the tripartite identity implies $F_{n}(w_{1},\ldots,w_{n})=\fE(w_{1})\cdots\fE(w_{n})$ (for all $n\ge 2$). We do it by induction on $n$.
    By definition we get
    \begin{multline}
        F_{n+1}(w_{1},\ldots,w_{n+1})\\
        =\sum_{i=1}^{n}\fez((w_{1}\cdots w_{i-1})\flr{w_{i}})\fE(\ful{w_{1}\cdots w_{i-1}}w_{i}\fur{w_{i+1}\cdots w_{n+1}})\fez(\fll{w_{i}}(w_{i+1}\cdots w_{n+1}))\\
        +\fez((w_{1}\cdots w_{n})\flr{w_{n+1}})\fE(\ful{w_{1}\cdots w_{n}}w_{n+1}).
    \end{multline}
    By putting $A_{i}\coloneqq \ful{w_{1}\cdots w_{i-1}}w_{i}\fur{w_{i+1}\cdots w_{n}}$, the first sum on the right-hand side becomes
    \begin{multline}
        \sum_{i=1}^{n}\fez((w_{1}\cdots w_{i-1})\flr{w_{i}})\fE(\ful{w_{1}\cdots w_{i-1}}w_{i}\fur{w_{i+1}\cdots w_{n+1}})\fez(\fll{w_{i}}(w_{i+1}\cdots w_{n+1}))\\
        =\sum_{i=1}^{n}\fez((w_{1}\cdots w_{i-1})\flr{w_{i}})\fez(\fll{w_{i}}(w_{i+1}\cdots w_{n}))\fE(A_{i}\fur{w_{n+1}})\fE(\fll{A_{i}}w_{n+1}).
    \end{multline}
    Applying the tripartite identity to $\fE(A_{i}\fur{w_{n+1}})\fE(\fll{A_{i}}w_{n+1})$, we obtain
    \begin{align}
        &\sum_{i=1}^{n}\fez((w_{1}\cdots w_{i-1})\flr{w_{i}})\fE(\ful{w_{1}\cdots w_{i-1}}w_{i}\fur{w_{i+1}\cdots w_{n+1}})\fez(\fll{w_{i}}(w_{i+1}\cdots w_{n+1}))\\
        &=\sum_{i=1}^{n}\fez((w_{1}\cdots w_{i-1})\flr{w_{i}})\fez(\fll{w_{i}}(w_{i+1}\cdots w_{n}))\left(\fE(A_{i})\fE(w_{n+1})-\fE(\ful{A_{i}}w_{n+1})\fE(A_{i}\flr{w_{n+1}})\right)\\
        &=\begin{multlined}[t]\sum_{i=1}^{n}\fez((w_{1}\cdots w_{i-1})\flr{w_{i}})\fez(\fll{w_{i}}(w_{i+1}\cdots w_{n}))\\\cdot\left(\fE(\ful{w_{1}\cdots w_{i-1}}w_{i}\fur{w_{i+1}\cdots w_{n}})\fE(w_{n+1})-\fE(\ful{w_{1}\cdots w_{n}}w_{n+1})\fE(A_{i}\flr{w_{n+1}})\right)\end{multlined}\\
        &=\begin{multlined}[t]F_{n}(w_{1},\ldots,w_{n})\fE(w_{n+1})\\-\sum_{i=1}^{n}\fez((w_{1}\cdots w_{i-1})\flr{w_{i}})\fez(\fll{w_{i}}(w_{i+1}\cdots w_{n}))\fE(\ful{w_{1}\cdots w_{n}}w_{n+1})\fE(A_{i}\flr{w_{n+1}}).\end{multlined}
    \end{align}
    By induction hypothesis, it is sufficient to show 
    \begin{multline}\fez((w_{1}\cdots w_{n})\flr{w_{n+1}})\fE(\ful{w_{1}\cdots w_{n}}w_{n+1})\\=\sum_{i=1}^{n}\fez((w_{1}\cdots w_{i-1})\flr{w_{i}})\fez(\fll{w_{i}}(w_{i+1}\cdots w_{n}))\fE(\ful{w_{1}\cdots w_{n}}w_{n+1})\fE(A_{i}\flr{w_{n+1}}).\end{multline}
    We then put $w'_{i}\coloneqq w_{i}\flr{w_{n+1}}$ for $1\le i\le n$. Then $A_{i}\flr{w_{n+1}}=\ful{w'_{1}\cdots w'_{i-1}}w'_{i}\fur{w'_{i+1}\cdots w'_{n}}$ holds and thus the right-hand side becomes
    \begin{align}
        &\sum_{i=1}^{n}\fez((w_{1}\cdots w_{i-1})\flr{w_{i}})\fez(\fll{w_{i}}(w_{i+1}\cdots w_{n}))\fE(\ful{w_{1}\cdots w_{n}}w_{n+1})\fE(A_{i}\flr{w_{n+1}})\\
        &=F_{n}(w'_{1},\ldots,w'_{n})\fE(\ful{w_{1}\cdots w_{n}}w_{n+1})\\
        &=\fez(w'_{1},\ldots,w'_{n})\fE(\ful{w_{1}\cdots w_{n}}w_{n+1})\\
        &=\fez((w_{1}\cdots w_{n})\flr{w_{n+1}})\fE(\ful{w_{1}\cdots w_{n}}w_{n+1}).
    \end{align}
    Note that in the last equality we have used the induction hypothesis. Then we complete the proof.
\end{proof}

\section{Primary bimoulds}\label{sec:primary}
In what follows, we fix a flexion unit $\fE$. We denote by $\fO$ its \emph{conjugate unit} defined by $\fO\coloneqq\swap(\fE)$.

\begin{definition}[Primary bimoulds]
    We define $\fez\coloneqq\invmu(1-\fE)$ and $\fes\coloneqq\expari(\fE)$.
    They have an explicit expression:
    \begin{align}
        \fez(w_{1},\ldots,w_{r})&=\fE(w_{1})\cdots\fE(w_{r}),\\
        \fes\binom{u_{1},\ldots,u_{r}}{v_{1},\ldots,v_{r}}&=\fE\binom{u_{1}}{v_{1}-v_{2}}\cdots\fE\binom{u_{1}+\cdots+u_{r-1}}{v_{r-1}-v_{r}}\fE\binom{u_{1}+\cdots+u_{r}}{v_{r}}.
    \end{align}
    Similarly we put $\foz\coloneqq\invmu(1-\fO)$ and $\fos\coloneqq\expari(\fO)$.
    By definition, the identity $\swap(\fes)=\foz$ and $\swap(\fez)=\fos$ hold.
\end{definition}

\begin{remark}
    By definition, $\fes$ has the following remarkable property: if the sequence $\bw$ is decomposed as $\bw=\ba\bb$, it always holds that
    \begin{equation}\label{eq:es_split}
        \fes(\bw)=\fes(\ba\flr{\bb})\fes(\ful{\ba}\bb).
    \end{equation}
\end{remark}

\begin{proposition}\label{prop:ez_and_es}
    We have the following.
    \begin{enumerate}
        \item $\invgani(\fez)=(\pari\circ\anti)(\fes)$.
        \item $\fes$ is symmetral.
        \item Both $\fez$ and $\fes$ are $\neg\circ\pari$-invariant. In particular, we have $\invmu(\fes)=\push(\fes)$.
    \end{enumerate}
\end{proposition}
\begin{proof}
    \begin{enumerate}
        \item We shall prove $\gani((\pari\circ\anti)(\fes),\fez)=1$.
        Since the components of length $0$ clearly coincide, it is sufficient to show
        \[\ganit(\fez)((\pari\circ\anti)(\fes))(w_{1},\ldots,w_{r})=\begin{cases}-\fE(w_{1}) & \text{if }r=1,\\ 0 & \text{if }r\ge 2.\end{cases}\]
        Let us put\footnote{Our $D_{2s}(\bw)$ is same as \cite{komiyama21}, but our $D_{2s}^{1}(\bw)$ and his object with the same symbol are different.}
        \begin{align}
            D_{2s}(\bw)&\coloneqq\{(\bb_{1};\bc_{1};\cdots;\bb_{s};\bc_{s})\in\fV^{2s}\mid \bb_{1},\ldots,\bb_{s},\bc_{1},\ldots,\bc_{s-1}\neq\emp\},\\
            D_{2s}^{10}(\bw)&\coloneqq\{(\bb_{1};\bc_{1};\cdots;\bb_{s};\bc_{s})\mid\ell(\bb_{s})=1,~\bc_{s}=\emp\}\subseteq D_{2s}(\bw),\\
            D_{2s}^{\ast 0}(\bw)&\coloneqq\{(\bb_{1};\bc_{1};\cdots;\bb_{s};\bc_{s})\mid\ell(\bb_{s})\ge 2,~\bc_{s}=\emp\}\subseteq D_{2s}(\bw),\\
            D_{2s}^{1}(\bw)&\coloneqq\{(\bb_{1};\bc_{1};\cdots;\bb_{s};\bc_{s})\mid\ell(\bc_{s})=1\}\subseteq D_{2s}(\bw),
        \end{align}
        and
        \[D_{2s}^{\ast}(\bw)\coloneqq\{(\bb_{1};\bc_{1};\cdots;\bb_{s};\bc_{s})\mid\ell(\bc_{s})\ge 2\}\subseteq D_{2s}(\bw).\]
        Moreover, for $\bv=(\bb_{1};\bc_{1};\cdots;\bb_{s};\bc_{s})\in D_{2s}(\bw)$ ($s\ge 1$), we define
        \[F(\bv)\coloneqq \anti(\fes)(\bb_{1}\fur{\bc_{1}}\cdots\bb_{s}\fur{\bc_{s}})\prod_{i=1}^{s}\fez(\fll{\bb_{i}}\bc_{i})\]
        and
        \[\sigma(\bv)\coloneqq\sum_{i=1}^{s}\ell(\bb_{i}).\]
        Then we have
        \begin{align}
            \ganit(\fez)((\pari\circ\anti)(\fes))(w_{1},\ldots,w_{r})
            &=\sum_{s=1}^{\infty}\sum_{\bv\in D_{2s}(\bw)}(-1)^{\sigma(\bv)}F(\bv)\\
            &=\sum_{\bullet\in\{10,\ast0,1,\ast\}}\sum_{s=1}^{\infty}\sum_{\bv\in D_{2s}^{\bullet}(\bw)}(-1)^{\sigma(\bv)}F(\bv).
        \end{align}
        It is easy to see that, for every $s\ge 1$, both the maps
        \begin{align}
            \phi_{-}\colon D_{2s}^{10}(\bw)\to D_{2(s-1)}^{\ast}(\bw);&~(\cdots;\bc_{s-1};\bb_{s};\emp)\mapsto (\cdots;\bc_{s-1}\bb_{s}),\\
            \phi_{+}\colon D_{2s}^{1}(\bw)\to D_{2s}^{\ast 0}(\bw);&~(\cdots;\bb_{s};\bc_{s})\mapsto (\cdots;\bb_{s}\bc_{s};\emp)
        \end{align}
        are bijetive and the identity $(\sigma\circ\phi_{\pm})(\bv)=\sigma(\bv)\pm 1$ always holds.
        Furthermore, since
        \[\anti(\fes)(w_{1},\ldots,w_{r-1},w_{r})=\fE(\fll{w_{1}\cdots w_{r-1}}w_{r})\anti(\fes)((w_{1}\cdots w_{r-1})\fur{w_{r}})\]
        holds, we can show $F\circ\phi_{\pm}=F$.
        Therefore we obtain
        \begin{align}
            &\ganit(\fez)((\pari\circ\anti)(\fes))(\bw)\\
            &=\begin{multlined}[t]
                \sum_{s=1}^{\infty}\left(\sum_{\bv\in D_{2s}^{1}(\bw)}(-1)^{\sigma(\bv)}F(\bv)+\sum_{\bv\in D_{2s}^{\ast 0}(\bw)}(-1)^{\sigma(\bv)}F(\bv)\right)\\
                +\sum_{s=1}^{\infty}\left(\sum_{\bv\in D_{2s}^{10}(\bw)}(-1)^{\sigma(\bv)}F(\bv)+\sum_{\bv\in D_{2s}^{\ast}(\bw)}(-1)^{\sigma(\bv)}F(\bv)\right)
            \end{multlined}\\
            &=\begin{multlined}[t]
                \sum_{s=1}^{\infty}\left(\sum_{\bv\in D_{2s}^{1}(\bw)}(-1)^{\sigma(\bv)}F(\bv)+\sum_{\bv\in D_{2s}^{1}(\bw)}(-1)^{\sigma(\phi_{+}(\bv))}F(\phi_{+}(\bv))\right)\\
                +\sum_{s=1}^{\infty}\left(\sum_{\bv\in D_{2s}^{10}(\bw)}(-1)^{\sigma(\bv)}F(\bv)+\sum_{\bv\in D_{2(s+1)}^{10}(\bw)}(-1)^{\sigma(\phi_{-}(\bv))}F(\phi_{-}(\bv))\right)
            \end{multlined}\\
            &=\sum_{\bv\in D_{2}^{10}(\bw)}(-1)^{\sigma(\bv)}F(\bv).
        \end{align}
        Since $D_{2}^{10}(\bw)$ is explicitly shown as
        \[D_{2}^{10}(w_{1},\ldots,w_{r})=\begin{cases}(w_{1};\emp) & \text{if }r=1,\\ \emp & \text{otherwise,}\end{cases}\]
        we get $\ganit(\fez)((\pari\circ\anti)(\fes))(\bw)$ is $0$ if $r\ge 2$ and $-F(w_{1};\emp)=-\fE(w_{1})$ if $r=1$.
        \item The unit $\fE$ is obviously alternal because it is of length $1$. Then we can use Proposition \ref{prop:expari_alternal} to show that $\expari(\fE)$ is symmetral.
        \item Since both $\neg$ and $\pari$ preserve the product $\mmu$, they commutes with $\invmu$. Therefore we have
        \[(\neg\circ\pari)(\fez)=(\neg\circ\pari\circ\invmu)(1-\fE)=(\invmu\circ\neg\circ\pari)(1-\fE).\]
        As obviously $1-\fE$ is $\neg\circ\pari$-invariant, so is $\invmu(1-\fE)=\fez$.
        Similarly, as $\neg$ and $\pari$ are also distributive for $\preari$, they commutes with $\expari$. Then we see that the above argument, which we used for $\fez$, is also valid for $\fes$ by replacing $\invmu$ by $\expari$ and $1-\fE$ by $\fE$.
        Finally, we prove the rest assertion $\invmu(\fes)=\push(\fes)$: as the operator $\push$ is equal to the composition $\neg\circ\anti\circ\swap\circ\anti\circ\swap$ (Proposition \ref{prop:negpush}), we obtain
        \begin{align}
            \push(\fes)
            &=(\neg\circ\anti\circ\swap\circ\anti\circ\swap)(\fez)\\
            &=(\neg\circ\anti\circ\swap\circ\anti)(\foz)\\
            &=(\neg\circ\anti\circ\swap)(\foz)\\
            &=(\neg\circ\anti)(\fes).
        \end{align}
        From the assertion (3) we know that $\fes$ is $\gantar$-invariant, which yields
        \[(\neg\circ\anti)(\fes)=(\neg\circ\pari\circ\invmu)(\fes)=(\invmu\circ\neg\circ\pari)(\fes)=\invmu(\fes).\]
    \end{enumerate}
\end{proof}

\section{On the group $\GIFF_{x}$}\label{sec:appendix}
\subsection{$\GIFF_{x}$ and its Lie algebra}
In this section, we show basic properties of the pro-algebraic group $\GIFF_{x}$. 
First we put $\cO\coloneqq\bbK[u_{1},u_{2},\ldots]/(u_{1}-1)$. 

\begin{definition}
    Let $\GIFF_{x}$ be the affine scheme over $\bbK$ represented by $\cO$. For a commutative $\bbK$-algebra $R$, since any $R$-rational point $\varphi\in\GIFF_{x}(R)=\Hom_{\Alg_{\bbK}}(\cO,R)$ of $\GIFF_{x}$ is determined by $\varphi(u_{r})$ ($r\ge 2$), we identify the set $\GIFF_{x}(R)$ with $x+x^{2}R\jump{x}$ throughout the bijection $\varphi\mapsto \sum_{r=0}^{\infty}\varphi(u_{r+1})x^{r+1}$. We always use the symbol\footnote{The definition includes $a_{0}^{f}=1$ for any $f$.} $a_{r}^{f}$ to indicate the coefficient at $x^{r+1}$ of $f\in\GIFF_{x}(R)=x+x^{2}R\jump{x}$, namely
    \[f(x)=x+\sum_{r=1}^{\infty}a_{r}^{f}x^{r}.\]
\end{definition}

Let $\varphi$ and $\psi$ be two $R$-rational point of $\GIFF_{x}$. As their constant terms are $0$, we can define a product of these points as the composition of power series:
\[\varphi\circ\psi\coloneqq\sum_{r=0}^{\infty}\varphi(u_{r+1})\left(\sum_{s=0}^{\infty}\psi(u_{s+1})x^{s+1}\right)^{r+1}.\]
Moreover, since $\varphi(u_{1})=\psi(u_{1})=1$ holds, we can always take the inverse of this product. Therefore we obtain:

\begin{proposition}
    The affine scheme $\GIFF_{x}$ is a pro-algebraic group over $\bbK$, and consequently its coordinate ring $\cO$ becomes a commutative Hopf $\bbK$-algebra.
\end{proposition}
\begin{proof}
    The group law of $\GIFF_{x}$ is described above. For the corresponding coproduct, we can compute the image of $u_{N}$ under it for $N\ge 2$ as
    \[\Delta(u_{N})=\sum_{r=1}^{\infty}\sum_{\substack{m_{1},\ldots,m_{r}\ge 1\\ m_{1}+\cdots+m_{r}=N}}u_{r}\otimes u_{m_{1}}\cdots u_{m_{r}}.\]
    The counit is given as an algebra homomorphism by $\epsilon(u_{2})=\epsilon(u_{3})=\cdots=0$.
\end{proof}

Since $\GIFF_{x}$ is a pro-algebraic group, we can consider its Lie algebra $\Lie(\GIFF_{x})$. More precisely, denote by $R[\ep]\coloneqq R\jump{t}/(t^{2})$ the ring of dual numbers and define the set $\Lie(\GIFF_{x})(R)$ as the kernel of $\GIFF_{x}(R[\ep])\to\GIFF_{x}(R)$, which is induced from the natural surjection $R[\ep]\to R;~a+b\ep\mapsto a$. 

\begin{proposition}\label{prop:diff}
    The set $\Lie(\GIFF_{x})(R)$ is identified with
    \begin{equation}\label{eq:diff}
        \DIFF_{x}(R)\coloneqq\left\{\left.\sum_{r=1}^{\infty}\ep_{r}x^{r+1}\frac{d}{dx}~\right|~\ep_{1},\ep_{2},\ldots\in R\right\}
    \end{equation}
    and its bracket (Lie-law) is given by the anti-commutator as differential operators (namely, $[A,B]=BA-AB$).
\end{proposition}
\begin{proof}
    Since we have by definition
    \[\Lie(\GIFF_{x})(R)=\{\id+\ep f\mid f\in x^{2}R\jump{x}\},\]
    we get a natural bijection
    \[\Lie(\GIFF_{x})(R)\to\DIFF_{x}(R);~\id+\ep f\mapsto f(x)\frac{d}{dx}.\]
    Using the standard way to get a bracket, which linearizes the adjoint map $\Ad\colon\GIFF_{x}\to\GL(\DIFF_{x})$ between pro-algebraic groups, for $a\in\Lie(\GIFF_{x})(R)$ and $b\in\Lie(\GIFF_{x})(R[\ep])$, we see that
    \[\ad_{a}(b)=b_{0}+(\alpha\circ b_{0}-b_{0}\circ\alpha+b_{1})\ep,\]
    where $b$ is decomposed as $b=b_{0}+b_{1}\ep$ through the direct sum decomposition $\Lie(\GIFF_{x})(R[\ep])=\Lie(\GIFF_{x})(R)\oplus\ep\Lie(\GIFF_{x})(R)$, and we put $a=\id+\ep\alpha$ using the unique element $\alpha\in x^{2}R\jump{x}$.
    Therefore we have the expression of the bracket $[a,b_{0}]=\alpha\circ b_{0}-b_{0}\circ\alpha$ with the same notation as above, which is valid for any $a,b\in\Lie(\GIFF_{x})(R)$.
    On the other hand, computing it directly as the composition of power series, we obtain
    \begin{align}
        [a,b_{0}](x)
        &=(\alpha\circ b_{0}-b_{0}\circ\alpha)(x)\\
        &=(\alpha\circ b_{0}-b_{0}\circ\alpha)(x)-\alpha(x)+\alpha(x)\\
        &=(\alpha\circ \ep\beta-\ep\beta\circ\alpha)(x),
    \end{align}
    where we decomposed as $b_{0}=\id+\ep\beta$ by $\beta\in x^{2}R\jump{x}$.
    We write the coefficients of $\alpha$ and $\beta$ as
    \[\alpha(x)=\sum_{r=1}^{\infty}\alpha_{r}x^{r+1}\]
    and
    \[\beta(x)=\sum_{s=1}^{\infty}\beta_{s}x^{s+1},\]
    respectively.
    Continuing this calculation, we have
    \begin{align}
        [a,b_{0}](x)
        &=(\alpha\circ \ep\beta-\beta\circ\ep\alpha)(x)\\
        &=(\alpha\circ b_{0}-\beta\circ a)(x)-\alpha(x)+\beta(x)\\
        &=\sum_{r=1}^{\infty}\alpha_{r}(b_{0}(x)^{r+1}-x^{r+1})-\sum_{s=1}^{\infty}\beta_{s}(a(x)^{s+1}-x^{s+1})\\
        &=\sum_{r=1}^{\infty}\alpha_{r}((x+\ep\beta(x))^{r+1}-x^{r+1})-\sum_{s=1}^{\infty}\beta_{s}((x+\ep\alpha(x))^{s+1}-x^{s+1})\\
        &=\sum_{r=1}^{\infty}\alpha_{r}(r+1)x^{r}\ep\beta(x)-\sum_{s=1}^{\infty}\beta_{s}(s+1)x^{s}\ep\alpha(x)\\
        &=\ep\left(\beta(x)\frac{d}{dx}\alpha(x)-\alpha(x)\frac{d}{dx}\beta(x)\right)\\
        &=\ep\left(\beta(x)\frac{d}{dx}\alpha(x)\frac{d}{dx}-\alpha(x)\frac{d}{dx}\beta(x)\frac{d}{dx}\right)(x).
    \end{align}
    This result shows that, the bracket of $\DIFF_{x}(R)$ is simply given by the anti-commutator as differential operators under the identification \eqref{eq:diff}.
\end{proof}

Due to the following proposition, we can consider the exponential isomorphism from $\DIFF_{x}$ to $\GIFF_{x}$. 
To prove this, we recall that $\BIFF_{x}\coloneqq x\bbK\jump{x}$ naturally has a decreasing filtration
\[\Fil^{>n}\BIFF_{x}\coloneqq\left\{\left.\sum_{r=n+1}^{\infty}a_{r}x^{r}~\right|~a_{r}\in\bbK~(r>n)\right\}=x^{n+1}\bbK\jump{x}\]
and is completed along it. Note that this filtration is also separative (i.e., the intersection of all $\Fil^{>n}\BIFF_{x}$ is empty) and each piece of the associated graded space \[\gr\BIFF_{x}\coloneqq\bigoplus_{n=0}^{\infty}\Fil^{>n}\BIFF_{x}/\Fil^{>n+1}\BIFF_{x}\] is a finite-dimensional $\bbK$-vector space.

\begin{proposition}\label{prop:GIFF_pro_unipotent}
    The pro-algebraic group $\GIFF_{x}$ is pro-unipotent.
\end{proposition}
\begin{proof}
    Define the right-action $\rho\colon\GIFF_{x}\to\GL(\BIFF_{x})$ by $\rho_{R}(g)\coloneqq (f\mapsto f\circ g)$ for $g\in\GIFF_{x}(R)$ and $f\in R\hotimes\BIFF_{x}$. It suffices to show that $(\rho,\BIFF_{x})$ gives a faithful pro-unipotent representation of $\GIFF_{x}$.
    To show the faithfulness, assume $\rho_{R}(g)=\id$ for some $g\in\GIFF_{x}(R)$.
    It means that $\rho_{R}(g)(f)$ is always $f$, and it yields $g(x)=x$ as $\GIFF_{x}(R)$ is a group.\\
    
    Next, we prove that each value of $\rho$ preserves the filtration.
    Let $g\in\GIFF_{x}(R)$ and take a non-negative integer $n$.
    Then, for any $f\in R\otimes\Fil^{>n}\BIFF_{x}$, we have
    \[\rho_{R}(g)(f)(x)=(f\circ g)(x)=\sum_{r=n+1}^{\infty}a_{r}g(x)^{n+1}\]
    for some $a_{n+1},a_{n+2},\ldots\in R$.
    Since $g(x)$ belongs to $x+x^{2}R\jump{x}$, we see that the above series $\rho_{R}(g)(f)(x)$ belongs to $x^{n+1}R\jump{x}$, which is nothing but $R\otimes\Fil^{>n}\BIFF_{x}$.\\

    Finally we show the triviality of $\rho$ on the associated graded space $R\otimes\gr\BIFF_{x}$.
    As every element of each graded piece $R\otimes(\Fil^{>n}\BIFF_{x}/\Fil^{>n+1}\BIFF_{x})$ is represented by $I_{a}(x)\coloneqq ax^{n+1}$ by some $a\in R$, we obtain
    \[\rho_{R}(g)(I_{a})(x)=ag(x)^{n+1}=ax^{n+1}+(\text{terms with degree }\ge n+2)\equiv I_{a}(x)~\mod~R\otimes\cF^{n+1}\BIFF_{x},\]
    by virtue of the fact $g(x)\in x+x^{2}R\jump{x}$. Therefore we see that $\rho_{R}(g)$ identically behaves on each graded piece, and we now complete the proof. 
\end{proof}

\begin{corollary}\label{cor:exp_giff}
    The exponential map $\exp\colon\DIFF_{x}\to\GIFF_{x}$ gives an isomorphism of schemes and it is explicitly expressed as
    \[\exp(D)(x)=\sum_{n=0}^{\infty}\frac{D^{n}}{n!}(x)=x+\sum_{n=1}^{\infty}\frac{1}{n!}\underbrace{D\cdots D}_{n}(x)\]
    as usual, for any $D\in\DIFF_{x}(R)$.
\end{corollary}
\begin{proof}
    Remember that $(\rho,\BIFF_{x})$ is a pro-unipotent representation of $\GIFF_{x}$ from the proof of Proposition \ref{prop:GIFF_pro_unipotent}.
    For any $g\in R[\ep]\hotimes\BIFF_{x}$, which is decomposed as $g=g_{0}+\ep g_{1}$ by some $g_{0},g_{1}\in R\hotimes\BIFF_{x}$, we have
    \begin{align}
        \rho_{R[\ep]}\left(f(x)\frac{d}{dx}\right)(g)(x)
        &=(g\circ(\id+\ep f))(x)\\
        &=\sum_{r=1}^{\infty}(g_{0,r}+\ep g_{1,r})(x+\ep f(x))^{r}\\
        &=\sum_{r=1}^{\infty}(g_{0,r}+\ep g_{1,r})(x^{r}+\ep rx^{r-1}f(x))\\
        &=\sum_{r=1}^{\infty}(g_{0,r}x^{r}+\ep g_{1,r}x^{r}+\ep rg_{0,r}x^{r-1}f(x))\\
        &=g_{0}(x)+\ep\left(g_{1}(x)+f(x)\frac{d}{dx}g_{0}(x)\right)
    \end{align}
    through the identification \eqref{eq:diff}, where we also denoted as $g_{i}(x)=\sum_{r=1}^{\infty}g_{i,r}x^{r}$ for $i\in\{0,1\}$.
    Therefore it holds that $\rho_{R[\ep]}(D)=\id+\ep D$ for an arbitrary $D\in\DIFF_{x}(R)$.
    Thus we have
        \[\rho_{R}(\exp(D))(g)(x)=\sum_{n=0}^{\infty}\frac{D^{n}g}{n!}(x)\]
        due to \cite[Proposition 4.6]{racinet00}, and taking $g=\id$ in the above equality, we get the desired result.
\end{proof}

Following \cite[\S 4.1]{ecalle11}, we denote by $f_{\ast}\in x^{2}R\jump{x}$ the \emph{infinitesimal generator} of each $f\in\GIFF_{x}(R)$, which is defined to be the ``logarithm'' of $f$:
\[\exp\left(f_{\ast}(x)\frac{d}{dx}\right)(x)=f(x).\]
Once we fix $f$, the coefficients of $f_{\ast}$ are denoted as
\[f_{\ast}(x)=\sum_{r=1}^{\infty}\epsilon_{r}^{f}x^{r+1}.\]
Similarly, we define the \emph{infinitesimal dilator} by the formula
\[f_{\#}(x)\coloneqq x-\frac{f(x)}{f'(x)}\in x^{2}R\jump{x}.\]
Its coefficient is usually expressed as $\gamma_{r}^{f}$.

\subsection{$\GARI$ and $\ARI$ as schemes}
Recall $\BIMU(R)$ is defined as the direct product $\prod_{r=0}^{\infty}\cX_{r}(R)$ of sets.
Using the notation $C_{\cX,r}$ standing for the coordinate ring of $\cX_{r}$ regarded as an affine scheme, we can apply the well-known commutation between categorical limits and representable functors:
\begin{align}
    \BIMU(R)
    &=\prod_{r=0}^{\infty}\cX_{r}(R)\\
    &=\prod_{r=0}^{\infty}\Hom_{\CAlg_{\bbK}}(C_{\cX,r},R)\\
    &\simeq \Hom_{\CAlg_{\bbK}}\left(\bigotimes_{r=0}^{\infty} C_{\cX,r},R\right),
\end{align}
which is guaranteed by the fact that $\CAlg_{\bbK}$ is a cocomplete category.
From the assumption $\cX_{0}\simeq\bbA_{\bbK}^{1}$, we can identify $C_{\cX,0}$ with $\bbK[x]$ as a $\bbK$-algebra.
Then, by their deinitions, both $\MU$ and $\LU$ are also specified as closed subschemes of $\BIMU$:
\begin{align}
    \MU&=\Spec\left((\bbK[x]/(x-1))\otimes\bigotimes_{r=1}^{\infty}C_{\cX,r}\right),\\
    \LU&=\Spec\left((\bbK[x]/(x))\otimes\bigotimes_{r=1}^{\infty}C_{\cX,r}\right).
\end{align}
Moreover, together with the binary product $\gari$, the functor $\GARI$, which is isomorphic to $\MU$ as a scheme, becomes pro-algebraic group.

\begin{proposition}\label{prop:GARI_pro_unipotent}
    The pro-algebraic group $\GARI$ is pro-unipotent.
\end{proposition}
\begin{proof}
    Consider the morphism $\tau\colon\GARI\to\GL(\BIMU(\bbK))$ of pro-algebraic groups defined by\footnote{In a strict sense, we have not defined the $\gari$-product of $M\in R\hotimes\BIMU(\bbK)$ and $A\in\GARI(R)$. However, one will easily understand that it can be defined by the quite same procedure as that defining the ordinary $\gari$, as it is linear about the first component.} $\tau_{R}(A)\colon M\mapsto\gari(M,A)$ for $A\in\GARI(R)$ and $M\in R\hotimes\BIMU(\bbK)$. This $\tau$ defines a right-action, and gives a faithful representation of pro-algebraic groups since the group law $\gari$ admits the unit $1$.
    We can also check that $\tau$ preserves the filtration: when a given $M$ belongs to $\Fil^{\ge n}\BIMU(\bbK)$ for some $n\ge 0$, we can compute each specified component of its image under $\tau_{R}(A)$ ($A\in\GARI(R)$) as
    \begin{align}
        &\tau_{R}(A)(M)(\bw)\\
        &=\gari(M,A)(\bw)\\
        &=\mmu(\garit(A)(M),A)(\bw)\\
        &=\sum_{\bw=\bw'\bb}\garit(A)(M)(\bw')A(\bb)\\
        &=\sum_{\bw=\bw'\bb}\sum_{t=1}^{\infty}\sum_{\{(\ba_{j};w_{i_{j}};\bc_{j})\}_{j=1}^{t}\in\Sigma_{t}(\bw')}M\left(\prod_{j=1}^{t}\ful{\ba_{j}}w_{i_{j}}\fur{\bc_{j}}\right)A(\bb)\prod_{j=1}^{t}\left(A(\ba_{j}\flr{w_{i_{j}}})\invmu(A)(\fll{w_{i_{j}}}\bc_{j})\right),
    \end{align}
    where we used \eqref{eq:gaxit_simple} in the last equality.
    Since $M$ is assumed to be in $R\otimes\Fil^{\ge n}\BIMU(\bbK)$, the above sum about $t$ is restricted to one over the more narrow range $t\ge n$.
    As a consequence, if the sequence $\bw$ has the length less than $n$, we cannot take any required element $\{(\ba_{j};w_{i_{j}};\bc_{j})\}_{j=1}^{t}$ of $\Sigma_{t}(\bw)$.
    It means that the above sum is empty, which yields that $\tau_{R}(A)(M)$ belongs to $R\otimes\Fil^{\ge n}\BIMU(\bbK)$.\\

    On the associated graded space, it is reduced to the identity: let $M$ be taken from $R\otimes(\Fil^{\ge n}\BIMU(\bbK)/\Fil^{\ge n+1}\BIMU(\bbK))$ for some $n$.
    Then, in the equality
    \begin{multline}\tau_{R}(A)(M)(\bw)\\=\sum_{\bw=\bw'\bb}\sum_{t=1}^{\infty}\sum_{\{(\ba_{j};w_{i_{j}};\bc_{j})\}_{j=1}^{t}\in\Sigma_{t}(\bw')}M\left(\prod_{j=1}^{t}\ful{\ba_{j}}w_{i_{j}}\fur{\bc_{j}}\right)A(\bb)\prod_{j=1}^{t}\left(A(\ba_{j}\flr{w_{i_{j}}})\invmu(A)(\fll{w_{i_{j}}}\bc_{j})\right)\end{multline}
    we obtained above, the $t=n$ term only remains while the others vanish.
    Moreover, since we showed that $\tau_{R}(A)$ preserves filtration, we can assume that the length of $\bw$ is $n$.
    Thus all of the subsequences $\ba_{1},\bc_{1},\ldots,\ba_{t},\bc_{t}$ and $\bb$ are necessarily $\emp$, which means 
    \[\tau_{R}(A)(M)(w_{1},\ldots,w_{n})=M(w_{1},\ldots,w_{n}).\]
    It completes the proof.
\end{proof}

\begin{proposition}\label{prop:expari}
    The map
    \[\expari\colon\ARI(R)\to\GARI(R);~A\mapsto 1+\sum_{n=1}^{\infty}\frac{1}{n!}\preari(\underbrace{A,\ldots,A}_{n})\]
    gives the Lie-exponential map, where $\preari(\cdots)$ is taken from left to right: $\preari(A,A,A)=\preari(\preari(A,A),A)$ for example.
\end{proposition}
\begin{proof}
    This proof proceeds in an almost same way as that of Corollary \ref{cor:exp_giff}.
    In the proof of Proposition \ref{prop:GARI_pro_unipotent}, we construct the pro-unipotent representation $(\tau,\BIMU(\bbK))$ of $\GARI$.
    We consider the linearlization of $\tau_{R[\ep]}(B)$ for $B'\in\ARI(R)$: here we use the identification $B=1+\ep B'$ of $B'\in\ARI(R)$ with the subgroup of $B\in\GARI(R[\ep])$.
     since it is defined to be the $\gari$-product from right, from Remark \ref{rem:linearization} we see that $\tau_{R[\ep]}(B)=\id+\ep p_{B}$, where $p_{B}(A)\coloneqq\preari(A,B')$.
    Hence, by \cite[Proposition 4.6]{racinet00}, we have
        \[\tau_{R}(\exp_{\ARI\to\GARI}(B'))(A)(x)=\sum_{n=0}^{\infty}\frac{p_{B}^{\circ n}(A)}{n!},\]
    where $\exp_{\ARI\to\GARI}$ denotes the Lie-exponential map. Putting $A=1\in\GARI(R)$ in this equality, we obtain $\exp_{\ARI\to\GARI}$ coincides with $\expari$ defined in the assertion.
\end{proof}

\subsection{Relationship of $\GIFF_{x}$ and $\GARI$}
Define the family $\{\fre_{r}\}_{r\ge 1}$ of bimoulds by the following recursive rules:
\begin{align}
    \fre_{1}&\coloneqq\fE,\\
    \fre_{r+1}&\coloneqq\arit(\fre_{r})(\fE)\qquad (r\ge 1).
\end{align}
This definition implies that each $\fre_{r}$ has the only non-zero component at length $r$, that is, $\fre_{r}\in\BIMU_{r}$.
Moreover, since $\fE$ is a bimould of length $1$, we see that $\fre_{r}$ is alternal for every $r\ge 1$ by virtue of \cite[(A.3)]{fk23} which states that if $X$ and $Y$ are alternal then so is $\arit(Y)(X)$.
We also note that, by definition, $\fre_{r}$ is explicitly written in terms of $\fE$ and $\fre_{r-1}$ as
\begin{equation}\label{eq:re_explicit}
    \fre_{r}(w_{1},\ldots,w_{r})=\fE(w_{1}\fur{w_{2}\cdots w_{r}})\fre_{r-1}(\fll{w_{1}}(w_{2}\cdots w_{r}))-\fE(\ful{w_{1}\cdots w_{r-1}}w_{r})\fre_{r-1}((w_{1}\cdots w_{r-1})\flr{w_{r}}).
\end{equation}

\begin{lemma}[{\cite[(4.59)]{ecalle11}}]\label{lem:ecalle_459}
    For each $r\ge 1$, we have
    \begin{equation}\label{eq:ecalle_459}
        \dro_{r}(\bw)\coloneqq\swap(\fre_{r})(\bw)=\sum_{\bw=\ba w_{i}\bb}(r+1-i)\foz(\ba\flr{w_{i}})\fO(\ful{\ba}w_{i}\fur{\bb})\foz(\fll{w_{i}}\bb)
    \end{equation}
\end{lemma}
\begin{proof}
    We use Schneps' identities \cite[(2.4.7), (2.4.8)]{schneps15}
    \begin{equation}\label{eq:amit_swap}
        (\swap\circ\amit(\swap(A))\circ\swap)(B)=\amit(A)(B)+\mmu(B,A)-\swap(\mmu(\swap(B),\swap(A)))
    \end{equation}
    and
    \begin{equation}\label{eq:anit_swap}
        (\swap\circ\anit(\swap(A))\circ\swap)(B)=\anit(\push(A))(B),
    \end{equation}
    that are valid for any $A\in\LU(R)$ and $B\in\BIMU(R)$.
    Considering the swap-version of the recursive definition $\fre_{r}=\arit(\fre_{r-1})(\fE)$, we see that
    \begin{align}
        \dro_{r}
        &=(\swap\circ\amit(\fre_{r-1}))(\fE)-(\swap\circ\anit(\fre_{r-1}))(\fE)\\
        &=\amit(\dro_{r-1})(\fO)+\mmu(\fO,\dro_{r-1})-\swap(\mmu(\swap(\fO),\swap(\dro_{r-1})))-\anit(\push(\dro_{r-1}))(\fO)\\
        &=\amit(\dro_{r-1})(\fO)+\mmu(\fO,\dro_{r-1})-\swap(\mmu(\fE,\fre_{r-1}))-\anit(\push(\dro_{r-1}))(\fO).
    \end{align}
    By computing explicitly, the terms $\amit(\dro_{r-1})(\fO)$ and $\swap(\mmu(\fre_{r-1},\fO))$ cancel each other out.
    Thus we get the equality
    \[\dro_{r}=\mmu(\fO,\dro_{r-1})-\anit(\push(\dro_{r-1}))(\fO).\]
    Since this formula and the initial condition $\dro_{1}=\fO$ determine every $\dro_{r}$, it is sufficient\footnote{The same initial condition for $b_{1}$ is automatically satisfied.} to show that
    \[b_{r}=\mmu(\fO,b_{r-1})-\anit(\push(b_{r-1}))(\fO)\]
    holds for any $r\ge 2$, where we put
    \[b_{r}(\bw)\coloneqq\sum_{\bw=\ba w_{i}\bb}(r+1-i)\foz(\ba\flr{w_{i}})\fO(\ful{\ba}w_{i}\fur{\bb})\foz(\fll{w_{i}}\bb).\]
    First we note that the factor $\foz(\ba\flr{w_{i}})\fO(\ful{\ba}w_{i}\fur{\bb})\foz(\fll{w_{i}}\bb)$ is simply written as $-\push^{\circ i}(\foz)(\ba w_{i}\bb)$.
    Therefore Proposition \ref{prop:unit_push} (3) yields
    \begin{equation}\label{eq:b_transform}
        b_{r}(\bw)=\foz(\bw)+\sum_{\bw=\ba w_{i}\bb}(r-i)\foz(\ba\flr{w_{i}})\fO(\ful{\ba}w_{i}\fur{\bb})\foz(\fll{w_{i}}\bb).
    \end{equation}
    Moreover, for any sequence $(w_{1},\ldots,w_{s})$ it holds that
    \begin{align}
        &\push(b_{s})(w_{1},\ldots,w_{s})\\
        &=b_{s}\binom{-u_{1}-\cdots-u_{s},u_{1},\ldots,u_{s-1}}{-v_{s},v_{1}-v_{s},\ldots,v_{s-1}-v_{s}}\\
        &=\begin{multlined}[t]s\fO\binom{-u_{s}}{-v_{s}}\foz\binom{u_{1},\ldots,u_{s-1}}{v_{1},\ldots,v_{s-1}}\\+\sum_{i=1}^{s-1}(s-i)\foz\binom{-u_{1}-\cdots-u_{s},u_{1},\ldots,u_{i-1}}{-v_{i},v_{1}-v_{i},\ldots,v_{i-1}-v_{i}}\fO\binom{-u_{s}}{v_{i}-v_{s}}\foz\binom{u_{i+1},\ldots,u_{s-1}}{v_{i+1}-v_{i},\ldots,v_{s-1}-v_{i}}\end{multlined}\\
        &=\begin{multlined}[t]-s\foz(w_{1}\cdots w_{s})\\+\sum_{i=1}^{s-1}(s-i)\fO(\ful{w_{1}\cdots w_{i-1}}w_{i}\fur{w_{i+1}\cdots w_{s}})\foz((w_{1}\cdots w_{i-1})\flr{w_{i}})\foz(\fll{w_{i}}(w_{i+1}\cdots w_{s}))\end{multlined}\\
        &=-s\foz(w_{1},\ldots,w_{s})+\sum_{\substack{w_{1}\cdots w_{s}=\ba w_{i}\bb\\ \bb\neq\emp}}(s-i)\foz(\ba\flr{w_{i}})\fO(\ful{\ba}w_{i}\fur{\bb})\foz(\fll{w_{i}}\bb),
    \end{align}
    where we used the parity condition for flexion units in the third equality.
    Using this expression and putting $\bw=w_{1}\bw_{1}$, we have
    \begin{align}
        &\anit(\push(b_{r-1}))(\fO)(\bw)\\
        &=\fO(w_{1}\fur{\bw_{1}})\push(b_{r-1})(\fll{w_{1}}\bw_{1})\\
        &=\begin{multlined}[t]-(r-1)\fO(w_{1}\fur{\bw_{1}})\foz(\fll{w_{1}}\bw_{1})\\
            +\fO(w_{1}\fur{\bw_{1}})\sum_{\substack{\bw_{1}=\ba w_{i}\bb\\ \bb\neq\emp}}(r-i)\foz((\fll{w_{1}}\ba)\flr{(\fll{w_{1}}w_{i})})\fO(\ful{(\fll{w_{1}}\ba)}\fll{w_{1}}w_{i}\fur{(\fll{w_{1}}\bb)})\foz(\fll{(\fll{w_{1}}w_{i})}(\fll{w_{1}}\bb))\end{multlined}\\
        &=-(r-1)\fO(w_{1}\fur{\bw_{1}})\foz(\fll{w_{1}}\bw_{1})+\fO(w_{1}\fur{\bw_{1}})\sum_{\substack{\bw_{1}=\ba w_{i}\bb\\ \bb\neq\emp}}(r-i)\foz(\ba\flr{w_{i}})\fO(\fll{w_{1}}\ful{\ba}w_{i}\fur{\bb})\foz(\fll{w_{i}}\bb)\\
    \end{align}
    Then we can apply the tripartite identity \eqref{eq:tripartite} as
    \begin{align}
        \fO(w_{1}\fur{\bw_{1}})\fO(\fll{w_{1}}\ful{\ba}w_{i}\fur{\bb})
        &=\fO(w_{1}\fur{(\ful{\ba}w_{i}\fur{\bb})})\fO(\fll{w_{1}}\ful{\ba}w_{i}\fur{\bb})\\
        &=\fO(w_{1})\fO(\ful{\ba}w_{i}\fur{\bb})-\fO(w_{1}\flr{(\ful{\ba}w_{i}\fur{\bb})})\fO(\ful{w_{1}}\ful{\ba}w_{i}\fur{\bb})\\
        &=\fO(w_{1})\fO(\ful{\ba}w_{i}\fur{\bb})-\fO(w_{1}\flr{w_{i}})\fO(\ful{w_{1}\ba}w_{i}\fur{\bb}).
    \end{align}
    It yields that
    \begin{align}
        &\anit(\push(b_{r-1}))(\fO)(\bw)\\
        &=\begin{multlined}[t]-(r-1)\fO(w_{1}\fur{\bw_{1}})\foz(\fll{w_{1}}\bw_{1})\\+\sum_{\substack{\bw_{1}=\ba w_{i}\bb\\ \bb\neq\emp}}(r-i)\foz(\ba\flr{w_{i}})\left(\fO(w_{1})\fO(\ful{\ba}w_{i}\fur{\bb})-\fO(w_{1}\flr{w_{i}})\fO(\ful{w_{1}\ba}w_{i}\fur{\bb})\right)\foz(\fll{w_{i}}\bb)\end{multlined}\\
        &=\begin{multlined}[t]-(r-1)\fO(w_{1}\fur{\bw_{1}})\foz(\fll{w_{1}}\bw_{1})+\fO(w_{1})\sum_{\substack{\bw_{1}=\ba w_{i}\bb\\ \bb\neq\emp}}(r-i)\foz(\ba\flr{w_{i}})\fO(\ful{\ba}w_{i}\fur{\bb})\foz(\fll{w_{i}}\bb)\\
            -\sum_{\substack{\bw_{1}=\ba w_{i}\bb\\ \bb\neq\emp}}(r-i)\foz(\ba\flr{w_{i}})\fO(w_{1}\flr{w_{i}})\fO(\ful{w_{1}\ba}w_{i}\fur{\bb})\foz(\fll{w_{i}}\bb)\end{multlined}\\
        &=\begin{multlined}[t]-(r-1)\fO(w_{1}\fur{\bw_{1}})\foz(\fll{w_{1}}\bw_{1})+\fO(w_{1})(b_{r-1}(\bw_{1})\\-\foz(\bw_{1}))-\sum_{\substack{\bw=\ba' w_{i}\bb\\ \ba',\bb\neq\emp}}(r-i)\foz(\ba'\flr{w_{i}})\fO(\ful{\ba'}w_{i}\fur{\bb})\foz(\fll{w_{i}}\bb)\end{multlined},
    \end{align}
    where we used \eqref{eq:b_transform} in the last equality.
    Hence we obtain
    \begin{align}
        &(\mmu(\fO,b_{r-1})-\anit(\push(b_{r-1}))(\fO))(\bw)\\
        &=(r-1)\fO(w_{1}\fur{\bw_{1}})\foz(\fll{w_{1}}\bw_{1})-\foz(\bw)+\sum_{\substack{\bw=\ba' w_{i}\bb\\ \ba',\bb\neq\emp}}(r-i)\foz(\ba'\flr{w_{i}})\fO(\ful{\ba'}w_{i}\fur{\bb})\foz(\fll{w_{i}}\bb)\\
        &=-\foz(\bw)+\sum_{\bw=\ba' w_{i}\bb}(r-i)\foz(\ba'\flr{w_{i}})\fO(\ful{\ba'}w_{i}\fur{\bb})\foz(\fll{w_{i}}\bb).
    \end{align}
    Using \eqref{eq:b_transform} again, we get the desired result.
\end{proof}

\begin{proposition}[{\cite[(4.7)]{ecalle11}}]\label{prop:ecalle_47}
    Let $r$ and $s$ be positive integers. Then we have $\ari(\fre_{r},\fre_{s})=(r-s)\fre_{r+s}$.
\end{proposition}
\begin{proof}
    We define three auxiliary quantities $A,D,E$ by
    \begin{align}
        A(p,q)&\coloneqq\ari(\fre_{p},\fre_{q})-(p-q)\fre_{p+q},\\
        D(p,q)&\coloneqq\arit(\fre_{p-1})(\fre_{q+1})-\arit(\fre_{p})(\fre_{q})-\fre_{p+q}+\lu(\fre_{p},\fre_{q}),\\
        E(p,q)&\coloneqq \arit(\fre_{p})(\fre_{q})-q\fre_{p+q}-\sum_{i=1}^{q-1}\lu(\fre_{i},\fre_{p+q-i})
    \end{align}
    for positive integers $p$ and $q$, except for $D(1,q)$; we do not define it since $\fre_{0}$ is not defined, however we adopt such notations to make $D(p,q)$ a bimould of length $p+q$ like $A(p,q)$ and $E(p,q)$.
    We divide the proof into three steps: for every $p,q\ge 1$, we shall show that
    \begin{enumerate}
        \item It always holds that $E(p,2)=0$.
        \item The equalities $A(p+1,q)=0$ and $D(p+1,q)=0$ are equivalent under the assumption $A(p,q)=0$.
        \item Under the assumption that $E(p,q-1)=0$ and $A(i,1)=0$ hold for all $i\in\{1,2,\ldots,p+q-1\}$, the equalities $E(p+1,q-1)=0$ and $E(p,q)=0$ are equivalent.
    \end{enumerate}
    Here is how we show the desired identity $A(r,s)=0$ from these items: by definition, the identities
    \begin{align}
        A(p,q)&=E(q,p)-E(p,q)=-A(q,p),\\
        D(p+1,q)&=E(p,q+1)-E(p+1,q) 
    \end{align}
    and $E(p,1)=0$ are easily checked for arbitrary $p$ and $q$. Hence (1) leads to $D(p,1)=E(p-1,2)-E(p,1)=0$ for all $p\ge 1$. Therefore we know that $A(p,1)=0$ always implies $A(p+1,1)=0$ from (2), and it makes us use induction on $p$ to prove $A(p,1)=0$ together with the trivial identity $A(1,1)=0$. Finally, we use a nested inductions to prove $E(p,q)=0$: if we have some $N$ such that $E(i,j)=0$ for all $i$ and $j$ satisfying $i+j<N$, due to (3), we have the implication $E(N-q,q)\implies E(N-q-1,q+1)$. Since the initial condition $E(1,1)=0$ is already obtained, we get $E(p,q)=0$ for all $p$ and $q$ and consequently $A(p,q)=E(q,p)-E(p,q)=0$.
    \begin{proof}[Proof of (1)]
        This is shown by direct computation. Remember that what we should prove is the equality $\arit(\fre_{p})(\fre_{2})=2\fre_{p+2}+\lu(\fE,\fre_{p+1})$ for $p\ge 1$.
        First, applying \eqref{eq:re_explicit} to the term involving $\fre_{p+2}$ in the right-hand side, we have
        \begin{multline}
            (2\fre_{p+2}+\lu(\fE,\fre_{p+1}))(w_{1}\cdots w_{p+2})\\
            =2\fE(\ful{w_{1}\cdots w_{p+1}}w_{p+2})\fre_{p+1}((w_{1}\cdots w_{p+1})\flr{w_{p+2}})-2\fE(w_{1}\fur{w_{2}\cdots w_{p+2}})\fre_{p+1}(\fll{w_{1}}(w_{2}\cdots w_{p+2}))\\
            +\fE(w_{1})\fre_{p+1}(w_{2}\cdots w_{p+2})-\fE(w_{p+2})\fre_{p+1}(w_{1}\cdots w_{p+1}).
        \end{multline}
        We apply \eqref{eq:re_explicit} again to all terms with $\fre_{p+1}$ to get
        \begin{multline}
            (2\fre_{p+2}+\lu(\fE,\fre_{p+1}))(w_{1}\cdots w_{p+2})\\
            =(\fE(w_{1})\fE(\ful{w_{1}\cdots w_{p+1}}w_{p+2})-2\fE(w_{1}\fur{w_{2}\cdots w_{p+2}})\fE(\fll{w_{1}}\ful{w_{2}\cdots w_{p+1}}w_{p+2}))\fre_{p}((w_{2}\cdots w_{p+1})\flr{w_{p+2}})\\
            -(\fE(w_{1})\fE(w_{2}\fur{w_{3}\cdots w_{p+2}})-2\fE(w_{1}\fur{w_{2}\cdots w_{p+2}})\fE(\fll{w_{1}}w_{2}\fur{w_{3}\cdots w_{p+2}}))\fre_{p}(\fll{w_{2}}(w_{2}\cdots w_{p+1}))\\
            -(\fE(\ful{w_{1}\cdots w_{p}}w_{p+1})\fE(w_{p+2})-2\fE(\ful{w_{1}\cdots w_{p}}w_{p+1}\flr{w_{p+2}})\fE(\ful{w_{1}\cdots w_{p+1}}w_{p+2}))\fre_{p}((w_{1}\cdots w_{p})\flr{w_{p+1}})\\
            +(\fE(w_{1}\fur{w_{2}\cdots w_{p+2}})\fE(w_{p+2})-2\fE(w_{1}\fur{w_{2}\cdots w_{p+1}}\flr{w_{p+2}})\fE(\ful{w_{1}\cdots w_{p+1}}w_{p+2}))\fre_{p}(\fll{w_{1}}(w_{2}\cdots w_{p+1})).
            \end{multline}
            Let us focus on the four parenthesized factors above including the product of $\fE$-values. Applying the tripartite identity \eqref{eq:tripartite} (Definition \ref{def:unit} (2)) to them, we obtain
            \begin{align}
                &\begin{multlined}\fE(w_{1})\fE(\ful{w_{1}\cdots w_{p+1}}w_{p+2})-2\fE(w_{1}\fur{w_{2}\cdots w_{p+2}})\fE(\fll{w_{1}}\ful{w_{2}\cdots w_{p+1}}w_{p+2})\\
                    =\fE(w_{1}\flr{w_{p+2}})\fE(\ful{w_{1}\cdots w_{p+1}}w_{p+2})-\fE(w_{1}\fur{w_{2}\cdots w_{p+2}})\fE(\fll{w_{1}}\ful{w_{2}\cdots w_{p+1}}w_{p+2}),
                \end{multlined}\\
                &\begin{multlined}
                \fE(w_{1})\fE(w_{2}\fur{w_{3}\cdots w_{p+2}})-2\fE(w_{1}\fur{w_{2}\cdots w_{p+2}})\fE(\fll{w_{1}}w_{2}\fur{w_{3}\cdots w_{p+2}})\\
                =\fE(w_{1}\flr{w_{2}})\fE(\ful{w_{1}}w_{2}\fur{w_{3}\cdots w_{p+2}})-\fE(w_{1}\fur{w_{2}\cdots w_{p+2}})\fE(\fll{w_{1}}w_{2}\fur{w_{3}\cdots w_{p+2}}),
                \end{multlined}\\
                &\begin{multlined}
                \fE(\ful{w_{1}\cdots w_{p}}w_{p+1})\fE(w_{p+2})-2\fE(\ful{w_{1}\cdots w_{p}}w_{p+1}\flr{w_{p+2}})\fE(\ful{w_{1}\cdots w_{p+1}}w_{p+2})\\
                =\fE(\ful{w_{1}\cdots w_{p}}w_{p+1}\fur{w_{p+2}})\fE(\fll{w_{p+1}}w_{p+2})-\fE(\ful{w_{1}\cdots w_{p}}w_{p+1}\flr{w_{p+2}})\fE(\ful{w_{1}\cdots w_{p+1}}w_{p+2}),
                \end{multlined}\\
                &\begin{multlined}
                \fE(w_{1}\fur{w_{2}\cdots w_{p+2}})\fE(w_{p+2})-2\fE(w_{1}\fur{w_{2}\cdots w_{p+1}}\flr{w_{p+2}})\fE(\ful{w_{1}\cdots w_{p+1}}w_{p+2})\\
                =\fE(w_{1}\fur{w_{2}\cdots w_{p+2}})\fE(\fll{w_{1}}w_{p+2})-\fE(w_{1}\fur{w_{2}\cdots w_{p+1}}\flr{w_{p+2}})\fE(\ful{w_{1}\cdots w_{p+1}}w_{p+2}).
                \end{multlined}
            \end{align}
            As a special case of \eqref{eq:re_explicit}, we can apply the expression $\fre_{2}(x,y)=\fE(x\flr{y})\fE(\ful{x}y)-\fE(x\fur{y})\fE(\fll{x}y)$ to the all of right-hand sides of the above equalities as
            \begin{align}
                \fE(w_{1})\fE(\ful{w_{1}\cdots w_{p+1}}w_{p+2})-2\fE(w_{1}\fur{w_{2}\cdots w_{p+2}})\fE(\fll{w_{1}}\ful{w_{2}\cdots w_{p+1}}w_{p+2})&=\fre_{2}(w_{1}\ful{w_{2}\cdots w_{p+1}}w_{p+2}),\\
                \fE(w_{1})\fE(w_{2}\fur{w_{3}\cdots w_{p+2}})-2\fE(w_{1}\fur{w_{2}\cdots w_{p+2}})\fE(\fll{w_{1}}w_{2}\fur{w_{3}\cdots w_{p+2}})&=\fre_{2}(w_{1}w_{2}\fur{w_{3}\cdots w_{p+2}}),\\
                \fE(\ful{w_{1}\cdots w_{p}}w_{p+1})\fE(w_{p+2})-2\fE(\ful{w_{1}\cdots w_{p}}w_{p+1}\flr{w_{p+2}})\fE(\ful{w_{1}\cdots w_{p+1}}w_{p+2})&=-\fre_{2}(\ful{w_{1}\cdots w_{p}}w_{p+1}w_{p+2}),\\
                \fE(w_{1}\fur{w_{2}\cdots w_{p+2}})\fE(w_{p+2})-2\fE(w_{1}\fur{w_{2}\cdots w_{p+1}}\flr{w_{p+2}})\fE(\ful{w_{1}\cdots w_{p+1}}w_{p+2})&=-\fre_{2}(w_{1}\fur{w_{2}\cdots w_{p+1}}w_{p+2}).
            \end{align}
            Thus we obtain
            \begin{align}
                &(2\fre_{p+2}+\lu(\fE,\fre_{p+1}))(w_{1}\cdots w_{p+2})\\
                &=\begin{multlined}[t]
            \fre_{2}(w_{1}\ful{w_{2}\cdots w_{p+1}}w_{p+2})\fre_{p}((w_{2}\cdots w_{p+1})\flr{w_{p+2}})
            -\fre_{2}(w_{1}w_{2}\fur{w_{3}\cdots w_{p+2}})\fre_{p}(\fll{w_{2}}(w_{2}\cdots w_{p+1}))\\
            +\fre_{2}(\ful{w_{1}\cdots w_{p}}w_{p+1}w_{p+2})\fre_{p}((w_{1}\cdots w_{p})\flr{w_{p+1}})
            -\fre_{2}(w_{1}\fur{w_{2}\cdots w_{p+1}}w_{p+2})\fre_{p}(\fll{w_{1}}(w_{2}\cdots w_{p+1}))
            \end{multlined}\\
            &=\arit(\fre_{p})(\fre_{2})(w_{1}\cdots w_{p+2}).
        \end{align}
    \end{proof}
    \begin{proof}[Proof of (2)]
        We shall prove the identity
        \[A(p+1,q)-D(p+1,q)=\arit(A(p,q))(\fE).\]
        From Proposition \ref{prop:arit_composition}, the right-hand side is
        \begin{align}
            \arit(A(p,q))(\fE)
            &=\arit(\ari(\fre_{p},\fre_{q})-(p-q)\fre_{p+q})(\fE)\\
            &=(\arit(\fre_{q})\circ\arit(\fre_{p})-\arit(\fre_{p})\circ\arit(\fre_{q})-(p-q)\fre_{p+q})(\fE)\\
            &=\arit(\fre_{q})(\fre_{p+1})-\arit(\fre_{p})(\fre_{q+1})-(p-q)\fre_{p+q+1}.
        \end{align}
        On the other hand, it follows that
        \begin{align}
            &A(p+1,q)-D(p+1,q)\\
            &=\begin{multlined}[t]\arit(\fre_{q})(\fre_{p+1})-\arit(\fre_{p+1})(\fre_{q})+\lu(\fre_{p+1},\fre_{q})-(p+1-q)\fre_{p+q+1}\\-\arit(\fre_{p})(\fre_{q+1})+\arit(\fre_{p+1})(\fre_{q})+\fre_{p+q+1}-\lu(\fre_{p+1},\fre_{q})\end{multlined}\\
            &=\arit(\fre_{q})(\fre_{p+1})-\arit(\fre_{p})(\fre_{q+1})-(p-q)\fre_{p+q+1}.
        \end{align}
        Thus both sides coincide.
    \end{proof}
    \begin{proof}[Proof of (3)]
        Assume that $E(p,q-1)=0$ and $A(i,1)=0$ hold for all $i\in\{1,2,\ldots,p+q-1\}$.
        After the computation
        \begin{align}
            &(q-1)\arit(\fre_{p})(\fre_{q})+(p-1)\arit(\fre_{p+1})(\fre_{q-1})\\
            &=\arit(\fre_{p})(\fre_{q})+\arit(\fre_{p})((q-2)\fre_{q})+\arit((p-1)\fre_{p+1})(\fre_{q-1})\\
            &=\arit(\fre_{p})(\fre_{q})+\arit(\fre_{p})(\ari(\fre_{q-1},\fE)-A(q-1,1))+\arit(\ari(\fre_{p},\fE)-A(p,1))(\fre_{q-1}),
        \end{align}
        we can delete $A(p-1,1)$ and $A(q,1)$ by the assumption, so that
        \begin{align}
            &(q-1)\arit(\fre_{p})(\fre_{q})+(p-1)\arit(\fre_{p+1})(\fre_{q-1})\\
            &=\arit(\fre_{p})(\fre_{q})+\arit(\fre_{p})(\ari(\fre_{q-1},\fE))+\arit(\ari(\fre_{p},\fE))(\fre_{q-1})\\
            &=\begin{multlined}[t]\arit(\fre_{p})(\fre_{q})+\arit(\fre_{p})(\arit(\fE)(\fre_{q-1})-\fre_{q}+\lu(\fre_{q-1},\fE))\\+(\arit(\fE)\circ\arit(\fre_{p})-\arit(\fre_{p})\circ\arit(\fE))(\fre_{q-1})\end{multlined}\\
            &=(\arit(\fE)\circ\arit(\fre_{p}))(\fre_{q-1})+\arit(\fre_{p})(\lu(\fre_{q-1},\fE)),
        \end{align}
        where we used Proposition \ref{prop:arit_composition} in the second equality.
    
    Since $\arit(X)$ is a derivation about $\lu$ (a consequence of Proposition \ref{prop:axit_derivation}), we have
    \begin{align}
        &(q-1)\arit(\fre_{p})(\fre_{q})+(p-1)\arit(\fre_{p+1})(\fre_{q-1})\\
        &=(\arit(\fE)\circ\arit(\fre_{p}))(\fre_{q-1})+\lu(\arit(\fre_{p})(\fre_{q-1}),\fE)+\lu(\fre_{q-1},\fre_{p+1})\\
        &=\begin{multlined}[t]\arit(\fE)\left(E(p,q-1)+(q-1)\fre_{p+q-1}+\sum_{i=1}^{q-2}\lu(\fre_{i},\fre_{p+q-1-i})\right)\\
            +\lu\left(\left(E(p,q-1)+(q-1)\fre_{p+q-1}+\sum_{i=1}^{q-2}\lu(\fre_{i},\fre_{p+q-1-i})\right),\fE\right)+\lu(\fre_{q-1},\fre_{p+1})\end{multlined}\\
        &=\begin{multlined}[t](q-1)\arit(\fE)(\fre_{p+q-1})+\lu(\fre_{q-1},\fre_{p+1})+(q-1)\lu(\fre_{p+q-1},\fE)\\
            +\sum_{i=1}^{q-2}\arit(\fE)\left(\lu(\fre_{i},\fre_{p+q-1-i})\right)+\sum_{i=1}^{q-2}\lu\left(\lu(\fre_{i},\fre_{p+q-1-i}),\fE\right)\end{multlined}\\
        &=\begin{multlined}[t](q-1)\arit(\fE)(\fre_{p+q-1})+\lu(\fre_{q-1},\fre_{p+1})+(q-1)\lu(\fre_{p+q-1},\fE)+\\\sum_{i=1}^{q-2}\left(\lu(\arit(\fE)(\fre_{i}),\fre_{p+q-1-i})+\lu(\fre_{i},\arit(\fE)(\fre_{p+q-1-i}))\right)+\sum_{i=1}^{q-2}\lu\left(\lu(\fre_{i},\fre_{p+q-1-i}),\fE\right).\end{multlined}
    \end{align}
    Here we used $E(p,q-1)=0$ in the third equality. 
    The first summation part can be deformed as
    \begin{align}
        &\sum_{i=1}^{q-2}\left(\lu(\arit(\fE)(\fre_{i}),\fre_{p+q-1-i})+\lu(\fre_{i},\arit(\fE)(\fre_{p+q-1-i}))\right)\\
        &=\begin{multlined}[t]\sum_{i=1}^{q-2}\left(\lu\left(A(i,1)+\fre_{i+1}-\lu(\fE,\fre_{i})+(i-1)\fre_{i+1},\fre_{p+q-1-i}\right)+\lu\left(\fre_{i},A(p+q-1-i,1)\right.\right.\\
            \left.\left.+\fre_{p+q-i}-\lu(\fE,\fre_{p+q-1-i}+(p+q-2-i)\fre_{p+q-i})\right)\right)\end{multlined}\\
        &=\sum_{i=1}^{q-2}\left(\lu\left(i\fre_{i+1}-\lu(\fE,\fre_{i}),\fre_{p+q-1-i}\right)+\lu\left(\fre_{i},(p+q-1-i)\fre_{p+q-i}-\lu(\fE,\fre_{p+q-1-i})\right)\right)\\
        &=\begin{multlined}[t]\sum_{i=1}^{q-2}\left(i\lu(\fre_{i+1},\fre_{p+q-1-i})+(p+q-1-i)\lu(\fre_{i},\fre_{p+q-i})\right)\\
            -\sum_{i=1}^{q-2}\left(\lu(\lu(\fE,\fre_{i}),\fre_{p+q-1-i})+\lu(\fre_{i},\lu(\fE,\fre_{p+q-1-i}))\right),\end{multlined}
    \end{align}
    where the second equality is a consequence of the assumption $A(\bullet,1)=0$.
    The first sum is telescopically calculated as
    \begin{align}
        &\sum_{i=1}^{q-2}\left(i\lu(\fre_{i+1},\fre_{p+q-1-i})+(p+q-1-i)\lu(\fre_{i},\fre_{p+q-i})\right)\\
        &=(p+q-2)\sum_{i=1}^{q-2}\lu(\fre_{i},\fre_{p+q-i})+\sum_{i=1}^{q-2}\left(i\lu(\fre_{i+1},\fre_{p+q-1-i})-(i-1)\lu(\fre_{i},\fre_{p+q-i})\right)\\
        &=(p+q-2)\sum_{i=1}^{q-2}\lu(\fre_{i},\fre_{p+q-i})+(q-2)\lu(\fre_{q-1},\fre_{p+1}),
    \end{align}
    while the second sum is equal to, by Jacobi's identity,
    \begin{align}
        &\sum_{i=1}^{q-2}\left(\lu(\lu(\fE,\fre_{i}),\fre_{p+q-1-i})+\lu(\fre_{i},\lu(\fE,\fre_{p+q-1-i}))\right)\\
        &=\sum_{i=1}^{q-2}\left(\lu(\lu(\fE,\fre_{i}),\fre_{p+q-1-i})+\lu(\lu(\fre_{p+q-1-i},\fE),\fre_{i})\right)\\
        &=-\sum_{i=1}^{q-2}\lu(\lu(\fre_{i},\fre_{p+q-1-i}),\fE).
    \end{align}  
    Combining them, we obtain
    \begin{align}
        &(q-1)\arit(\fre_{p})(\fre_{q})+(p-1)\arit(\fre_{p+1})(\fre_{q-1})\\
        &=\begin{multlined}[t](q-1)\arit(\fE)(\fre_{p+q-1})+\lu(\fre_{q-1},\fre_{p+1})+(q-1)\lu(\fre_{p+q-1},\fE)\\
            +\sum_{i=1}^{q-2}\left(\lu(\arit(\fE)(\fre_{i}),\fre_{p+q-1-i})+\lu(\fre_{i},\arit(\fE)(\fre_{p+q-1-i}))\right)+\sum_{i=1}^{q-2}\lu\left(\lu(\fre_{i},\fre_{p+q-1-i}),\fE\right)\end{multlined}\\
        &=\begin{multlined}[t](q-1)\arit(\fE)(\fre_{p+q-1})+\lu(\fre_{q-1},\fre_{p+1})+(q-1)\lu(\fre_{p+q-1},\fE)\\-\sum_{i=1}^{q-2}\lu\left(\lu(\fre_{i},\fre_{p+q-1-i}),\fE\right)+(p+q-2)\sum_{i=1}^{q-2}\lu(\fre_{i},\fre_{p+q-i})+(q-2)\lu(\fre_{q-1},\fre_{p+1})\\+\sum_{i=1}^{q-2}\lu(\lu(\fre_{i},\fre_{p+q-1-i}),\fE)\end{multlined}\\
        &=\begin{multlined}[t](q-1)\arit(\fE)(\fre_{p+q-1})+(q-1)\lu(\fre_{p+q-1},\fE)\\+(q-1)\lu(\fre_{q-1},\fre_{p+1})+(p+q-2)\sum_{i=1}^{q-2}\lu(\fre_{i},\fre_{p+q-i})\end{multlined}\\
        &=\begin{multlined}[t](q-1)(A(p+q-1,1)+\fre_{p+q}-\lu(\fre_{p+q-1},\fE)+(p+q-2)\fre_{p+q})\\
            +(q-1)\lu(\fre_{p+q-1},\fE)+(q-1)\lu(\fre_{q-1},\fre_{p+1})+(p+q-2)\sum_{i=1}^{q-2}\lu(\fre_{i},\fre_{p+q-i})\end{multlined}\\
        &=(q-1)(p+q-1)\fre_{p+q}+(q-1)\lu(\fre_{q-1},\fre_{p+1})+(p+q-2)\sum_{i=1}^{q-2}\lu(\fre_{i},\fre_{p+q-i}),
    \end{align}
    where the last equality is due to $A(\bullet,q)=0$.
    On the other hand, the left-hand side becomes
    \begin{align}
        &(q-1)\arit(\fre_{p})(\fre_{q})+(p-1)\arit(\fre_{p+1})(\fre_{q-1})\\
        &=\begin{multlined}[t](q-1)\left(E(p,q)+q\fre_{p+q}+\sum_{i=1}^{q-1}\lu(\fre_{i},\fre_{p+q-i})\right)\\+(p-1)\left(E(p+1,q-1)+(q-1)\fre_{p+q}+\sum_{i=1}^{q-2}\lu(\fre_{i},\fre_{p+q-i})\right)\end{multlined}\\
        &=\begin{multlined}[t](q-1)E(p,q)+(p-1)E(p+1,q-1)+(q-1)(p+q-1)\fre_{p+q}\\+(q-1)\lu(\fre_{q-1},\fre_{p+1})+(p+q-2)\sum_{i=1}^{q-2}\lu(\fre_{i},\fre_{p+q-i}).\end{multlined}
    \end{align}
    Therefore we obtain
    \begin{align}
        &(q-1)E(p,q)+(p-1)E(p+1,q-1)\\
        &=\begin{multlined}[t](q-1)(p+q-1)\fre_{p+q}+(q-1)\lu(\fre_{q-1},\fre_{p+1})+(p+q-2)\sum_{i=1}^{q-2}\lu(\fre_{i},\fre_{p+q-i})\\-(q-1)(p+q-1)\fre_{p+q}-(q-1)\lu(\fre_{q-1},\fre_{p+1})-(p+q-2)\sum_{i=1}^{q-2}\lu(\fre_{i},\fre_{p+q-i})\end{multlined}\\
        &=0.
    \end{align}
    This equality guarantees the desired equivalence.
\end{proof}
\end{proof}

\begin{proposition}\label{prop:he_lie}
    The continuous $R$-linear map $\fHe_{R}\colon\DIFF_{x}(R)\to\ARI(R)$ given by $\rre_{r}\coloneqq x^{r+1}\frac{d}{dx}\mapsto\fre_{r}$ is a homomorphism of Lie algebras.
\end{proposition}
\begin{proof}
    It is sufficient to show that $\fHe_{R}([\rre_{r},\rre_{s}])=\ari(\fre_{r},\fre_{s})$ for $r,s\ge 1$.
    Since the Lie structure on $\DIFF_{x}(R)$ is given by the anti-commutator (Proposition \ref{prop:diff}), we have
    \begin{align}
        [\rre_{r},\rre_{s}]
        &=x^{s+1}\frac{d}{dx}x^{r+1}\frac{d}{dx}-x^{r+1}\frac{d}{dx}x^{s+1}\frac{d}{dx}\\
        &=x^{s+1}\left((r+1)x^{r}\frac{d}{dx}+x^{r+1}\frac{d^{2}}{dx^{2}}\right)-x^{r+1}\left((s+1)x^{s}\frac{d}{dx}+x^{s+1}\frac{d^{2}}{dx^{2}}\right)\\
        &=(r-s)x^{r+s+1}\frac{d}{dx}\\
        &=(r-s)\rre_{r+s}.
    \end{align}
    On the other hand, Proposition \ref{prop:ecalle_47} says that the family $\{\fre_{r}\}_{r}$ satisfies the quite same relations.
    It means that $\fHe_{R}$ preserves the Lie brackets.
\end{proof}

\subsection{Some corollaries}
\begin{definition}\label{def:Se}
    We define\footnote{We slightly changed the notation from \'{E}calle's original one \cite[\S 4.1]{ecalle11}. He denoted it simply by $\fSe_{f}$, but we would like to reserve the symbol $\fSe$ for the morphism of schemes rather than that of rational points.} the morphism $\fSe\colon\GIFF_{x}\to\GARI$ of schemes as follows: it is what makes the following diagram of schemes commutative.
    \[\xymatrix{
        \GIFF_{x} \ar[r]^{\fSe} & \GARI\\
        \DIFF_{x} \ar[u]^{\exp} \ar[r]_{\fHe} & \ARI \ar[u]_{\expari}
    }\]    
\end{definition}

\begin{proposition}[{\cite[(4.1.6)]{schneps15}}]\label{prop:se_gari}
    Let $f$ and $g$ belong to $\GIFF_{x}(R)$.
    Then we have $\fSe_R(f\circ g)=\gari(\fSe_{R}(f),\fSe_{R}(g))$.
\end{proposition}
\begin{proof}
    This is a consequence of the diagram in Definition \ref{def:Se} and the fact that the horizontal lower arrow $\fHe$ is compatible with the Lie algebra structure (Proposition \ref{prop:he_lie}).
\end{proof}

\begin{proposition}\label{prop:se_symmetral}
    For any $f\in\GIFF_{x}$, the bimould $\fSe_{R}(f)$ is symmetral.
\end{proposition}
\begin{proof}
    This is a consequence of the fact that $\fre_{r}$ is alternal for every $r\ge 1$ and Proposition \ref{prop:expari_alternal}.
\end{proof}

\begin{proposition}\label{prop:neg_pari_se}
    For any $f\in\GIFF_{x}$, the bimould $\fSe_{R}(f)$ is $\neg\circ\pari$-invariant.
\end{proposition}
\begin{proof}
    As Lemma \ref{lem:neg_pari_axit} holds, we see that $\neg$ and $\pari$ are distributive for $\preari$, and thus they commute with $\expari$.
    Hence we show
    \begin{align}
        (\neg\circ\pari\circ\fSe_{R})(f)
        &=(\neg\circ\pari\circ\expari)\left(\sum_{r=1}^{\infty}\epsilon^{f}_{r}\fre_{r}\right)\\
        &=\expari\left(\sum_{r=1}^{\infty}\epsilon^{f}_{r}(\neg\circ\pari)(\fre_{r})\right).
    \end{align}
    So we check that $\fre_{r}$ is $\neg\circ\pari$-invariant for $r\ge 1$.
    For the $r=1$ case, $\fre_{1}$ is simply the unit $\fE$ and its invariance is nothing but a part of its definition (the parity condition).
    When $\fre_{r}$ is $\neg\circ\pari$-invariant, from Lemma \ref{lem:neg_pari_axit} we have
    \begin{align}
        (\neg\circ\pari)(\fre_{r+1})
        &=(\neg\circ\pari\circ\arit(\fre_{r}))(\fE)\\
        &=\arit((\neg\circ\pari)(\fre_{r})\circ\neg\circ\pari)(\fE)\\
        &=\arit(\fre_{r})(\fE)\\
        &=\fre_{r+1}.
    \end{align}
    Hence we met with success in the induction on $r$.
\end{proof}

\subsection{Separation lemma}
Let $\gepar$ denote the operator $A\mapsto\mmu((\anti\circ\swap)(A),\swap(A))$.
This subsection is devoted to show the following.

\begin{theorem}[\'{E}calle's \emph{separation lemma}; {\cite[(4.39)]{ecalle11}}]\label{thm:separation}
    Let $f$ belong to $\GIFF_{x}(R)$.
    Then we have
    \[\gepar(\fSe_{f})=\sum_{r=0}^{\infty}(r+1)a_{r}^{f}\mmu^{r}(\fO),\]
\end{theorem}

For each $f\in\GIFF_{x}(R)$, regarding $f_{\#}$ as an element of $\DIFF_{x}(R)$, we define
\[\fTe_{R}(f)\coloneqq\fHe_{R}(f_{\#})=\sum_{r=1}^{\infty}\gamma_{r}^{f}\fre_{r},\]
when $f_{\#}(x)$ is expanded as $\sum_{r=1}^{\infty}\gamma_{r}^{f}x^{r+1}$.

There is another useful operator $\der$: it acts on bimoulds as
\[\der(A)(w_{1},\ldots,w_{r})\coloneqq rA(w_{1},\ldots,w_{r}),\]
and obviously gives a derivation on the $R$-algebra $(\BIMU(R),\mmu)$.

\begin{lemma}[{\cite[(4.20)]{ecalle11}}]\label{lem:Se_schneps}
    For every $f\in\GIFF_{x}(R)$, we have
    \[(\der\circ\fSe_{R})(f)=\preari(\fSe_{R}(f),\fTe_{R}(f)).\]
\end{lemma}
\begin{proof}
    Remembering Remark \ref{rem:linearization}, it suffices to prove
    \begin{equation}\label{eq:Se_schneps}
        \fSe_{R}(f)+\ep(\der\circ\fSe_{R})(f)=\gari(\fSe_{R}(f),1+\ep\fTe_{R}(f)).
    \end{equation}
    Since it obviously holds that
    \[\exp\left(\ep f_{\#}(x)\frac{d}{dx}\right)(x)=x+\ep f_{\#}(x),\]
    we can show
    \begin{align}
        \fSe_{R}(\id+\ep f_{\#})
        &=(\expari\circ\fHe_{R})\left(\ep f_{\#}(x)\frac{d}{dx}\right)\\
        &=1+\ep\fHe_{R}\left(\ep f_{\#}(x)\frac{d}{dx}\right)\\
        &=1+\ep\fTe_{R}(f).
    \end{align}
    Therefore, by virtue of Proposition \ref{prop:se_gari}, the right-hand side of \eqref{eq:Se_schneps} becomes
    \begin{align}
        \gari(\fSe_{R}(f),1+\ep\fTe_{R}(f))
        &=\gari(\fSe_{R}(f),\fSe_{R}(\id+\ep f_{\#}))\\
        &=\fSe_{R}(f\circ(\id+\ep f_{\#})).
    \end{align}
    Computing the composition directly, we have
    \begin{align}
        (f\circ(\id+\ep f_{\#}))(x)
        &=x+\ep f_{\#}(x)+\sum_{r=1}^{\infty}a_{r}^{f}(x+\ep f_{\#}(x))^{r+1}\\
        &=x+\ep f_{\#}(x)+\sum_{r=1}^{\infty}a_{r}^{f}(x^{r+1}+\ep(r+1)x^{r}f_{\#}(x))\\
        &=f(x)+\ep(f_{\#}(x)+(f'(x)-1)f_{\#}(x)).
    \end{align}
    Then the definition of $f_{\#}(x)$ leads
    \[(f\circ(\id+\ep f_{\#}))(x)=f(x)+\ep(xf'(x)-f(x)).\]
    On the other hand, remembering that the exponential of an operator $g_{\ast}(x)\frac{d}{dx}$ in $\DIFF_{x}$ is generally given as
    \[\exp\left(g_{\ast}(x)\frac{d}{dx}\right)(x)=x+\sum_{n=1}^{\infty}\sum_{r_{1},\ldots,r_{n}\ge 1}\ep_{r_{1}}^{g}\cdots\ep_{r_{n}}^{g}(\rre_{r_{1}}\cdots\rre_{r_{n}})(x)\]
    with the notation $\rre_{r}\coloneqq x^{r+1}\frac{d}{dx}$, we observe that
    \begin{align}
        &x+\sum_{n=1}^{\infty}\sum_{r_{1},\ldots,r_{n}\ge 1}(r_{1}+\cdots+r_{n})\ep_{r_{1}}^{f}\cdots\ep_{r_{n}}^{f}(\rre_{r_{1}}\cdots\rre_{r_{n}})(x)\\
        &=x+\sum_{n=1}^{\infty}\sum_{r_{1},\ldots,r_{n}\ge 1}\ep_{r_{1}}^{f}\cdots\ep_{r_{n}}^{f}\left(x\frac{d}{dx}-\id\right)(\rre_{r_{1}}\cdots\rre_{r_{n}})(x)\\
        &=\left(x\frac{d}{dx}-\id\right)\left(\exp\left(f_{\ast}(x)\frac{d}{dx}\right)\right)\\
        &=\left(x\frac{d}{dx}-\id\right)(f(x))\\
        &=xf'(x)-f(x).
    \end{align}
    Therefore the right-hand side of \eqref{eq:Se_schneps} becomes
    \begin{align}
        \gari(\fSe_{R}(f),1+\ep\fTe_{R}(f))
        &=\fSe_{R}((f\circ(\id+\ep f_{\#})))\\
        &=1+\sum_{n=1}^{\infty}\sum_{r_{1},\ldots,r_{n}\ge 1}(r_{1}+\cdots+r_{n})\ep_{r_{1}}^{f}\cdots\ep_{r_{n}}^{f}\preari(\fre_{r_{1}},\ldots,\fre_{r_{n}}).
    \end{align}
    Since each $\fre_{r_{i}}$ belongs to $\BIMU_{r_{i}}$ and $\preari$ keeps them (namely, $\preari(\fre_{r_{1}},\ldots,\fre_{r_{n}})$ is an element of $\BIMU_{r_{1}+\cdots+r_{n}}$), we obtain
    \begin{align}
        \gari(\fSe_{R}(f),1+\ep\fTe_{R}(f))
        &=1+\sum_{n=1}^{\infty}\sum_{r_{1},\ldots,r_{n}\ge 1}\ep_{r_{1}}^{f}\cdots\ep_{r_{n}}^{f}\der(\preari(\fre_{r_{1}},\ldots,\fre_{r_{n}}))\\
        &=\der(\fSe_{f}).
    \end{align}
    This proves \eqref{eq:Se_schneps}.
\end{proof}

For $f\in\GIFF_{x}(R)$, we put $\dSo_{R}(f)\coloneqq\swap(\fSe_{R}(f))$ and $\dTo_{R}(f)\coloneqq\swap(\fTe_{R}(f))$.

\begin{corollary}\label{cor:ecalle_1195}
    For $f\in\GIFF_{x}(R)$, we have
    \[(\der\circ\dSo_{R})(f)=\iwat(\dTo_{R}(f))(\dSo_{R}(f))+\mmu(\dSo_{R}(f),\dTo_{R}(f)).\]
\end{corollary}
\begin{proof}
    Applying \eqref{eq:amit_swap} and \eqref{eq:anit_swap}, we have
    \begin{align}
        &\iwat(\dTo_{R}(f))(\dSo_{R}(f))+\mmu(\dSo_{R}(f),\dTo_{R}(f))\\
        &=\amit(\dTo_{R}(f))(\dSo_{R}(f))+\anit(\anti(\dTo_{R}(f)))(\dSo_{R}(f))+\mmu(\dSo_{R}(f),\dTo_{R}(f))\\
        &=\begin{multlined}[t]
            \swap(\amit(\fTe_{R}(f))(\fSe_{R}(f))+\mmu(\fSe_{R}(f),\fTe_{R}(f))-\mmu(\dSo_{R}(f),\dTo_{R}(f)))\\
            +\swap(\anit((\push\circ\swap\circ\anti)(\dTo_{R}(f)))(\fSe_{R}(f)))+\mmu(\dSo_{R}(f),\dTo_{R}(f))
        \end{multlined}\\
        &=\begin{multlined}[t](\swap\circ\amit(\fTe_{R}(f))\circ\fSe_{R})(f)+\swap(\mmu(\fSe_{R}(f),\fTe_{R}(f)))\\
            +(\swap\circ\anit((\neg\circ\anti)(\fTe_{R}(f)))\circ\fSe_{R})(f),\end{multlined}
    \end{align}
    where we used \eqref{eq:negpush} in the last equality.
    Since each $\fre_{r}$ is invariant under $\mantar$ and $\neg\circ\pari$ (indeed we showed the latter in the proof of Proposition \ref{prop:neg_pari_se}, and we already have its alternality), we see that
    \begin{align}
        (\neg\circ\anti)(\fTe_{R}(f))
        &=\sum_{r=1}^{\infty}\gamma_{r}^{f}(\neg\circ\anti)(\fre_{r})\\
        &=-\sum_{r=1}^{\infty}\gamma_{r}^{f}(\neg\circ\pari\circ\mantar)(\fre_{r})\\
        &=-\sum_{r=1}^{\infty}\gamma_{r}^{f}\fre_{r}\\
        &=-\fTe_{R}(f).
    \end{align}
    Therefore we obtain
    \begin{align}
        &\iwat(\dTo_{R}(f))(\dSo_{R}(f))+\mmu(\dSo_{R}(f),\dTo_{R}(f))\\
        &=(\swap\circ\amit(\fTe_{R}(f))\circ\fSe_{R})(f)+\swap(\mmu(\fSe_{R}(f),\fTe_{R}(f)))-(\swap\circ\anit(\fTe_{R}(f))\circ\fSe_{R})(f)\\
        &=\swap(\preari(\fSe_{R}(f),\fTe_{R}(f))).
    \end{align}
    The last term becomes $\der(\dSo_{R}(f))$, because Lemma \ref{lem:Se_schneps} holds and $\swap$ commutes with $\der$.
\end{proof}

\begin{lemma}[{\cite[(11.97)]{ecalle11}}]\label{lem:ecalle_1197}
    Let $f$ belong to $\GIFF_{x}(R)$ and put
    \[\dO_{\ast}(f)\coloneqq 1+\sum_{r=1}^{\infty}(r+1)a_{r}^{f}\leng_{r}(\foz).\]
    Then we have the identity
    \[\der(\dO_{\ast}(f))=\iwat(\dTo_{R}(f))(\dO_{\ast}(f))+\mmu(\dO_{\ast}(f),\dTo_{R}(f))+\mmu(\anti(\dTo_{R}(f)),\dO_{\ast}(f)).\]
\end{lemma}
\begin{proof}
    We observe that the right-hand side is decomposed as
    \begin{multline}
        \iwat(\dTo_{R}(f))(\dO_{\ast}(f))+\mmu(\dO_{\ast}(f),\dTo_{R}(f))+\mmu(\anti(\dTo_{R}(f)),\dO_{\ast}(f))\\
        =(\id+\anti)(\amit(\dTo_{R}(f))(\dO_{\ast}(f))+\mmu(\dO_{\ast}(f),\dTo_{R}(f))),
    \end{multline}
    because $\dO_{\ast}(f)$ is $\anti$-invariant and $\anti\circ\anit(\anti(\dTo_{R}(f)))$ is equal to $\amit(\dTo_{R}(f))\circ\anti$.
    So we shall compute $\amit(\dTo_{R}(f))(\dO_{\ast}(f))+\mmu(\dO_{\ast}(f),\dTo_{R}(f))$.
    By definition, we have
    \begin{align}
        &\amit(\dTo_{R}(f))(\dO_{\ast}(f))(\bw)+\mmu(\dO_{\ast}(f),\dTo_{R}(f))(\bw)\\
        &=\sum_{\substack{\bw=\ba\bb\bc\\ \bb\neq\emp}}\dO_{\ast}(f)(\ba\ful{\bb}\bc)\dTo_{R}(f)(\bb\flr{\bc})\\
        &=\sum_{\substack{\bw=\ba\bb\bc\\ \bb\neq\emp}}(\ell(\ba)+\ell(\bc)+1)a_{\ell(\ba)+\ell(\bc)}^{f}\gamma_{\ell(\bb)}^{f}\foz(\ba\ful{\bb}\bc)\dro_{\ell(\bb)}(\bb\flr{\bc}).
    \end{align}
    The explicit formula \eqref{eq:ecalle_459} for $\dro_{-}$ leads to
    \begin{align}
        &\amit(\dTo_{R}(f))(\dO_{\ast}(f))(\bw)+\mmu(\dO_{\ast}(f),\dTo_{R}(f))(\bw)\\
        &=\begin{multlined}[t]
            \sum_{\substack{\bw=\ba\bb\bc\\ \bb\neq\emp}}(\ell(\ba)+\ell(\bc)+1)a_{\ell(\ba)+\ell(\bc)}^{f}\gamma_{\ell(\bb)}^{f}\foz(\ba\ful{\bb}\bc)\\
            \cdot\sum_{\bb=\bd w_{i}\be}(\ell(\be)+1)\foz((\bd\flr{\bc})\flr{(w_{i}\flr{\bc})})\fO(\ful{(\bd\flr{\bc})}w_{i}\flr{\bc}\fur{(\be\flr{\bc})})\foz(\fll{(w_{i}\flr{\bc})}(\be\flr{\bc}))\end{multlined}\\
        &=\begin{multlined}[t]\sum_{\bw=\ba\bd w_{i}\be\bc}(\ell(\ba)+\ell(\bc)+1)(\ell(\be)+1)a_{\ell(\ba)+\ell(\bc)}^{f}\gamma_{\ell(\bd)+\ell(\be)+1}^{f}\\
            \cdot\foz(\ba\ful{\bd w_{i}\be}\bc)\foz(\bd\flr{w_{i}})\fO(\ful{\bd}w_{i}\flr{\bc}\fur{\be})\foz(\fll{w_{i}}\be).\end{multlined}
    \end{align}
    First we consider the terms with $\bc\neq\emp$ in the above sum.
    By putting $\bc=w_{j}\bc'$, it becomes
    \begin{align}
        &\sum_{\substack{\bw=\ba\bd w_{i}\be\bc\\ \bc\neq\emp}}(\ell(\ba)+\ell(\bc)+1)(\ell(\be)+1)a_{\ell(\ba)+\ell(\bc)}^{f}\gamma_{\ell(\bd)+\ell(\be)+1}^{f}\foz(\ba\ful{\bd w_{i}\be}\bc)\foz(\bd\flr{w_{i}})\fO(\ful{\bd}w_{i}\flr{\bc}\fur{\be})\foz(\fll{w_{i}}\be)\\
        &=\begin{multlined}[t]\sum_{\bw=\ba\bd w_{i}\be w_{j}\bc'}(\ell(\ba)+\ell(\bc')+2)(\ell(\be)+1)a_{\ell(\ba)+\ell(\bc')+1}^{f}\gamma_{\ell(\bd)+\ell(\be)+1}^{f}\\
            \cdot\foz(\ba)\fO(\ful{\bd w_{i}\be}w_{j})\foz(\bc')\foz(\bd\flr{w_{i}})\fO(\ful{\bd}w_{i}\flr{w_{j}\bc'}\fur{\be})\foz(\fll{w_{i}}\be).\end{multlined}
    \end{align}
    Applying here the tripartite identity as
    \begin{align}
        \fO(\ful{\bd}w_{i}\flr{w_{j}\bc'}\fur{\be})\fO(\ful{\bd w_{i}\be}w_{j})
        &=\fO((\ful{\bd}w_{i}\fur{\be})\flr{w_{j}})\fO(\ful{(\ful{\bd}w_{i}\fur{\be})}w_{j})\\
        &=\fO(\ful{\bd}w_{i}\fur{\be})\fO(w_{j})-\fO((\ful{\bd}w_{i}\fur{\be})\fur{w_{j}})\fO(\fll{(\ful{\bd}w_{i}\fur{\be})}w_{j})\\
        &=\fO(\ful{\bd}w_{i}\fur{\be})\fO(w_{j})-\fO(\ful{\bd}w_{i}\fur{\be w_{j}})\fO(\fll{w_{i}}w_{j}),
    \end{align}
    we obtain
    \begin{align}
        &\sum_{\substack{\bw=\ba\bd w_{i}\be\bc\\ \bc\neq\emp}}(\ell(\ba)+\ell(\bc)+1)(\ell(\be)+1)a_{\ell(\ba)+\ell(\bc)}^{f}\gamma_{\ell(\bd)+\ell(\be)+1}^{f}\foz(\ba\ful{\bd w_{i}\be}\bc)\foz(\bd\flr{w_{i}})\fO(\ful{\bd}w_{i}\flr{\bc}\fur{\be})\foz(\fll{w_{i}}\be)\\
        &=\begin{multlined}[t]\sum_{\bw=\ba\bd w_{i}\be w_{j}\bc'}(\ell(\ba)+\ell(\bc')+2)(\ell(\be)+1)a_{\ell(\ba)+\ell(\bc')+1}^{f}\gamma_{\ell(\bd)+\ell(\be)+1}^{f}\\
            \cdot\foz(\ba)\foz(\bd\flr{w_{i}})\left(\fO(\ful{\bd}w_{i}\fur{\be})\fO(w_{j})-\fO(\ful{\bd}w_{i}\fur{\be w_{j}})\fO(\fll{w_{i}}w_{j})\right)\foz(\fll{w_{i}}\be)\foz(\bc')\end{multlined}\\
        &=\begin{multlined}[t]\sum_{\substack{\bw=\ba\bd w_{i}\be\bc\\ \bc\neq\emp}}(\ell(\ba)+\ell(\bc')+1)(\ell(\be)+1)a_{\ell(\ba)+\ell(\bc)}^{f}\gamma_{\ell(\bd)+\ell(\be)+1}^{f}\foz(\ba)\foz(\bd\flr{w_{i}})\fO(\ful{\bd}w_{i}\fur{\be})\foz(\fll{w_{i}}\be)\foz(\bc)\\
            -\sum_{\substack{\bw=\ba\bd w_{i}\be'\bc'\\ \be'\neq\emp}}(\ell(\ba)+\ell(\bc')+2)\ell(\be')a_{\ell(\ba)+\ell(\bc')+1}^{f}\gamma_{\ell(\bd)+\ell(\be')}^{f}\foz(\ba)\foz(\bd\flr{w_{i}})\fO(\ful{\bd}w_{i}\fur{\be'})\foz(\fll{w_{i}}\be')\foz(\bc').
        \end{multlined}
    \end{align}
    Then we call back the $\bc=\emp$ terms as
    \begin{multline}
        \amit(\dTo_{R}(f))(\dO_{\ast}(f))(\bw)+\mmu(\dO_{\ast}(f),\dTo_{R}(f))(\bw)\\
        =\sum_{\bw=\ba\bd w_{i}\be\bc}\left((\ell(\ba)+\ell(\bc)+1)(\ell(\be)+1)a_{\ell(\ba)+\ell(\bc)}^{f}\gamma_{\ell(\bd)+\ell(\be)+1}^{f}\right.\\\left.-(\ell(\ba)+\ell(\bc)+2)\ell(\be)a_{\ell(\ba)+\ell(\bc)+1}^{f}\gamma_{\ell(\bd)+\ell(\be)}^{f}\right)\\
        \cdot\foz(\ba)\foz(\bd\flr{w_{i}})\fO(\ful{\bd}w_{i}\fur{\be})\foz(\fll{w_{i}}\be)\foz(\bc).
    \end{multline}
    Recalling the equality $\foz(\bd\flr{w_{i}})\fO(\ful{\bd}w_{i}\fur{\be})\foz(\fll{w_{i}}\be)=-\push^{\circ(\ell(\bd)+1)}(\foz)(\bd w_{i}\be)$, we have
    \begin{align}
        &\amit(\dTo_{R}(f))(\dO_{\ast}(f))(\bw)+\mmu(\dO_{\ast}(f),\dTo_{R}(f))(\bw)\\
        &=\begin{multlined}[t]-\sum_{\substack{\bw=\ba\bb\bc\\ \bb\neq\emp}}\sum_{i=1}^{\ell(\bb)}\left((\ell(\ba)+\ell(\bc)+1)(\ell(\bb)-i+1)a_{\ell(\ba)+\ell(\bc)}^{f}\gamma_{\ell(\bb)}^{f}\right.\\\left.-(\ell(\ba)+\ell(\bc)+2)(\ell(\bb)-i)a_{\ell(\ba)+\ell(\bc)+1}^{f}\gamma_{\ell(\bb)-1}^{f}\right)\\
            \cdot\foz(\ba)\push^{\circ i}(\foz)(\bb)\foz(\bc)\end{multlined}\\
        &=\begin{multlined}[t]-\sum_{\beta=1}^{r}\sum_{\alpha=0}^{r-\beta}\sum_{i=1}^{\beta}\left((r-\beta+1)(\beta-i+1)a_{r-\beta}^{f}\gamma_{\beta}^{f}-(r-\beta+2)(\beta-i)a_{r-\beta+1}^{f}\gamma_{\beta-1}^{f}\right)\\
            \cdot\foz(w_{1},\ldots,w_{\alpha})\push^{\circ i}(\foz)(w_{\alpha+1},\ldots,w_{\alpha+\beta})\foz(w_{\alpha+\beta+1},\ldots,w_{r})
        \end{multlined}\\
        &=\begin{multlined}[t]-\sum_{\beta=1}^{r}\sum_{\alpha=\beta}^{r}\sum_{i=1}^{\beta}\left((r-\beta+1)(\beta-i+1)a_{r-\beta}^{f}\gamma_{\beta}^{f}-(r-\beta+2)(\beta-i)a_{r-\beta+1}^{f}\gamma_{\beta-1}^{f}\right)\\\cdot\foz(w_{1},\ldots,w_{r-\alpha})\push^{\circ i}(\foz)(w_{r-\alpha+1},\ldots,w_{r-\alpha+\beta})\foz(w_{r-\alpha+\beta+1},\ldots,w_{r})\end{multlined}\\
        &=\begin{multlined}[t]-\sum_{1\le i\le \beta\le\alpha\le r}\left((r-\beta+1)(\beta-i+1)a_{r-\beta}^{f}\gamma_{\beta}^{f}-(r-\beta+2)(\beta-i)a_{r-\beta+1}^{f}\gamma_{\beta-1}^{f}\right)\\
            \cdot\foz(w_{1},\ldots,w_{r-\alpha})\push^{\circ i}(\foz)(w_{r-\alpha+1},\ldots,w_{r-\alpha+\beta})\foz(w_{r-\alpha+\beta+1},\ldots,w_{r}).\end{multlined}
    \end{align}
    Based on this expression, we can compute 
    \begin{align}
        &\anti(\amit(\dTo_{R}(f))(\dO_{\ast}(f))+\mmu(\dO_{\ast}(f),\dTo_{R}(f)))(\bw)\\
        &=\begin{multlined}[t]-\sum_{1\le i\le \beta\le\alpha\le r}\left((r-\beta+1)(\beta-i+1)a_{r-\beta}^{f}\gamma_{\beta}^{f}-(r-\beta+2)(\beta-i)a_{r-\beta+1}^{f}\gamma_{\beta-1}^{f}\right)\\
            \cdot\foz(w_{r},\ldots,w_{\alpha+1})\push^{\circ i}(\foz)(w_{\alpha},\ldots,w_{\alpha-\beta+1})\foz(w_{\alpha-\beta},\ldots,w_{1})\end{multlined}\\
        &=\begin{multlined}[t]-\sum_{1\le i\le \beta\le\alpha\le r}\left((r-\beta+1)(\beta-i+1)a_{r-\beta}^{f}\gamma_{\beta}^{f}-(r-\beta+2)(\beta-i)a_{r-\beta+1}^{f}\gamma_{\beta-1}^{f}\right)\\
            \cdot\foz(w_{1},\ldots,w_{\alpha-\beta})(\anti\circ\push^{\circ i})(\foz)(w_{\alpha-\beta+1},\ldots,w_{\alpha})\foz(w_{\alpha+1},\ldots,w_{r})\end{multlined}\\
        &=\begin{multlined}[t]-\sum_{1\le i\le \beta\le\alpha\le r}\left((r-\beta+1)(\beta-i+1)a_{r-\beta}^{f}\gamma_{\beta}^{f}-(r-\beta+2)(\beta-i)a_{r-\beta+1}^{f}\gamma_{\beta-1}^{f}\right)\\\cdot\foz(w_{1},\ldots,w_{\alpha-\beta})\push^{\circ (\beta+1-i)}(\foz)(w_{\alpha-\beta+1},\ldots,w_{\alpha})\foz(w_{\alpha+1},\ldots,w_{r}).\end{multlined}
    \end{align}
    Here, in the last equality, we used the fact that $(\anti\circ\push^{j})(M)=(\push^{l+1-j}\circ\anti)(M)$ always holds for any non-negative integer $l$ and $M\in\BIMU_{l}(R)$.
    Applying substitution $(\iota,\sigma)\coloneqq(\beta+1-i,r-\alpha+\beta)$, this sum becomes
    \begin{multline}
        \anti(\amit(\dTo_{R}(f))(\dO_{\ast}(f))+\mmu(\dO_{\ast}(f),\dTo_{R}(f)))(\bw)\\
        =-\sum_{1\le \iota\le \beta\le\sigma\le r}\left((r-\beta+1)\iota a_{r-\beta}^{f}\gamma_{\beta}^{f}-(r-\beta+2)(\iota-1)a_{r-\beta+1}^{f}\gamma_{\beta-1}^{f}\right)\\
        \cdot\foz(w_{1},\ldots,w_{r-\sigma})\push^{\circ\iota}(\foz)(w_{r-\sigma+1},\ldots,w_{r-\sigma+\beta})\foz(w_{r-\sigma+\beta+1},\ldots,w_{r}),
    \end{multline}
    and consequently
    \begin{align}
        &(\id+\anti)(\amit(\dTo_{R}(f))(\dO_{\ast}(f))+\mmu(\dO_{\ast}(f),\dTo_{R}(f)))(\bw)\\
        &=\begin{multlined}[t]-\sum_{1\le i\le \beta\le\alpha\le r}\left((r-\beta+1)(\beta+1)a_{r-\beta}^{f}\gamma_{\beta}^{f}-(r-\beta+2)(\beta-1)a_{r-\beta+1}^{f}\gamma_{\beta-1}^{f}\right)\\
            \cdot\foz(w_{1},\ldots,w_{\alpha-\beta})\push^{\circ i}(\foz)(w_{\alpha-\beta+1},\ldots,w_{\alpha})\foz(w_{\alpha+1},\ldots,w_{r})\end{multlined}\\
        &=\begin{multlined}[t]\sum_{1\le\beta\le r}\left((r-\beta+1)(\beta+1)a_{r-\beta}^{f}\gamma_{\beta}^{f}-(r-\beta+2)(\beta-1)a_{r-\beta+1}^{f}\gamma_{\beta-1}^{f}\right)\\
            \cdot\sum_{\alpha=\beta}^{r}\foz(w_{1},\ldots,w_{\alpha-\beta})\left(-\sum_{i=1}^{\beta}\push^{\circ i}(\foz)(w_{\alpha-\beta+1},\ldots,w_{\alpha})\right)\foz(w_{\alpha+1},\ldots,w_{r})\end{multlined}\\
        &=\begin{multlined}[t]\sum_{1\le\beta\le r}\left((r-\beta+1)(\beta+1)a_{r-\beta}^{f}\gamma_{\beta}^{f}-(r-\beta+2)(\beta-1)a_{r-\beta+1}^{f}\gamma_{\beta-1}^{f}\right)\\
            \cdot\sum_{\alpha=\beta}^{r}\foz(w_{1},\ldots,w_{\alpha-\beta})\foz(w_{\alpha-\beta+1},\ldots,w_{\alpha})\foz(w_{\alpha+1},\ldots,w_{r})\end{multlined}\\
        &=\sum_{1\le\beta\le r}(r-\beta+1)\left((r-\beta+1)(\beta+1)a_{r-\beta}^{f}\gamma_{\beta}^{f}-(r-\beta+2)(\beta-1)a_{r-\beta+1}^{f}\gamma_{\beta-1}^{f}\right)\foz(\bw),
    \end{align}
    where we used the $\push$-neutrality of $\foz$ (Proposition \ref{prop:unit_push} (3)) in the third equality.
    For the rest sum along $\beta$, we calculate it by taking the generating function as
    \begin{align}
        &\sum_{r\ge 1}\left(\sum_{1\le\beta\le r}(r-\beta+1)\left((r-\beta+1)(\beta+1)a_{r-\beta}^{f}\gamma_{\beta}^{f}-(r-\beta+2)(\beta-1)a_{r-\beta+1}^{f}\gamma_{\beta-1}^{f}\right)\right)x^{r}\\
        &=\sum_{\beta=1}^{\infty}\sum_{r'=0}^{\infty}(r'+1)((r'+1)(\beta+1)a_{r'}^{f}\gamma_{\beta}^{f}-(r'+2)(\beta-1)a_{r'+1}^{f}\gamma_{\beta-1}^{f})x^{\beta+r'}\\
        &=\left(\sum_{r'=0}^{\infty}(r'+1)^{2}a_{r'}^{f}x^{r'}\right)\left(\sum_{\beta=1}^{\infty}(\beta+1)\gamma_{\beta}^{f}x^{\beta}\right)-\left(\sum_{r'=0}^{\infty}(r'+1)(r'+2)a_{r'+1}^{f}x^{r'}\right)\left(\sum_{\beta=1}^{\infty}(\beta-1)\gamma_{\beta-1}^{f}x^{\beta}\right)\\
        &=\left(\frac{d}{dx}x\frac{d}{dx}f(x)\right)\left(\frac{d}{dx}f_{\#}(x)\right)-\left(\frac{d^{2}}{dx^{2}}f(x)\right)\left(x^{2}\frac{d}{dx}\frac{f_{\#}(x)}{x}\right)\\
        &=\frac{d}{dx}\left(f'(x)f_{\#}(x)\right)\\
        &=\frac{d}{dx}\left(xf'(x)-f(x)\right)\\
        &=xf''(x)\\
        &=\sum_{r=1}^{\infty}r(r+1)a_{r}x^{r}.
    \end{align}
    Therefore we obtain
    \begin{align}
        (\id+\anti)(\amit(\dTo_{R}(f))(\dO_{\ast}(f))+\mmu(\dO_{\ast}(f),\dTo_{R}(f)))(\bw)
        &=r(r+1)a_{r}\foz(\bw)\\
        &=\der(\dO_{\ast}(f))(\bw).
    \end{align}
\end{proof}
Now we can give a proof for Theorem \ref{thm:separation}.
\begin{proof}
    First we note that $\iwat(X)$ is always derivation about $\mmu$ (Proposition \ref{prop:axit_derivation}) and commutes with $\anti$: indeed, we see that
    \begin{align}
    \anti\circ\iwat(X)\circ\anti
    &=\anti\circ\amit(X)\circ\anti+\anti\circ\anit(\anti(X))\circ\anti\\
    &=\anit(\anti(X))+\amit(X)\\
    &=\iwat(X),
    \end{align}
    for any $X\in\LU(R)$.
    Consequently, we obtain
    \begin{align}
        &(\iwat(\dTo_{R}(f))\circ\gepar)(\fSe_{R}(f))\\
        &=\iwat(\dTo_{R}(f))(\mmu(\anti(\dSo_{R}(f)),\dSo_{R}(f)))\\
        &=\mmu((\iwat(\dTo_{R}(f))\circ\anti)(\dSo_{R}(f)),\dSo_{R}(f))+\mmu(\anti(\dSo_{R}(f)),\iwat(\dTo_{R}(f))(\dSo_{R}(f)))\\
        &=\mmu((\anti\circ\iwat(\dTo_{R}(f)))(\dSo_{R}(f)),\dSo_{R}(f))+\mmu(\anti(\dSo_{R}(f)),\iwat(\dTo_{R}(f))(\dSo_{R}(f))).
    \end{align}
    By using Corollary \eqref{cor:ecalle_1195} here, it follows that
    \begin{align}
        &(\iwat(\dTo_{R}(f))\circ\gepar)(\fSe_{R}(f))\\
        &=\begin{multlined}[t]\mmu(\anti(\der(\dSo_{R}(f))-\mmu(\dSo_{R}(f),\dTo_{R}(f))),\dSo_{R}(f))\\+\mmu(\anti(\dSo_{R}(f)),\der(\dSo_{R}(f))-\mmu(\dSo_{R}(f),\dTo_{R}(f)))\end{multlined}\\
        &=\begin{multlined}[t]\mmu((\der\circ\anti)(\dSo_{R}(f)),\dSo_{R}(f))+\mmu(\anti(\dSo_{R}(f)),\der(\dSo_{R}(f)))\\-\mmu(\anti(\dTo_{R}(f)),\anti(\dSo_{R}(f)),\dSo_{R}(f))-\mmu(\anti(\dSo_{R}(f)),\dSo_{R}(f),\dTo_{R}(f))\end{multlined}\\
        &=\der(\gepar(\fSe_{R}(f)))-\mmu(\anti(\dTo_{R}(f)),\gepar(\fSe_{R}(f)))-\mmu(\gepar(\fSe_{R}(f)),\dTo_{R}(f)).
    \end{align}
    Thus we have the equality
    \begin{multline}\der(\gepar(\fSe_{R}(f)))\\=(\iwat(\dTo_{R}(f))\circ\gepar)(\fSe_{R}(f))+\mmu(\gepar(\fSe_{R}(f)),\dTo_{R}(f))\\
        +\mmu(\anti(\dTo_{R}(f)),\gepar(\fSe_{R}(f))).\end{multline}
    Because this equation and the initial condition $\emp\mapsto 1$ determines $\gepar(\fSe_{R}(f))$, combining it with Lemma \ref{lem:ecalle_1197}, we obtain the desired result.
\end{proof}

\section{Secondary bimoulds}\label{sec:secondary}
\subsection{Definition}
Moreover, we define a bimould $\fess$ which \'{E}calle called a \emph{secondary bimould} and are useful for transportation of dimorphy.
For construction of them, the machinery of $\GIFF_{x}$ is very important (see Section \ref{sec:appendix}). 

\begin{definition}[{\cite[\S 4.2]{ecalle11}}]
    Let $\re(x)\coloneqq 1-\exp(-x)$ be the element of $\GIFF_{x}(\bbK)$.
    Then we define $\fess\coloneqq\fSe_{\re}$.
    Similarly, using $\fO$ instead of $\fE$ in the above definition, we define $\foss$.
    Denote by $\doss$ (resp.~$\dess$) the swappee of $\fess$ (resp.~$\foss$).
\end{definition}

\begin{remark}
    In the original paper \cite{ecalle11}, \'{E}calle seems to usually includes $\fe$ (resp.~$\fo$) in the notation which stands for some bimoulds constructed from $\fE$ (resp.~$\fO$), while he uses $\ddot{\fo}$ (resp.~$\ddot{\fe}$) in the same position in the notation when it indicates the swappee of a bimould whose name includes $\fe$ (resp.~$\fo$).
    We basically adopt such notations in this paper. 
\end{remark}

The following is almost obvious but very useful.
\begin{lemma}\label{prop:neg_pari_ess}
    Both $\fess$ and $\dess$ are $\neg\circ\pari$-invariant.
\end{lemma}
\begin{proof}
    Because $\swap$ commutes with $\neg$ and $\pari$, the invariance for $\dess$ follows from that of $\foss$. The rest is a direct consequence of Proposition \ref{prop:neg_pari_se}.
\end{proof}

Define an operator $\slash\colon\GARI(R)\to\GARI(R)$ by
\[\slash(A)\coloneqq\fragari(\neg(A),A)=\gari(\neg(A),\invgari(A)).\] 

\begin{proposition}[{\cite[(4.85)]{ecalle11}}]\label{prop:slash_pil}
    We have $\slash(\fess)=\fes$.
\end{proposition}
\begin{proof}
    For a power series $\id(x)=x$, we get $\epsilon^{\id}_{r}=0$ for all $r\ge 1$ and therefore $\fSe_{\id}=\expari(0)=1$.
    Thus it holds that $\fSe_{f^{-1}}=\invgari(\fSe_{f})$ for any $f(x)\in\GIFF_{x}(R)$.
    In other words, for such $f$ and $g$, we have $\fSe_{f\circ g^{-1}}=\fragari(\fSe_{f},\fSe_{g})$.
    Next, we want to find the power series $g^{-}$ such that $\fSe_{g^{-}}=\neg(\fSe_{g})$ for each $g$.
    If we get such $g^{-}$, we find that $\fSe_{g^{-}\circ g^{-1}}=\slash(\fSe_{g})$ holds.
    Since the identity $\preari(\neg(A),\neg(B))=\neg(\preari(A,B))$ holds for any $A\in\BIMU(R)$ and $B\in\LU(R)$, we have $\neg\circ\expari=\expari\circ\neg$.
    Therefore
    \begin{align}
        \neg(\fSe_{g})
        &=(\neg\circ\expari)\left(\sum_{r=1}^{\infty}\epsilon_{r}^{g}\fre_{r}\right)\\
        &=\expari\left(\sum_{r=1}^{\infty}\epsilon_{r}^{g}\neg(\fre_{r})\right).
    \end{align}
    On the other hand, we have the $\neg\circ\pari$-invariance of each $\fre_{r}$ (the proof of Proposition \ref{prop:neg_pari_se}).
    Therefore, for the power series $g^{-}$ we mentioned above, it should satisfy $\epsilon_{r}^{g^{-}}=(-1)^{r}\epsilon_{r}^{g}$.
    Applying this, we have
    \begin{align}
        g^{-}(x)
        &=\exp\left(\sum_{r=1}^{\infty}\epsilon_{r}^{g^{-}}x^{r+1}\frac{d}{dx}\right)(x)\\
        &=\exp\left(\sum_{r=1}^{\infty}\epsilon_{r}^{g}(-x)^{r+1}\frac{d}{d(-x)}\right)(x)\\
        &=x-\left(\exp\left(\sum_{r=1}^{\infty}\epsilon_{r}^{g}(-x)^{r+1}\frac{d}{d(-x)}\right)(-x)+x\right)\\
        &=-g(-x).
    \end{align}
    Then, putting $g(x)=\re(x)=1-e^{-x}$, we find that
    \[s(g)(x)\coloneqq(g^{-}\circ g^{-1})(x)=-1+\exp(-\log(1-x))=\frac{x}{1-x}.\]
    For this series $s(g)$, we immediately find $\epsilon^{s(g)}_{r}=0$ for all $r\ge 2$ and $\epsilon_{1}^{s(g)}=1$. Indeed, we have
    \begin{align}
        \frac{x}{1-x}
        &=\sum_{r=0}^{\infty}x^{r+1}\\
        &=\sum_{r=0}^{\infty}\frac{1}{r!}\left(x^{2}\frac{d}{dx}\right)^{r}(x)\\
        &=\exp\left(x^{2}\frac{d}{dx}\right)(x).
    \end{align}
    It yields
    \[\slash(\fess)=\slash(\fSe_{g})=\fSe_{s(g)}=\expari(\fre_{1})=\expari(\fE)=\fes.\]
\end{proof}

\begin{proposition}\label{prop:crash_pal}
    We have $\crash(\dess)=\fez$.
\end{proposition}
\begin{proof}
    We know that $\invgari(\foss)$ is $\gantar$-invariant by virtue of Propositions \ref{prop:se_symmetral} and \ref{prop:mantar_gantar}.
    Therefore we have
    \begin{align}
        (\push\circ\swap\circ\invmu\circ\invgari\circ\swap)(\dess)
        &=(\push\circ\swap\circ\invmu\circ\invgari)(\foss)\\
        &=(\push\circ\swap\circ\pari\circ\anti\circ\invgari)(\foss),
    \end{align}
    which is deformed by the identity \eqref{eq:negpush} as
\begin{align}\label{eq:rash_dess}
    \begin{split}
    (\push\circ\swap\circ\invmu\circ\invgari\circ\swap)(\dess)
    &=(\anti\circ\swap\circ\neg\circ\pari\circ\invgari)(\foss)\\
    &=(\anti\circ\swap\circ\invgari\circ\neg\circ\pari)(\foss)\\
    &=(\anti\circ\swap\circ\invgari)(\foss).
    \end{split}
    \end{align}
    Here we used the $\neg\circ\pari$-invariance of $\foss$ (Proposition \ref{prop:neg_pari_ess}) in the third equality, and the commutativity between $\neg\circ\pari$ and $\invgari$, which is a corollary of Lemma \ref{lem:neg_pari_gaxit}, in the second equality.
    Hence we get
    \begin{align}
        \crash(\dess)
        &=\mmu((\anti\circ\swap\circ\invgari)(\foss),(\swap\circ\invgari)(\foss))\\
        &=(\gepar\circ\invgari)(\foss).
    \end{align}
    By Proposition \ref{prop:se_gari} and the definition of $\foss$, we have $\invgari(\foss)=\fSe_{g}$ with $g(x)=-\log(1-x)$. Therefore the separation lemma (Theorem \ref{thm:separation}) shows $(\gepar\circ\invgari)(\foss)=\fez$.\\
\end{proof}

\subsection{$\gantar$-invariance of $\doss$}
In this subsection, we give a proof of the identity $\gantar(\doss)=\doss$.
The following assertion is a key tool to execute calculation, which extends the result of \cite[(244)]{ecalle15}.

\begin{proposition}\label{prop:244}
    For any $f\in\GIFF_{x}(R)$, we have the equality
    \[(-\der+\irat(\dTo_{R}(f)))\circ\ganit(\dO_{\ast}(f))=\ganit(\dO_{\ast}(f))\circ(-\der+\arit(\ganit(\dO_{\ast}(f))^{-1}(\dTo_{R}(f))))\]
    between operators on $\BIMU(R)$. 
\end{proposition}
\begin{proof}
    The statement is equivalent to 
    \begin{multline}\label{eq:244_rephrased1}
        (\id+\ep\irat(\dTo_{R}(f))-\ep\der)\circ\ganit(\dO_{\ast}(f))\\
        =\ganit(\dO_{\ast}(f))\circ(\id-\ep\der+\ep\arit(\ganit(\dO_{\ast}(f))^{-1}(\dTo_{R}(f))))
    \end{multline}
    between operators on $\BIMU(R[\ep])=\BIMU(R\jump{t}/(t^{2}))$.
    The ordinary linearization procedure shows $\id+\ep\irat(S)=\girat(1+\ep S)$ and $\id+\ep\arit(S)=\garit(1+\ep S)$ for any $S\in\LU(R)$.
    Thus \eqref{eq:244_rephrased1} is rephrased as
    \begin{multline}\label{eq:244_rephrased2}
        \girat(1+\ep\dTo_{R}(f))\circ\exp(-\ep\der)\circ\ganit(\dO_{\ast}(f))\\=\ganit(\dO_{\ast}(f))\circ\garit(\ganit(\dO_{\ast}(f))^{-1}(1+\ep\dTo_{R}(f)))\circ\exp(-\ep\der).
    \end{multline}
    By definition, $\exp(-\ep\der)$ acts on a bimould $S$ as
    \[\exp(-\ep\der)(S)=\sum_{r=0}^{\infty}e^{-\ep r}\leng_{r}(S).\]
    Hence we have the equality
    \[\exp(-\ep\der)\circ\ganit(\dO_{\ast}(f))=\ganit(\exp(-\ep\der)(\dO_{\ast}(f)))\circ\exp(-\ep\der),\]
    and \eqref{eq:244_rephrased2} becomes
    \begin{multline}\label{eq:244_rephrased3}
        \girat(1+\ep\dTo_{R}(f))\circ\ganit(\exp(-\ep\der)(\dO_{\ast}(f)))\\=\ganit(\dO_{\ast}(f))\circ\garit(\ganit(\dO_{\ast}(f))^{-1}(1+\ep\dTo_{R}(f))).
    \end{multline}
    We shall prove this identity.
    First, we remark that $-\push(\dTo_{R}(f))=\anti(\dTo_{R}(f))$ since $\dTo_{R}(f)$ is invariant under the operator
    \[-\push\circ\anti=\pari\circ\mantar\circ\push=\neg\circ\pari\circ\swap\circ\mantar\circ\swap,\]
    so that $\irat(\dTo_{R}(f))=\iwat(\dTo_{R}(f))$.
    Then, due to Proposition \ref{prop:gaxit_associative}, we see that the left-hand side becomes
    \begin{align}
        &\girat(1+\ep\dTo_{R}(f))\circ\ganit(\exp(-\ep\der)(\dO_{\ast}(f)))\\
        &=\gaxit(1+\ep\dTo_{R}(f),(\push\circ\swap\circ\invmu\circ\swap)(1+\ep\dTo_{R}(f)))\circ\gaxit(1,\exp(-\ep\der)(\dO_{\ast}(f)))\\
        &=\gaxit(1+\ep\dTo_{R}(f),1+\ep\anti(\dTo_{R}(f)))\circ\gaxit(1,\exp(-\ep\der)(\dO_{\ast}(f)))\\
        &=\gaxit(\gaxi((1,\exp(-\ep\der)(\dO_{\ast}(f))),(1+\ep\dTo_{R}(f),1+\ep\anti(\dTo_{R}(f)))))\\
        &=\begin{multlined}[t]\gaxit\left(\mmu(\gaxit(1+\ep\dTo_{R}(f),1+\ep\anti(\dTo_{R}(f)))(1),1+\ep\dTo_{R}(f)),\right.\phantom{\dO}\\\left.\mmu(1+\ep\anti(\dTo_{R}(f)),\gaxit(1+\ep\dTo_{R}(f),1+\ep\anti(\dTo_{R}(f)))(\exp(-\ep\der)(\dO_{\ast}(f))))\right)\end{multlined}\\
        &=\gaxit(1+\ep\dTo_{R}(f),\mmu(1+\ep\anti(\dTo_{R}(f)),(\girat(1+\ep\dTo_{R}(f)\circ\exp(-\ep\der))(\dO_{\ast}(f))))).
    \end{align}
    Moreover, by using $\ep^{2}=0$, the right component is equal to
    \begin{align}
        &\mmu(1+\ep\anti(\dTo_{R}(f)),(\girat(1+\ep\dTo_{R}(f)\circ\exp(-\ep\der))(\dO_{\ast}(f))))\\
        &=\mmu(1+\ep\anti(\dTo_{R}(f)),((\id+\ep\irat(\dTo_{R}(f)))\circ(\id-\ep\der))(\dO_{\ast}(f)))\\
        &=\mmu(1+\ep\anti(\dTo_{R}(f)),((\id+\ep\iwat(\dTo_{R}(f)))\circ(\id-\ep\der))(\dO_{\ast}(f)))\\
        &=\dO_{\ast}(f)+\ep(\mmu(\anti(\dTo_{R}(f)),\dO_{\ast}(f))+\iwat(\dTo_{R}(f))(\dO_{\ast}(f))-\der(\dO_{\ast}(f)))\\
        &=\mmu(\dO_{\ast}(f),1-\ep\dTo_{R}(f)),
    \end{align}
    where we used Lemma \ref{lem:ecalle_1197} in the last equality.
    Thus, using Propositions \ref{prop:gaxit_associative}, \ref{prop:gaxit_mu} and Corollary \ref{cor:gaxit_separation}, we obtain
    \begin{align}
        &\girat(1+\ep\dTo_{R}(f))\circ\ganit(\exp(-\ep\der)(\dO_{\ast}(f)))\\
        &=\gaxit(1+\ep\dTo_{R}(f),\mmu(\dO_{\ast}(f),1-\ep\dTo_{R}(f)))\\
        &=\gaxit(1+\ep\dTo_{R}(f),\mmu(\dO_{\ast}(f),(\ganit(\dO_{\ast}(f))\circ\ganit(\dO_{\ast}(f))^{-1})(1-\ep\dTo_{R}(f))))\\
        &=\gaxit(1+\ep\dTo_{R}(f),\gani(\ganit(\dO_{\ast}(f))^{-1}(1-\ep\dTo_{R}(f)),\dO_{\ast}(f)))\\
        &=\begin{multlined}[t]\ganit(\gani(\ganit(\dO_{\ast}(f))^{-1}(1-\ep\dTo_{R}(f)),\dO_{\ast}(f)))\\\circ\gamit(\ganit(\gani(\ganit(\dO_{\ast}(f))^{-1}(1-\ep\dTo_{R}(f)),\dO_{\ast}(f)))^{-1}(1+\ep\dTo_{R}(f)))\end{multlined}\\
        &=\begin{multlined}[t]\ganit(\dO_{\ast}(f))\circ\ganit(\ganit(\dO_{\ast}(f))^{-1}(1-\ep\dTo_{R}(f)))\\
            \circ\gamit((\ganit(\ganit(\dO_{\ast}(f))^{-1}(1-\ep\dTo_{R}(f)))^{-1}\circ\ganit(\dO_{\ast}(f))^{-1})(1+\ep\dTo_{R}(f)))\end{multlined}\\
        &=\ganit(\dO_{\ast}(f))\circ\gaxit(\ganit(\dO_{\ast}(f))^{-1}(1+\ep\dTo_{R}(f)),\ganit(\dO_{\ast}(f))^{-1}(1-\ep\dTo_{R}(f)))\\
        &=\ganit(\dO_{\ast}(f))\circ\gaxit(\ganit(\dO_{\ast}(f))^{-1}(1+\ep\dTo_{R}(f)),(\ganit(\dO_{\ast}(f))^{-1}\circ\invmu)(1+\ep\dTo_{R}(f)))\\
        &=\ganit(\dO_{\ast}(f))\circ\gaxit(\ganit(\dO_{\ast}(f))^{-1}(1+\ep\dTo_{R}(f)),(\invmu\circ\ganit(\dO_{\ast}(f))^{-1})(1+\ep\dTo_{R}(f)))\\
        &=\ganit(\dO_{\ast}(f))\circ\garit(\ganit(\dO_{\ast}(f))^{-1}(1+\ep\dTo_{R}(f))).
    \end{align}
    This shows \eqref{eq:244_rephrased3}.
\end{proof}

\begin{lemma}\label{lem:dilator_mantar}
    Let $S\in\MU(R)$ and $D\in\LU(R)$ be bimoulds related as
    \begin{equation}\label{eq:dilator}
        \der(S)=\preari(S,D).
    \end{equation}
    Then $S$ is $\gantar$-invariant if and only if $D$ is $\mantar$-invariant.
\end{lemma}
\begin{proof}
    The assumption $\der(S)=\preari(S,D)$ is rewritten as
    \[\exp(\ep\der)(S)=\gari(S,1+\ep D)\]
    which is an equality in $\BIMU(R[\ep])$.
    Since $\gantar$ preserves $\gari$ (Lemma \ref{lem:neg_pari_gaxit}) and it is linearized as
    \[\gantar(1+\ep D)=(\pari\circ\anti)(1-\ep D)=1+\ep\mantar(D),\]
    we obtain
    \[(\exp(\ep\der)\circ\gantar)(S)=\gari(\gantar(S),1+\ep\mantar(D)).\]
    Here we also used the commutativity of $\exp(\ep\der)$ and $\gantar$, which immediately follows from the definitions of them.
    If we assume $\gantar(S)=S$, this equality becomes
    \[\der(S)=\preari(S,\mantar(D)),\]
    and thus we have $\mantar(D)=D$ by comparing it and \eqref{eq:dilator}.
    Conversely, if $D$ is $\mantar$-invariant, we obtain the equality
    \[\der(\gantar(S))=\preari(\gantar(S),D),\]
    which determines the bimould $\gantar(S)$ length by length, from $D$.
    Hence, combining with \eqref{eq:dilator}, we get $\gantar(S)=S$.
\end{proof}

\begin{lemma}\label{lem:darapal}
    The bimould $\ganit(\foz)^{-1}(\dTo_{R}(\re^{-1}))$ is $\mantar$-invariant.
\end{lemma}
\begin{proof}
    To make our argument clear, most of this proof deals with a general $f\in\GIFF_{x}(R)$.
    First we compute $\ganit(\foz)\circ\mantar\circ\ganit(\foz)^{-1}$: this operator is deformed as
    \begin{align}
        &\ganit(\foz)\circ\mantar\circ\ganit(\foz)^{-1}\\
        &=-\pari\circ\ganit(\pari(\foz))\circ\anti\circ\ganit(\foz)^{-1}\\
        &=-\pari\circ\anti\circ\gamit((\pari\circ\anti)(\foz))\circ\ganit(\invgani(\foz))
    \end{align}
    Since the equality
    \begin{align}
    \invgani(\foz)
    &=(\gantar\circ\invgani(\foz))\\
    &=(\pari\circ\anti\circ\ganit(\invgani(\foz)))(\foz)\\
    &=(\pari\circ\anti\circ\ganit(\foz)^{-1}\circ\anti)(\foz)\\
    &=(\pari\circ\gamit(\anti(\foz))^{-1})(\foz)\\
    &=\gamit((\pari\circ\anti)(\foz))^{-1}(\pari(\foz))
    \end{align}
    holds due to Proposition \ref{prop:ez_and_es}, we obtain
    \begin{align}
        &\ganit(\foz)\circ\mantar\circ\ganit(\foz)^{-1}\\
        &=\mantar\circ\gamit((\pari\circ\anti)(\foz))\circ\ganit(\invgani(\foz))\\
        &=\mantar\circ\gamit((\pari\circ\anti)(\foz))\circ\ganit(\gamit((\pari\circ\anti)(\foz))^{-1}(\pari(\foz)))\\
        &=\mantar\circ\gaxit(\pari(\foz),\pari(\foz)).
    \end{align}
    Here we used Corollary \ref{cor:gaxit_separation} in the last equality.
    Hence, the $\mantar$-invariance
    \[\ganit(\foz)^{-1}(\dTo_{R}(f))=(\mantar\circ\ganit(\foz)^{-1})(\dTo_{R}(f))\]
    is equivalent to the equality
    \begin{equation}\label{eq:0623}
        \mantar(\dTo_{R}(f))=\gaxit(\pari(\foz),\pari(\foz))(\dTo_{R}(f)).
    \end{equation}
    On the other hand, from Lemmas \ref{lem:gaxit_simple} and \ref{lem:ecalle_459}, we can compute as
    \begin{align}
        &\gaxit(\pari(\foz),\pari(\foz))(\dTo_{R}(f))(\bw)\\
        &=\sum_{s=1}^{\infty}\sum_{\{(\bA_{j};w_{i_{j}};\bC_{j})\}_{j=1}^{s}\in\Sigma_{s}(\bw)}\dTo_{R}(f)(\ful{\bA_{1}}w_{i_{1}}\fur{\bC_{1}}\cdots \ful{\bA_{s}}w_{i_{s}}\fur{\bC_{s}})\prod_{j=1}^{s}\pari(\foz)(\bA_{j}\flr{w_{i_{j}}})\pari(\foz)(\fll{w_{i_{j}}}\bC_{j})\\
        &=\sum_{s=1}^{\infty}\sum_{\{(\bA_{j};w_{i_{j}};\bC_{j})\}_{j=1}^{s}\in\Sigma_{s}(\bw)}\gamma_{s}^{f}\dro_{s}(\ful{\bA_{1}}w_{i_{1}}\fur{\bC_{1}}\cdots \ful{\bA_{s}}w_{i_{s}}\fur{\bC_{s}})\prod_{j=1}^{s}\pari(\foz)(\bA_{j}\flr{w_{i_{j}}})\pari(\foz)(\fll{w_{i_{j}}}\bC_{j})\\
        &=\begin{multlined}[t]\sum_{s=1}^{\infty}\sum_{\{(\bA_{j};w_{i_{j}};\bC_{j})\}_{j=1}^{s}\in\Sigma_{s}(\bw)}\gamma_{s}^{f}\sum_{k=1}^{s}(s+1-k)\foz((\ful{\bA_{1}}w_{i_{1}}\fur{\bC_{1}}\cdots \ful{\bA_{k-1}}w_{i_{k-1}}\fur{\bC_{k-1}})\flr{(\ful{\bA_{k}}w_{i_{k}}\fur{\bC_{k}})})\\\cdot\fO(\ful{\ful{\bA_{1}}w_{i_{1}}\fur{\bC_{1}}\cdots\ful{\bA_{k-1}}w_{i_{k-1}}\fur{\bC_{k-1}}}\ful{\bA_{k}}w_{i_{k}}\fur{\bC_{k}}\fur{\ful{\bA_{k+1}}w_{i_{k+1}}\fur{\bC_{k+1}}\cdots \ful{\bA_{s}}w_{i_{s}}\fur{\bC_{s}}})\\\cdot\foz(\fll{(\ful{\bA_{k}}w_{i_{k}}\fur{\bC_{k}})}(\ful{\bA_{k+1}}w_{i_{k+1}}\fur{\bC_{k+1}}\cdots \ful{\bA_{s}}w_{i_{s}}\fur{\bC_{s}}))\prod_{j=1}^{s}\pari(\foz)(\bA_{j}\flr{w_{i_{j}}})\pari(\foz)(\fll{w_{i_{j}}}\bC_{j})\end{multlined}\\
        &=\begin{multlined}[t]\sum_{s=1}^{\infty}\sum_{\{(\bA_{j};w_{i_{j}};\bC_{j})\}_{j=1}^{s}\in\Sigma_{s}(\bw)}(-1)^{\ell(\bA_{1})+\ell(\bC_{1})+\cdots+\ell(\bA_{s})+\ell(\bC_{s})}\gamma_{s}^{f}\\\cdot\sum_{k=1}^{s}(s+1-k)\foz((\ful{\bA_{1}}w_{i_{1}}\fur{\bC_{1}}\cdots \ful{\bA_{k-1}}w_{i_{k-1}}\fur{\bC_{k-1}})\flr{w_{i_{k}}})\\\cdot\fO(\ful{\bA_{1}w_{i_{1}}\bC_{1}\cdots\bA_{k-1}w_{i_{k-1}}\bC_{k-1}\bA_{k}}w_{i_{k}}\fur{\bC_{k}\bA_{k+1}w_{i_{k+1}}\bC_{k+1}\cdots \bA_{s}w_{i_{s}}\bC_{s}})\\\cdot\foz(\fll{w_{i_{k}}}(\ful{\bA_{k+1}}w_{i_{k+1}}\fur{\bC_{k+1}}\cdots \ful{\bA_{s}}w_{i_{s}}\fur{\bC_{s}}))\prod_{j=1}^{s}\foz(\bA_{j}\flr{w_{i_{j}}})\foz(\fll{w_{i_{j}}}\bC_{j})\end{multlined}\\
        &=\begin{multlined}[t]\sum_{s=1}^{\infty}\sum_{\{(\bA_{j};w_{i_{j}};\bC_{j})\}_{j=1}^{s}\in\Sigma_{s}(\bw)}(-1)^{\ell(\bA_{1})+\ell(\bC_{1})+\cdots+\ell(\bA_{s})+\ell(\bC_{s})}\gamma_{s}^{f}\sum_{k=1}^{s}(s+1-k)\\\cdot\left(\prod_{j=1}^{k-1}\foz((\bA_{j}\flr{w_{i_{k}}})\flr{(w_{i_{j}})\flr{w_{i_{k}}}})\fO(\ful{(\bA_{j}\flr{w_{i_{k}}})}(w_{i_{j}}\flr{w_{i_{k}}})\fur{(\bC_{j}\flr{w_{i_{k}}})})\foz(\fll{(w_{i_{j}}\flr{w_{i_{k}}})}(\bC_{j}\flr{w_{i_{k}}}))\right)\\\cdot\foz(\bA_{k}\flr{w_{i_{k}}})\fO(\ful{\bA_{1}w_{i_{1}}\bC_{1}\cdots\bA_{k-1}w_{i_{k-1}}\bC_{k-1}\bA_{k}}w_{i_{k}}\fur{\bC_{k}\bA_{k+1}w_{i_{k+1}}\bC_{k+1}\cdots \bA_{s}w_{i_{s}}\bC_{s}})\foz(\fll{w_{i_{k}}}\bC_{k})\\\cdot\left(\prod_{j=k+1}^{s}\foz((\fll{w_{i_{k}}}\bA_{j})\flr{(\fll{w_{i_{k}}}w_{i_{j}})})\fO(\ful{(\fll{w_{i_{k}}}\bA_{j})}(\fll{w_{i_{k}}}w_{i_{j}})\fur{(\fll{w_{i_{k}}}\bC_{j})})\foz(\fll{(\fll{w_{i_{k}}}w_{i_{j}})}(\fll{w_{i_{k}}}\bC_{j}))\right)\end{multlined}\\
        &=\begin{multlined}[t]\sum_{s=1}^{\infty}\sum_{\substack{\bw=\bB_{1}\cdots\bB_{s}\\ \bB_{1},\ldots,\bB_{s}\neq\emp}}(-1)^{\ell(\bw)-s}\gamma_{s}^{f}\sum_{k=1}^{s}(s+1-k)\sum_{\bB_{k}=\bA_{k}w_{i_{k}}\bC_{k}}\left(\prod_{j=1}^{k-1}\left(\sum_{\bB_{j}=\bA_{j}w_{i_{j}}\bC_{j}}\foz((\bA_{j}\flr{w_{i_{k}}})\flr{(w_{i_{j}})\flr{w_{i_{k}}}})\right.\right.\\\phantom{\prod_{j=1}^{k-1}}\left.\left.\cdot\fO(\ful{(\bA_{j}\flr{w_{i_{k}}})}(w_{i_{j}}\flr{w_{i_{k}}})\fur{(\bC_{j}\flr{w_{i_{k}}})})\foz(\fll{(w_{i_{j}}\flr{w_{i_{k}}})}(\bC_{j}\flr{w_{i_{k}}}))\right)\right)\\\cdot\foz(\bA_{k}\flr{w_{i_{k}}})\fO(\ful{\bB_{1}\cdots\bB_{k-1}\bA_{k}}w_{i_{k}}\fur{\bC_{k}\bB_{k+1}\cdots\bB_{s}})\foz(\fll{w_{i_{k}}}\bC_{k})\\\cdot\begin{multlined}[t]\left(\prod_{j=k+1}^{s}\left(\sum_{\bB_{j}=\bA_{j}w_{i_{j}}\bC_{j}}\foz((\fll{w_{i_{k}}}\bA_{j})\flr{(\fll{w_{i_{k}}}w_{i_{j}})})\right.\right.\\\left.\left.\cdot\fO(\ful{(\fll{w_{i_{k}}}\bA_{j})}(\fll{w_{i_{k}}}w_{i_{j}})\fur{(\fll{w_{i_{k}}}\bC_{j})})\foz(\fll{(\fll{w_{i_{k}}}w_{i_{j}})}(\fll{w_{i_{k}}}\bC_{j}))\right)\right).\end{multlined}\end{multlined}
    \end{align}
    Using Proposition \ref{prop:ez_and_es} (3), it is simplified as
    \begin{align}
        &\gaxit(\pari(\foz),\pari(\foz))(\dTo_{R}(f))(\bw)\\
        &=\begin{multlined}[t]\sum_{s=1}^{\infty}\sum_{\substack{\bw=\bB_{1}\cdots\bB_{s}\\ \bB_{1},\ldots,\bB_{s}\neq\emp}}(-1)^{\ell(\bw)-s}\gamma_{s}^{f}\sum_{k=1}^{s}(s+1-k)\sum_{\bB_{k}=\bA_{k}w_{i_{k}}\bC_{k}}\left(\prod_{j=1}^{k-1}\foz(\bB_{j}\flr{w_{i_{k}}})\right)\\\cdot\left(\foz(\bA_{k}\flr{w_{i_{k}}})\fO(\ful{\bB_{1}\cdots\bB_{k-1}\bA_{k}}w_{i_{k}}\fur{\bC_{k}\bB_{k+1}\cdots\bB_{s}})\foz(\fll{w_{i_{k}}}\bC_{k})\right)\left(\prod_{j=k+1}^{s}\foz(\fll{w_{i_{k}}}\bB_{j})\right)\end{multlined}\\
        &=\begin{multlined}[t]\sum_{\bw=\bX w_{i}\bY}\sum_{s=1}^{\infty}(-1)^{\ell(\bw)-s}\gamma_{s}^{f}\sum_{k=1}^{s}(s+1-k)\sum_{\substack{\bX=\bB_{1}\cdots\bB_{k-1}\bA_{k}\\ \bB_{1},\ldots,\bB_{k}\neq\emp}}\sum_{\substack{\bY=\bC_{k}\bB_{k+1}\cdots\bB_{s}\\ \bB_{k+1},\ldots,\bB_{s}\neq\emp}}\foz((\bB_{1}\cdots\bB_{k-1}\bA_{k})\flr{w_{i}})\\\cdot\fO(\ful{\bB_{1}\cdots\bB_{k-1}\bA_{k}}w_{i}\fur{\bC_{k}\bB_{k+1}\cdots\bB_{s}})\foz(\fll{w_{i}}(\bC_{k}\bB_{k+1}\cdots\bB_{s}))\end{multlined}\\
        &=\begin{multlined}[t]\sum_{\bw=\bX w_{i}\bY}\sum_{s=1}^{\infty}(-1)^{\ell(\bw)-s}\gamma_{s}^{f}\sum_{k=1}^{s}(s+1-k)\\\cdot\left(\sum_{\substack{\bX=\bB_{1}\cdots\bB_{k-1}\bA_{k}\\ \bB_{1},\ldots,\bB_{k}\neq\emp}}\sum_{\substack{\bY=\bC_{k}\bB_{k+1}\cdots\bB_{s}\\ \bB_{k+1},\ldots,\bB_{s}\neq\emp}}1\right)\foz(\bX\flr{w_{i}})\fO(\ful{\bX}w_{i}\fur{\bY})\foz(\fll{w_{i}}\bY)\end{multlined}\\
        &=\sum_{\bw=\bX w_{i}\bY}\sum_{s=1}^{\infty}(-1)^{\ell(\bw)-s}\gamma_{s}^{f}\sum_{k=1}^{s}(s+1-k)\binom{\ell(\bX)}{k-1}\binom{\ell(\bY)}{s-k}\foz(\bX\flr{w_{i}})\fO(\ful{\bX}w_{i}\fur{\bY})\foz(\fll{w_{i}}\bY).
    \end{align}
    Therefore, combining with Lemma \ref{lem:ecalle_459}, we have
    \begin{multline}
        (\gaxit(\pari(\foz),\pari(\foz))-\mantar)(\dTo_{R}(f))(\bw)\\
        =(-1)^{\ell(\bw)}\sum_{\bw=\bX w_{i}\bY}\left(\sum_{s=1}^{\infty}(-1)^{s}\gamma_{s}^{f}\sum_{k=1}^{s}(s+1-k)\binom{\ell(\bX)}{k-1}\binom{\ell(\bY)}{s-k}+\gamma_{\ell(\bw)}^{f}(\ell(\bX)+1)\right)\\
        \cdot\foz(\bX\flr{w_{i}})\fO(\ful{\bX}w_{i}\fur{\bY})\foz(\fll{w_{i}}\bY).
    \end{multline}
    Then it is sufficient to show that, if $f=\re^{-1}$, the equality
    \begin{equation}\label{eq:generating_function}
        \sum_{s=1}^{\infty}(-1)^{s}\gamma_{s}^{f}\sum_{k=1}^{s}(s-k+1)\binom{m}{k-1}\binom{n}{s-k+1}=-\gamma_{m+n+1}^{f}(m+1)
    \end{equation}
    holds for any non-negative integers $m$ and $n$.
    Computing the generating function of the left-hand side, we see that
    \begin{align}
        &\sum_{m,n=0}^{\infty}\left(\sum_{s=1}^{\infty}(-1)^{s}\gamma_{s}^{f}\sum_{k=1}^{s}(s-k+1)\binom{m}{k-1}\binom{n}{s-k}\right)x^{m}y^{n}\\
        &=\sum_{m,n=0}^{\infty}\sum_{i,j=0}^{\infty}(-1)^{i+j+1}\gamma_{i+j+1}^{f}(j+1)\binom{m}{i}\binom{n}{j}x^{m}y^{n}\\
        &=\sum_{i,j=0}^{\infty}(-1)^{i+j+1}\gamma_{i+j+1}^{f}(j+1)x^{i}y^{j}(1-x)^{-i-1}(1-y)^{-j-1}\\
        &=-\frac{1}{(1-x)(1-y)}\sum_{N=0}^{\infty}\gamma_{N+1}^{f}\left(\sum_{j=0}^{N}(j+1)\left(\frac{x}{x-1}\right)^{N-j}\left(\frac{y}{y-1}\right)^{j}\right)\\
        &=\frac{1-y}{1-x}\frac{\partial}{\partial y}\sum_{N=0}^{\infty}\gamma_{N+1}^{f}\left(\sum_{j=0}^{N}\left(\frac{x}{x-1}\right)^{N-j}\left(\frac{y}{y-1}\right)^{j+1}\right)\\
        &=\frac{1-y}{1-x}\frac{\partial}{\partial y}\sum_{N=0}^{\infty}\gamma_{N+1}^{f}\frac{y}{y-1}\frac{\left(\frac{x}{x-1}\right)^{N+1}-\left(\frac{y}{y-1}\right)^{N+1}}{\frac{x}{x-1}-\frac{y}{y-1}}\\
        &=(y-1)\frac{\partial}{\partial y}\sum_{N=1}^{\infty}\gamma_{N}^{f}\frac{y}{y-x}\left(\left(\frac{x}{x-1}\right)^{N}-\left(\frac{y}{y-1}\right)^{N}\right)\\
        &=(y-1)\frac{\partial}{\partial y}\frac{y}{y-x}\left(\frac{x-1}{x}f_{\#}\left(\frac{x}{x-1}\right)-\frac{y-1}{y}f_{\#}\left(\frac{y}{y-1}\right)\right).
    \end{align}
    Applying the formula
    \[f_{\#}(t)=t-\frac{f(t)}{f'(t)}=t-\frac{\re^{-1}(t)}{(\re^{-1})'(t)}=t-(1-t)\log(1-t),\]
    we obtain
    \begin{align}
        &\sum_{m,n=0}^{\infty}\left(\sum_{s=1}^{\infty}(-1)^{s}\gamma_{s}^{f}\sum_{k=1}^{s}(s+1-k)\binom{m}{k-1}\binom{n}{s-k}\right)x^{m}y^{n}\\
        &=(y-1)\frac{\partial}{\partial y}\frac{y}{y-x}\left(\frac{1}{y}\log(1-y)-\frac{1}{x}\log(1-x)\right)\\
        &=\frac{y-1}{(y-x)^{2}}\left(\log(1-x)-\log(1-y)-\frac{y-x}{1-y}\right).
    \end{align}
    On the other hand, we have the following expression of the generating function for the right-hand side of \eqref{eq:generating_function}:
    \begin{align}
        \sum_{m,n=0}^{\infty}\left(-\gamma_{m+n+1}^{f}(m+1)\right)x^{m}y^{n}
        &=-\frac{\partial}{\partial x}\sum_{N=1}^{\infty}\gamma_{N}^{f}\left(\sum_{m=0}^{N-1}x^{m+1}y^{N-1-m}\right)\\
        &=-\frac{\partial}{\partial x}\sum_{N=1}^{\infty}\gamma_{N}^{f}x\frac{x^{N}-y^{N}}{x-y}\\
        &=-\frac{\partial}{\partial x}\frac{x}{x-y}\left(\frac{1}{x}f_{\#}(x)-\frac{1}{y}f_{\#}(y)\right)\\
        &=-\frac{\partial}{\partial x}\frac{x}{x-y}\left(-\frac{1-x}{x}\log(1-x)+\frac{1-y}{y}\log(1-y)\right)\\
        &=-\frac{1}{(x-y)^{2}}\left((1-y)\log(1-x)+(y-1)\log(1-y)+x-y\right)\\
        &=\frac{y-1}{(y-x)^{2}}\left(\log(1-x)-\log(1-y)-\frac{y-x}{1-y}\right).
    \end{align}
    Therefore we obtain \eqref{eq:generating_function} for the $f=\re^{-1}$ case.
\end{proof}

\begin{theorem}\label{thm:pal_bisymmetral}
    Both $\fess$ and $\dess$ are $\gantar$-invariant.
\end{theorem}
\begin{proof}
    The fact $\fess=\gantar(\fess)$ is a consequence of Proposition \ref{prop:se_symmetral}.
    So we show the symmetrality of $\doss$ using that of $\fess$.
    From Lemma \ref{lem:Se_schneps} and Proposition \ref{prop:se_gari}, we have
    \[(\der\circ\invgari)(\fSe_{R}(f))=\preari(\invgari(\fSe_{R}(f)), \fTe_{R}(f^{-1}))\]
    for any $f\in\GIFF_{x}(R)$.
    Hence, taking the swappee of both sides, we get
    \[(\der\circ\invgira)(\dSo_{R}(f))=(\irat(\dTo_{R}(f^{-1}))\circ\invgira)(\dSo_{R}(f))+\mmu(\invgira(\dSo_{R}(f)),\dTo_{R}(f^{-1})),\]
    by the obivous identity $\invgira=\swap\circ\invgari\circ\swap$.
    Then, by using Proposition \ref{prop:244}, it follows that
    \begin{align}
        &-\mmu(\invgira(\dSo_{R}(f)),\dTo_{R}(f^{-1}))\\
        &=((-\der+\irat(\dTo_{R}(f^{-1})))\circ\invgira)(\dSo_{R}(f))\\
        &=\begin{multlined}[t]\left(\ganit(\dO_{\ast}(f^{-1}))\circ(-\der+\arit(\ganit(\dO_{\ast}(f^{-1}))^{-1}(\dTo_{R}(f^{-1}))))\right.\\\left.\circ\ganit(\dO_{\ast}(f^{-1}))^{-1}\circ\invgira\right)(\dSo_{R}(f)).\end{multlined}
    \end{align}
    It is rewritten as
    \begin{align}\label{eq:ripal_dilator}
        \begin{split}
        &(\der\circ\ganit(\dO_{\ast}(f^{-1}))^{-1}\circ\invgira)(\dSo_{R}(f))\\
        &=\begin{multlined}[t](\arit(\ganit(\dO_{\ast}(f^{-1}))^{-1}(\dTo_{R}(f^{-1})))\circ\ganit(\dO_{\ast}(f^{-1}))^{-1}\circ\invgira)(\dSo_{R}(f))\\
        +\mmu((\ganit(\dO_{\ast}(f^{-1}))^{-1}\circ\invgira)(\dSo_{R}(f)),\ganit(\dO_{\ast}(f^{-1})^{-1}(\dTo_{R}(f^{-1}))))\end{multlined}\\
        &=\preari((\ganit(\dO_{\ast}(f^{-1}))^{-1}\circ\invgira)(\dSo_{R}(f)),\ganit(\dO_{\ast}(f^{-1}))^{-1}(\dTo_{R}(f^{-1}))).
    \end{split}
    \end{align}
    Now we put $f=\re$ so that $\dSo_{R}(f)=\doss$.
    In this case, the separation lemma (Theorem \ref{thm:separation}) yields $\dO_{\ast}(f^{-1})=\foz$, which is equal to $\crash(\doss)$ (Proposition \ref{prop:crash_pal}).
    Hence, from the fundamental identity \eqref{eq:fundamental}, we obtain
    \begin{align}
        (\ganit(\dO_{\ast}(f^{-1}))^{-1}\circ\invgira)(\dSo_{R}(f))
        &=(\ganit(\foz)\circ\invgira)(\doss)\\
        &=\ganit(\foz)^{-1}(\fragira(1,\doss))\\
        &=\fragari(1,\doss)\\
        &=\invgari(\doss).
    \end{align}
    Applying this to \eqref{eq:ripal_dilator}, we have the identity
    \[(\der\circ\invgari)(\doss)=\preari(\invgari(\doss),\ganit(\foz)^{-1}(\dTo_{R}(\re^{-1}))).\]
    Then we see that, from Lemma \ref{lem:darapal}, the second component $\ganit(\foz)^{-1}(\dTo_{R}(\re^{-1}))$ of the right-hand side is $\mantar$-invariant.
    Thus Lemma \ref{lem:dilator_mantar} yields the $\gantar$-invariance of $\invgari(\doss)$, which is equivalent to the desired equality $\gantar(\doss)=\doss$ as $\gantar$ preserves the $\gari$-product.
\end{proof}

\subsection{Images under $\slash$ and $\crash$}
\begin{lemma}\label{cor:girat_anti}
    The operator $\girat(\fess)$ and $\girat(\dess)$ commutes with $\anti$.
\end{lemma}
\begin{proof}
    Since $\doss$ is $\gantar$-invariant (Theorem \ref{thm:pal_bisymmetral}),
    we have
    \begin{align}
        (\push\circ\swap\circ\invmu\circ\swap)(\fess)
        &=(\push\circ\swap\circ\invmu)(\doss)\\
        &=(\push\circ\swap\circ\anti\circ\pari)(\doss)\\
        &=(\anti\circ\swap\circ\neg\circ\pari)(\doss)\\
        &=(\anti\circ\swap)(\doss)\\
        &=\anti(\fess),
    \end{align}
    where we used the identity \eqref{eq:negpush} and Proposition \ref{prop:neg_pari_ess} in the third and fourth equality, respectively.
    Hence we obtain
    \begin{align}
        \girat(\fess)
        &=\gaxit(\fess,(\push\circ\swap\circ\invmu\circ\swap)(\fess))\\
        &=\gaxit(\fess,\anti(\fess)).
    \end{align}
    Using the identity $\ganit(A)=\anti\circ\gamit(\anti(A))\circ\anti$, we have
    \begin{align}
        \girat(\fess)\circ\anti
        &=\giwat(\fess)\circ\anti\\
        &=\gamit(\fess)\circ\ganit((\gamit(\fess)^{-1}\circ\anti)(\fess))\circ\anti\\
        &=\anti\circ\ganit(\anti(\fess))\circ\anti\circ\ganit((\gamit(\fess)^{-1}\circ\anti)(\fess))\circ\anti\\
        &=\anti\circ\ganit(\anti(\fess))\circ\gamit((\anti\circ\gamit(\fess)^{-1}\circ\anti)(\fess))\\
        &=\anti\circ\ganit(\anti(\fess))\circ\gamit(\ganit(\anti(\fess))^{-1}(\fess))\\
        &=\anti\circ\gaxit(\fess,\anti(\fess))\\
        &=\anti\circ\giwat(\fess)\\
        &=\anti\circ\girat(\fess).
    \end{align}
    The assertion for $\girat(\dess)$ is shown in the same way.
\end{proof}

\begin{proposition}\label{prop:slash_pal}
    We have $\slash(\dess)=\fes$.
\end{proposition}
\begin{proof}
    From the fundamental identity (Theorem \ref{thm:fundamental}) and Proposition \ref{prop:crash_pal}, we have
    \begin{align}
        \slash(\dess)
        &=\fragari(\neg(\dess),\dess)\\
        &=\ganit(\crash(\dess))^{-1}(\fragira(\neg(\dess),\dess))\\
        &=\ganit(\fez)^{-1}(\fragira(\neg(\dess),\dess))\\
        &=\ganit(\fez)^{-1}\mmu((\girat(\dess)^{-1}\circ\neg)(\dess),\invgira(\dess)).
    \end{align}
    With the multiplicativity of $\girat(\dess)^{-1}=\girat(\invgira(\dess))$ (Proposition \ref{prop:gaxit_mu}) and the definition of $\invgira$, we obtain
    \begin{align}
        \slash(\dess)
        &=(\ganit(\fez)^{-1}\circ\girat(\invgira(\dess)))(\mmu(\neg(\dess),(\girat(\dess)\circ\invgira)(\dess)))\\
        &=(\ganit(\fez)^{-1}\circ\girat(\invgira(\dess)))(\mmu(\neg(\dess),\invmu(\dess))).
    \end{align}
    Remembering the identity $\ganit(\fez)(\fes)=\fez$, which is a consequence of Proposition \ref{prop:ez_and_es}, we see that it suffices to show\footnote{One can find a very similar formula in \cite[(11.89)]{ecalle11}, but the method of proof seems to be different: \'{E}calle used the reduction theorem of multitangent functions and a certain $\gari$-factorization for generating functions of multiple zeta values.}
    \begin{equation}\label{eq:ecalle_1189}
        \girat(\dess)(\fez)=\mmu(\neg(\dess),\invmu(\dess)).
    \end{equation}
    By using $\crash(\dess)=\fez$ (Proposition \ref{prop:crash_pal}) again and \eqref{eq:rash_dess}, the left-hand side becomes
    \begin{align}
        \girat(\dess)(\fez)
        &=(\girat(\dess)\circ\crash)(\dess)\\
        &=\girat(\dess)(\mmu(\anti\circ\swap\circ\invgari\circ\swap)(\dess),(\swap\circ\invgari\circ\swap)(\dess))\\
        &=\girat(\dess)(\mmu(\anti\circ\invgira)(\dess),\invgira(\dess)).
    \end{align}
    Since Corollary \ref{cor:girat_anti} and Proposition \ref{prop:gaxit_mu} holds, it is deformed as
    \begin{align}
        \girat(\dess)(\fez)
        &=\mmu((\anti\circ\girat(\dess)\circ\invgira)(\dess),(\girat(\dess)\circ\invgira(\dess))(\dess))\\
        &=\mmu((\anti\circ\invmu)(\dess),\invmu(\dess)).
    \end{align}
    Then the $\gantar$- and $\neg\circ\pari$-invariance of $\dess$ yields \eqref{eq:ecalle_1189}.
\end{proof}

\begin{theorem}\label{thm:crash_pil}
    We have $\crash(\fess)=\fez$.
\end{theorem}
\begin{proof}
    From Proposition \ref{prop:slash_pal}, it suffices to show
    \begin{equation}
        \gari((\swap\circ\crash)(\fess),\doss)=\neg(\doss).
    \end{equation}
    By the definition of $\gira$, we have
    \begin{align}
        &\swap(\gari((\swap\circ\crash)(\fess),\doss))\\
        &=\gira(\crash(\fess),\fess)\\
        &=\mmu((\girat(\fess)\circ\crash)(\fess),\fess)\\
        &=\begin{multlined}[t]\girat(\fess)\left(\mmu((\push\circ\swap\circ\invmu\circ\invgari\circ\swap)(\fess),\right.\\\left.(\swap\circ\invgari\circ\swap)(\fess),\girat(\fess)^{-1}(\fess))\right),\end{multlined}
    \end{align}
    where in the last equality we used Proposition \ref{prop:gaxit_mu}.
    About the latter factors, we know that
    \begin{align}
        &\girat(\fess)(\mmu((\swap\circ\invgari\circ\swap)(\fess),\girat(\fess)^{-1}(\fess)))\\
        &=\mmu((\girat(\fess)\circ\swap\circ\invgari\circ\swap)(\fess),\fess)\\
        &=\gira((\swap\circ\invgari\circ\swap)(\fess),\fess)\\
        &=\gari(\invgari(\doss),\doss)\\
        &=1.
    \end{align}
    Therefore, from Theorem \ref{thm:pal_bisymmetral}, \eqref{eq:negpush} and Proposition \ref{prop:neg_pari_ess} we obtain
    \begin{align}
        &\swap(\gari((\swap\circ\crash)(\fess),\doss))\\
        &=(\girat(\fess)\circ\push\circ\swap\circ\invmu\circ\invgari\circ\swap)(\fess)\\
        &=(\girat(\fess)\circ\push\circ\swap\circ\invmu\circ\invgari)(\doss)\\
        &=(\girat(\fess)\circ\push\circ\swap\circ\anti\circ\pari\circ\invgari)(\doss)\\
        &=(\girat(\fess)\circ\anti\circ\swap\circ\neg\circ\pari\circ\invgari)(\doss)\\
        &=(\girat(\fess)\circ\anti\circ\swap\circ\invgari)(\doss).
    \end{align}
    Moreover, using Corollary \ref{cor:girat_anti}, it holds that
    \begin{align}
        \swap(\gari((\swap\circ\crash)(\fess),\doss))
        &=(\anti\circ\girat(\fess)\circ\swap\circ\invgari)(\doss)\\
        &=\anti(\mmu((\girat(\fess)\circ\swap\circ\invgari)(\doss),\fess,\invmu(\fess)))\\
        &=\anti(\mmu(\gira((\swap\circ\invgari)(\doss),\fess),\invmu(\fess)))\\
        &=\anti(\mmu(\swap(\gari(\invgari(\doss),\doss)),\invmu(\fess)))\\
        &=(\anti\circ\invmu)(\fess)\\
        &=\pari(\fess)\\
        &=\neg(\fess),
    \end{align}
    where the $\gantar$-invariance (resp.~$\neg\circ\pari$-invariance) of $\fess$ is used in the sixth (resp.~seventh) equality. It yields the identity $\gari((\swap\circ\crash)(\fess),\doss)=\neg(\doss)$ and therefore $\crash(\fess)=(\swap\circ\slash)(\doss)=\fez$.
\end{proof}

\end{document}